\documentclass[10pt,a4paper]{article}
\RequirePackage{fix-cm}
 
\newcommand{\email}[1]{\hspace*{\stretch{1}}\emph{\texttt{#1}}}

\makeatletter
\def\blfootnote{\xdef\@thefnmark{$\star$}\@footnotetext}
\makeatother
\newenvironment{Authors}%
  {\begin{center}\begin{bfseries}}%
  {\end{bfseries}\end{center}}
\newenvironment{Addresses}%
  {\begin{flushleft}\begin{itshape}}%
  {\end{itshape}\end{flushleft}}

 \usepackage{fancyhdr}  
  
\fancypagestyle{plain}{
	\fancyhead{}
	\fancyhead[C]{\hfill submitted to ACOM, May 2018}
}
 
\usepackage{graphicx}

\usepackage[fleqn]{amsmath}
\usepackage{amstext}
\usepackage{array}
\usepackage{amsthm}
\usepackage{amsfonts}
\usepackage{amssymb}
\usepackage[font=small,labelfont=bf]{caption}
\usepackage{url}

\usepackage{algorithm}     
\usepackage{algpseudocode} 

\usepackage{marginnote}

\usepackage{csquotes}
\MakeOuterQuote{"}

\usepackage{tikz}
\usetikzlibrary{patterns}
\usepackage{bbm}
\usetikzlibrary{arrows}
\usepackage{mathtools}
 \usepackage{multirow}

\usepackage{blindtext}

\usepackage{marginnote}
\usepackage{subfig}

\newtheorem{theorem}{Theorem}[section]
\newtheorem{theorem1}{Theorem}[section]
\newtheorem{theorem2}{Theorem}[section]
\newtheorem{Proposition}[theorem]{Proposition}
\newtheorem{Lemma}[theorem1]{Lemma}
\newtheorem{remark}[theorem2]{Remark}

\begin{document}
\thispagestyle{plain}

\title{An offline/online  procedure for dual norm calculations of parameterized functionals: empirical quadrature and empirical test spaces}

\date{}

\maketitle

\vspace{-45pt}

\begin{Authors}
Tommaso Taddei$^{1}$
\end{Authors}

\begin{Addresses}
$^1$   Institut de Math{\'e}matiques de Bordeaux,
Team MEMPHIS, INRIA Bordeaux - Sud Ouest       \email{tommaso.taddei@inria.fr}\\ 
\end{Addresses}

\begin{abstract}
We present an offline/online computational procedure for computing the dual norm of parameterized
 linear functionals.
{ The approach is motivated by the need to efficiently compute residual dual norms, which are used in model reduction to estimate the  error of a given reduced solution.} 
  The key elements of the approach are
 (i) an empirical test space  
 for the manifold of Riesz elements associated with the parameterized functional, and
 (ii) an empirical quadrature procedure to efficiently deal with parametrically non-affine terms.
We present a number of theoretical and numerical results to  identify the different sources of error and to motivate the proposed technique,  and we compare the approach with other state-of-the-art techniques.
Finally, we investigate the effectiveness of our approach to reduce both offline and online costs associated with the computation of the time-averaged residual indicator proposed in   [Fick, Maday, Patera, Taddei, Journal of Computational Physics,
 2018].
\end{abstract}


\section{Introduction}
\label{sec:intro}
{ 
\emph{A posteriori} error estimators are designed to assess the accuracy  of a given numerical solution in a proper metric of interest.
In the context of  Model Order Reduction (MOR, \cite{quarteroni2015reduced,hesthaven2015certified}), \emph{a posteriori} error estimators are employed during the offline stage to drive the construction of the Reduced Order Model (ROM), and also during the online stage to certify the accuracy of the estimate. The vast majority of error estimators employed in MOR procedures relies on the evaluation of the dual norm of the residual: 
for nonlinear and non-parametrically affine problems, this task might be particularly demanding both in terms of memory requirements and of computational cost.
Motivated by these considerations, the objective of this paper is to develop and analyze an offline/online
computational strategy for the computation of the dual norm of parameterized functionals.
}

Given the parameter space $\mathcal{P} \subset \mathbb{R}^P$ and the domain $\Omega \subset \mathbb{R}^d$,  we introduce the Hilbert space $\mathcal{X}$ defined over $\Omega$ endowed with 
the inner product $(\cdot, \cdot)_{\mathcal{X}}$ and the induced norm
$\| \cdot \|_{\mathcal{X}} := \sqrt{(\cdot, \cdot)_{\mathcal{X}}}$.
We denote by $\mathcal{X}'$   the dual space of $\mathcal{X}$, and we define the Riesz operator $R_{\mathcal{X}}: \mathcal{X}' \to \mathcal{X}$ such that
$(R_{\mathcal{X}} \mathcal{L}, \, v)_{\mathcal{X}} = \mathcal{L}(v)$ for all $v \in \mathcal{X}$ and $\mathcal{L} \in \mathcal{X}'$. Exploiting these definitions,  our goal is to reduce 
the \emph{marginal} (i.e., in the limit of many queries) cost associated with the computation of the dual norm of the parameterized functional $\mathcal{L}_{\mu}$,
\begin{equation}
\label{eq:dual_norm_definition}
\|  \mathcal{L}_{\mu}  \|_{\mathcal{X}'}:=
\sup_{v \in \mathcal{X}} \, \frac{\mathcal{L}_{\mu}(v)}{\|  v\|_{\mathcal{X}}} =
\| R_{\mathcal{X}} \mathcal{L}_{\mu}  \|_{\mathcal{X}},
\end{equation}
for $\mu \in \mathcal{P}$.
We are interested in functionals of the form
\begin{equation}
\label{eq:functional_definition}
\mathcal{L}_{\mu}(v)
=
\int_{\Omega} \, \eta(x; v, \mu) \,  dx,
\quad
\eta(x; v, \mu) =
\Upsilon_{\mu}(x) \cdot F(x; v),
\end{equation}
where $\Upsilon: \Omega \times \mathcal{P} \to \mathbb{R}^D$ is a given function of spatial coordinate   and  parameter, and
$F$ is a linear function of $v$ and possibly its derivatives.
{ 
 Throughout the work, we shall consider
$H_0^1(\Omega) \subset \mathcal{X} \subset H^1(\Omega)$, and
$F(x; v) = [v(x), \nabla v(x)]$.  However, our discussion can be extended to other classes of functionals and other choices of the ambient space $\mathcal{X}$.}

{ 
If $ \mathcal{L}_{\mu}$ is \emph{parametrically-affine}, i.e.
$\mathcal{L}_{\mu}(v) = \sum_{m=1}^M \, \Theta_m(\mu) \mathcal{L}_m(v)$ with 
$M= \mathcal{O}(1)$,
then computations of  $\|  \mathcal{L}_{\mu} \|_{\mathcal{X}'}$ can be performed efficiently exploiting the linearity of the Riesz operator; on the other hand, if 
$ \mathcal{L}_{\mu}$ is not  parametrically-affine, hyper-reduction techniques should be employed.
}
Over the past decade, many authors have proposed hyper-reduction  procedures for the efficient 
evaluation of parameterized integrals: 
these techniques can be classified as 
\emph{Approximation-Then-Integration}  (ATI) approaches or 
\emph{Empirical Quadrature} (EQ)  approaches.
ATI approaches (i)  construct a suitable reduced basis and an associated interpolation/approximation system for $\Upsilon_{\mu}$ in \eqref{eq:functional_definition}, and then 
(ii) precompute all required integrals during an offline stage.
Conversely, EQ procedures 
--- {  also known as Reduced-Order Quadrature  procedures (\cite{antil2013two})} --- 
directly approximate the integrals in \eqref{eq:functional_definition}  by developing a specialized low-dimensional empirical quadrature rule.
Representative  ATI approaches for model reduction applications are Gappy-POD,
which was first proposed in \cite{everson1995karhunen} for image reconstruction and then 
adapted to MOR in 
\cite{bui2003proper,carlberg2011efficient,astrid2008missing},
and the Empirical Interpolation Method (EIM, \cite{barrault2004empirical,grepl2007efficient})
and related approaches (\cite{drohmann2012reduced,nguyen2008best,chaturantabut2010nonlinear,ryckelynck2009hyper,daversin2015simultaneous}).
On the other hand, EQ approaches have been proposed in 
\cite{an2008optimizing,antil2013two,farhat2015structure,patera2017lp}.

{ 
As explained in \cite{patera2017lp}, although ATI approaches are quite effective in practice, the objective of function approximation and integration are arguably quite different and it is thus difficult to relate the error in integrand approximation to the  error in  dual norm prediction.
As a result, rather conservative selections of the approximation tolerance are required in practice to ensure that the dual norm estimate is sufficiently accurate. On the other hand, since the test space in \eqref{eq:dual_norm_definition} is infinite-dimensional, EQ approaches cannot be applied as is, unless 
the Riesz element $\xi_{\mu} := R_{\mathcal{X}} \, \mathcal{L}_{\mu}$ is known explicitly.
}

We here propose an offline/online procedure that relies on two key ingredients: empirical test spaces, and empirical quadrature.
We resort to Proper Orthogonal Decomposition (POD, \cite{berkooz1993proper,sirovich1987turbulence,volkwein2011model})
to generate a 
$J_{\rm es}$-dimensional 
reduced space $\mathcal{X}_{J_{\rm es}}$ that approximates the manifold of Riesz elements associated with $\mathcal{L}$,
$\mathcal{M}_{\mathcal{L}}= \{ R_{\mathcal{X}} \mathcal{L}_{\mu}: \; \; \mu \in \mathcal{P} \}$. Then, we approximate  the dual norm
$\| \mathcal{L}_{\mu}  \|_{\mathcal{X}'}$ using
$\| \mathcal{L}_{\mu}  \|_{\mathcal{X}_{J_{\rm es}}'}$.
{  
Estimation of $\| \mathcal{L}_{\mu}  \|_{\mathcal{X}_{J_{\rm es}}'}$ involves the approximation of $J_{\rm es}$ integrals: we might thus rely on 
empirical quadrature procedures to estimate the latter dual norm.
}

The contributions of the present work are threefold:
(i) an actionable procedure for the construction of empirical test spaces;
(ii) the reinterpretation of the EQ problem as a sparse representation problem and the application of EQ to dual norm estimation; and
(iii) a thorough numerical and also theoretical  investigation of the performance of several ATI and EQ+ES hyper-reduction techniques for dual norm prediction.
\begin{itemize}
\item[(i)]
Empirical test spaces are closely related to $\ell_2$-embeddings, which have been recently proposed for MOR applications by Balabanov and Nouy in \cite{balabanov2018randomized}. In section \ref{sec:empirical_test_space}, 
we formally link empirical test spaces for dual norm calculations to $\ell_2$-embeddings, and we discuss the differences in their practical constructions.
Furthermore, in section \ref{sec:a_priori_error}, we present an \emph{a priori} error bound that motivates our construction.
\item[(ii)]
The problem of sparse representation --- or equivalently best subset selection --- 
 has been widely studied in the optimization, statistics and signal processing literature, and several solution strategies are available,
{  including $\ell^1$ relaxation \cite{donoho2006compressed},
Greedy algorithms \cite{tropp2004greed},
and more recently  mixed integer optimization procedures \cite{bertsimas2016best}.}
In this work, we show that the problem of EQ can be recast as a sparse representation problem, and we consider three different approaches based on   $\ell^1$ minimization (here referred to as $\ell^1$-EQ), on   the EIM greedy algorithm (EIM-EQ),  and
  on Mixed Integer Optimization (MIO-EQ), respectively.
  {We remark that    $\ell^1$-EQ   has been  first proposed in \cite{patera2017lp} for empirical quadrature,    while  EIM-EQ has been first proposed in    \cite{antil2013two}; on the other hand, 
 MIO-EQ  is new in this context. }
In order to reduce the cost associated with the construction of the quadrature rule, we further propose a divide-and-conquer strategy to reduce the dimension of the optimization problem to be solved offline.
\item[(iii)]
To our knowledge, a detailed comparison of ATI approaches and EQ approaches for hyper-reduction is currently missing in the literature. In sections \ref{sec:theoretical_comparison_ITI_vs_EQ_ES} and 
\ref{sec:numerics} we 
present numerical and also theoretical 
results that offer insights about the potential benefits and drawbacks of ATI and EQ techniques, in the context of dual norm prediction. 
\end{itemize}

The paper is organized as follows.
In section \ref{sec:methodology}, we present the computational procedure, and we prove an \emph{a priori} error bound that relates the prediction error in dual norm estimation to the quadrature error and to the discretization error associated with the introduction of the empirical test space.
In section \ref{sec:theoretical_comparison_ITI_vs_EQ_ES}, 
we review ATI approaches for dual norm calculations and we offer several remarks concerning the benefits and the drawbacks of the proposed strategies. Furthermore, in section 
\ref{sec:numerics}, we present numerical results to compare the performance of our EQ+ES method for the three  EQ procedures considered with a representative ATI approach; in particular, we apply the proposed technique to the computation of the dual time-averaged residual presented in 
\cite{fick2017reduced},
associated with the  Reduced Basis approximation of the solution to   a 2D unsteady incompressible Navier-Stokes problem.
Finally,  in section \ref{sec:conclusion}, we  summarize the key contributions and identify several future research directions.



\section{Methodology}
\label{sec:methodology}
 \subsection{Formulation}
\label{sec:formulation}

In view of the presentation of the methodology, we introduce the high-fidelity (truth) space  $\mathcal{X}_{\rm hf} = {\rm span} \{ \varphi_i\}_{i=1}^{\mathcal{N}} \subset \mathcal{X}$ and the high-fidelity quadrature rule
$$
\mathcal{Q}^{\rm hf}(v):= \sum_{i=1}^{\mathcal{N}_{\rm q}} \, \rho_i^{\rm hf} \, v(x_i^{\rm hf}).
$$
We endow $\mathcal{X}_{\rm hf}$ with the inner product $(w,v)_{\mathcal{X}_{\rm hf}} =
\mathcal{Q}^{\rm hf}( \lambda(\cdot; w,v) )$ for a suitable choice\footnote{
In all our examples, we consider 
$\lambda(\cdot; w, v) =  \nabla w \cdot \nabla v + u \, v$.
} of $\lambda$, and we approximate the functional in \eqref{eq:functional_definition} as
$\mathcal{L}_{\mu, \rm hf}(v) := \mathcal{Q}^{\rm hf}( \eta(\cdot; v, \mu)  )$.
In the remainder, we shall assume that $\| \mathcal{L}_{\mu, \rm hf}  \|_{\mathcal{X}_{\rm hf}'} \approx  \| \mathcal{L}_{\mu}  \|_{\mathcal{X}'} $ for all $\mu \in \mathcal{P}$.
Since our approach builds upon the high-fidelity discretization, in the following we exclusively deal with high-fidelity quantities: to simplify notation, 
we omit the subscript $(\cdot)_{\rm hf}$.

Given $v \in \mathcal{X}$, we denote by $\mathbf{v} \in \mathbb{R}^{\mathcal{N}}$ the corresponding vector of coefficients,
$v(x) = \sum_{i=1}^{\mathcal{N}} \, {v}_i \, \varphi_i(x)$;  similarly, given $\mathcal{L} \in \mathcal{X}'$, we define  $\boldsymbol{\mathcal{L}} \in \mathbb{R}^{\mathcal{N}}$ such that 
$\left(\boldsymbol{\mathcal{L}}   \right)_i = \mathcal{L}(\varphi_i)$.
By straightforward calculations, we find the following expression for the dual norm:
\begin{equation}
\label{eq:dual_norm_algebraic}
L(\mu) :=
\|  \mathcal{L}_{\mu} \|_{\mathcal{X}'}
=
\sup_{v \in \mathcal{X}} \, 
\frac{     \mathcal{Q}^{\rm hf}( \eta(\cdot; v, \mu)  )      }{\| v \|_{\mathcal{X}}} \, = \,
\sqrt{ 
\boldsymbol{\mathcal{L}}_{\mu}^T \, \mathbb{X}^{-1} 
\boldsymbol{\mathcal{L}}_{\mu}
},
\end{equation}
where $\mathbb{X}_{i,j} = (\varphi_j, \varphi_i)_{\mathcal{X}}$. 
Evaluation of $L$ in \eqref{eq:dual_norm_algebraic} for a given $\mu \in \mathcal{P}$ requires the solution to a linear system of size $\mathcal{N}$, which costs $C_{\rm riesz} = \mathcal{O}(\mathcal{N}^s)$ for some $s \in [1,2)$.

To reduce the costs, we propose to substitute $\mathcal{X}$ in \eqref{eq:dual_norm_algebraic} with the $J_{\rm es}$-dimensional empirical test space 
$\mathcal{X}_{J_{\rm es}} = {\rm span} \{ \phi_j  \}_{j=1}^{J_{\rm es}}$ where $(\phi_j, \phi_i)_{\mathcal{X}} = \delta_{i,j}$,
and the high-fidelity quadrature rule $ \mathcal{Q}^{\rm hf}$ with the $Q_{\rm eq}$-dimensional quadrature rule
$$
\mathcal{Q}^{\rm eq} (v) \, = \,
\sum_{q=1}^{Q_{\rm eq}} \, \rho_q^{\rm eq} \, v(x_q^{\rm eq}),
\quad
{\rm for \, some} \;
\{
\rho_q^{\rm eq}, x_q^{\rm eq}
\}_{q=1}^{Q_{\rm eq}} \subset \mathbb{R} \times \Omega.
$$
Exploiting the fact that $\{ \phi_j \}_j$ is an orthonormal basis of $\mathcal{X}_{J_{\rm es}}$, we obtain the EQ+ES estimate of $L(\mu)$:
\begin{subequations}
\label{eq:EQ_ES_estimate}
\begin{equation}
\begin{array}{rl}
\displaystyle{ L_{J_{\rm es}, Q_{\rm eq}}(\mu) \, = }
&
\displaystyle{
\sup_{v \in \mathcal{X}_{J_{\rm es}}} \,
\frac{\mathcal{Q}^{\rm eq}(  \eta(\cdot; v, \mu )   ) }{\|   v \|_{\mathcal{X}}} \,
= \,
\sqrt{
\sum_{j=1}^{J_{\rm es}} \,
\left(
\mathcal{Q}^{\rm eq}(  \eta(\cdot; \phi_j, \mu ) 
\right)^2
}
}
\\[3mm]
=
&
\displaystyle{
\|  \mathbb{H}(\mu) \boldsymbol{\rho}^{\rm eq} \|_2,
}
\end{array}
\end{equation}
where 
\begin{equation}
\left( \mathbb{H}(\mu)  \right)_{q,j} \, = \,
 \eta(x_q^{\rm eq}; \phi_j , \mu) \, = \,
 \sum_{i=1}^D \,
 \left( \Upsilon_{\mu}(x_q^{\rm eq})  \right)_i \,
  \left( F(x_q^{\rm eq}; v)  \right)_i .
\end{equation}
\end{subequations}
It is straightforward to verify that,  if  $\eta$ is of the form  \eqref{eq:functional_definition},
computation of the matrix $\mathbb{H}(\mu)$ scales with $\mathcal{O}(  D \, J_{\rm es} \,  Q_{\rm eq}  )$: provided that
$J_{\rm es}, \,  Q_{\rm eq}  \ll \mathcal{N}$, evaluation of \eqref{eq:EQ_ES_estimate} is thus significantly less expensive than the evaluation of \eqref{eq:dual_norm_algebraic}.

Below, we discuss how to practically build the space $\mathcal{X}_{J_{\rm es}}$ (section \ref{sec:empirical_test_space}), and  
the quadrature rule $\mathcal{Q}^{\rm eq}(\cdot)$
(section \ref{sec:empirical_quadrature}).  Then, in section \ref{sec:full_procedure}, we briefly summarize the overall procedure and we comment on offline and online costs.
Finally, in section \ref{sec:a_priori_error}, we present an \emph{a priori} error bound
that shows that the prediction
error is the sum of two contributions: a quadrature error, and a discretization
error associated with the empirical test space.

\begin{remark}
\textbf{EQ estimate of $\mathbf{L}(\boldsymbol{\mu})$.}
The EQ estimate of $L(\mu)$, 
$$
L_{Q_{\rm eq}}(\mu) = \sup_{v \in \mathcal{X}} \, \frac{\mathcal{Q}^{\rm eq}(\eta(\cdot; v, \mu))}{\|  v \|_{\mathcal{X}}},
$$
is not in general related to $L(\mu)$ for $Q_{\rm eq} < \mathcal{N}_{\rm q}$.
For this reason, EQ approaches cannot be applied as is to estimate $L(\mu)$.
\end{remark}

\subsection{Empirical test space}
\label{sec:empirical_test_space}

Recalling the Riesz representation theorem, we find  that
 $L({\mu})^2  =  \mathcal{L}_{\mu} \left( \xi_{\mu} \right),$ for all $\mu \in \mathcal{P}$,
$\xi_{\mu} = R_{\mathcal{X}} \, \mathcal{L}_{\mu}$; 
 therefore, if 
$\mathcal{X}_{J_{\rm es}}$  accurately approximates the elements of the manifold
$\mathcal{M}_{\mathcal{L}}
\, = \,
 \left\{ 
\xi_{\mu}:
 \; 
 \mu \in \mathcal{P}
\right\}$, we expect that
\begin{equation}
\label{eq:caL_J}
L(\mu)  = \sup_{v \in \mathcal{M}_{\mathcal{L}}} \, 
\frac{\mathcal{L}_{\mu}(v)}{\| v \|_{\mathcal{X}}} \,
\approx \, \sup_{v \in \mathcal{X}_{J_{\rm es}} } \,
\frac{\mathcal{L}_{\mu}(v)}{\| v \|_{\mathcal{X}}}
=: \|  \mathcal{L}_{\mu}    \|_{\mathcal{X}_{J_{\rm es}}'}.
\end{equation}
We provide a rigorous justification of this approximation in section \ref{sec:a_priori_error}.

We construct the approximation space $\mathcal{X}_{J_{\rm es}}$ using POD.
First, we generate the training set $\Xi^{\rm train, es} = \{  \mu^{\ell} \}_{\ell=1}^{     n_{\rm train}^{\rm es}} \subset \mathcal{P}$,
where   $\mu^1,\ldots,\mu^{n_{\rm train}^{\rm es}} 
\overset{\rm iid}{\sim} {\rm Uniform}(\mathcal{P})$;
then we compute $\xi^{\ell} =  \xi_{\mu^{\ell}}$ for $\ell=1,\ldots,n_{\rm train}^{\rm es}$; finally, we use the snapshots $\{  \xi^{\ell} \}_{\ell=1}^{n_{\rm train}^{\rm es}}$ to compute the POD space $\mathcal{X}_{J_{\rm es}}$ (see \cite{sirovich1987turbulence}) based on the $\mathcal{X}$-inner product.

\begin{remark}
\textbf{Choice of $n_{\rm train}^{\rm es}$ and ${J_{\rm es}}$.}
In order to validate the choice of $n_{\rm train}^{\rm es}$ and ${J_{\rm es}}$, 
we might introduce $n_{\rm test}^{\rm es}$ additional samples
$\tilde{\mu}^1,\ldots,\tilde{\mu}^{n_{\rm test}^{\rm es}} \sim {\rm Uniform}(\mathcal{P})$, and 
compute the  error indicators\footnote{We observe that 
$E_{{J_{\rm es}},n_{\rm train}^{\rm es}, n_{\rm test}^{\rm es}}^{(\infty)}$ is equivalent to the error indicator proposed in 
\cite{buhr2017randomized}.
}
\begin{equation}
\label{eq:error_indicator_max}
E_{{J_{\rm es}},n_{\rm train}^{\rm es}, n_{\rm test}^{\rm es}}^{(\infty)}
=
\max_{k=1,\ldots,n_{\rm test}^{\rm es}}
\,
\| 
\Pi_{\mathcal{X}_{J_{\rm es}}^{\perp}}
\,
\xi_{\tilde{\mu}^k}
  \|_{\mathcal{X}},
\end{equation}
and
\begin{equation}
\label{eq:error_indicator_L2}
E_{{J_{\rm es}},n_{\rm train}^{\rm es}, n_{\rm test}^{\rm es}}^{(2)}
=
\frac{1}{n_{\rm test}^{\rm es}} \,
\sum_{k=1}^{n_{\rm test}^{\rm es}} \,
\| 
\Pi_{\mathcal{X}_{J_{\rm es}}^{\perp}}
\,
\xi_{\tilde{\mu}^k}
  \|_{\mathcal{X}}^2,
\end{equation}
where $ \mathcal{X}_{J_{\rm es}}^{\perp} $ denotes the orthogonal complement of $ \mathcal{X}_{J_{\rm es}}$, and
$\Pi_{\mathcal{X}_{J_{\rm es}}^{\perp}}: \mathcal{X} \to
{\mathcal{X}_{J_{\rm es}}^{\perp}}
$ is the orthogonal projection operator.
$E_{{J_{\rm es}},n_{\rm train}^{\rm es}, n_{\rm test}^{\rm es}}^{(\infty)}$ provides an estimate of the discretization error that enters in the \emph{a priori} error bound in Proposition \ref{th:a_priori}, 
while $E_{{J_{\rm es}},n_{\rm train}^{\rm es}, n_{\rm test}^{\rm es}}^{(2)}$ can be compared with the in-sample error 
$ E_{{J_{\rm es}},n_{\rm train}^{\rm es}}^{(2)}
=
\frac{1}{n_{\rm train}^{\rm es}} \,
\sum_{\ell=1}^{n_{\rm train}^{\rm es}} \,
\| 
\Pi_{\mathcal{X}_{J_{\rm es}}^{\perp}} \, \xi_{\mu^{\ell}}
  \|_{\mathcal{X}}^2$ to assess the representativity of the training set $\Xi^{\rm train, es}$.
\end{remark}

\subsubsection*{Connection with $\ell_2$-embeddings}

Exploiting  notation introduced in section \ref{sec:formulation}, and recalling that
$\{ \phi_j \}_j$ is an orthonormal basis of $\mathcal{X}_{J_{\rm es}}$,
 we can rewrite \eqref{eq:caL_J} as follows:
\begin{equation}
\label{eq:calLJ_algebraic}
\|  \mathcal{L}_{\mu}    \|_{\mathcal{X}_{J_{\rm es}}'}
 = \| \mathbb{X}_{J_{\rm es}}^T \boldsymbol{\mathcal{L}}_{\mu}  \|_2,
\end{equation}
where $\mathbb{X}_{J_{\rm es}} = [ \boldsymbol{\phi}_1,\ldots,  \boldsymbol{\phi}_{J_{\rm es}}]$.
In \cite{balabanov2018randomized} (see \cite[section 3.1]{balabanov2018randomized}), the authors propose to estimate $L(\mu) $ as 
\begin{equation}
\label{eq:calLJ_algebraic_2}
L_{\Theta}(\mu) = \| \Theta    \mathbb{X}^{-1} \boldsymbol{\mathcal{L}}_{\mu}  \|_2,
\end{equation}
where $\Theta  \in \mathbb{R}^{J_{\rm es} \times \mathcal{N}}$ is called $\mathcal{X} \to \ell_2$ embedding.
By comparing \eqref{eq:calLJ_algebraic} with \eqref{eq:calLJ_algebraic_2}, we deduce that the approach proposed here corresponds to that in 
\cite{balabanov2018randomized}, provided that $\Theta =  \mathbb{X}_{J_{\rm es}}^T \mathbb{X} $.

The key difference between the two approaches is in the practical construction of 
$\Theta$. In \cite{balabanov2018randomized}, the authors consider $\Theta = \boldsymbol{\Omega} \mathbb{Q}$ where 
$\mathbb{Q} \in \mathbb{R}^{\mathcal{N} \times \mathcal{N}}$ is the upper-triangular matrix associated with the  Cholesky factorization  of $\mathbb{X}$, while $\boldsymbol{\Omega} \in \mathbb{R}^{J_{\rm es} \times \mathcal{N}}$ is the realization of a random matrix
--- distributed according to  the rescaled Gaussian distribution, the rescaled Rademacher distribution, or the partial subsampled randomized Hadamard transform (P-SRHT).
For these three choices of the sampling distribution, the authors prove \emph{a priori} error bounds in probability, which provide estimates for the minimum value  of $J_{\rm es}$ required to achieve a target accuracy.

{ 
Recalling the optimality of POD (see, e.g., \cite{volkwein2011model}),
for sufficiently large values of $n_{\rm train}^{\rm es}$, we expect that our approach leads to smaller test spaces --- and thus more efficient online calculations for any target accuracy.
On the other hand, the construction of  $\Theta$  in \cite{balabanov2018randomized} requires significantly less offline resources than the construction of $\mathcal{X}_{J_{\rm es}}$.}
For this reason,   the choice between the two approaches is extremely problem- and architecture-dependent.

\subsection{Empirical quadrature}
 \label{sec:empirical_quadrature} 

We shall now address the problem of determining the quadrature rule
$$
\mathcal{Q}^{\rm eq} (v) \, = \,
\sum_{q=1}^{Q_{\rm eq}} \, \rho_q^{\rm eq} \, v(x_q^{\rm eq})
$$
in \eqref{eq:EQ_ES_estimate}. Towards this end, we assume that 
$\{  x_q^{\rm eq} \}_q \subset \{  x_i^{\rm hf} \}_i$; then, we define the quadrature rule operator
$\mathcal{Q}: C(\Omega) \times \mathbb{R}^{\mathcal{N}_{\rm q}} \to \mathbb{R}$ such that
$$
\mathcal{Q} (v, \boldsymbol{\rho} ) \, = \,
\sum_{i=1}^{  \mathcal{N}_{\rm q}} \, \rho_i \, v(x_i^{\rm hf}).
$$
Exploiting the latter definition, we formulate the problem of finding $\{ \rho_q^{\rm eq},  x_q^{\rm eq} \}_q$ as the problem of finding $\boldsymbol{\rho}^{\star} \in \mathbb{R}^{\mathcal{N}_{\rm q}} $ such that
\begin{enumerate}
\item
the number of nonzero entries in $\boldsymbol{\rho}^{\star}$ is as small as possible;
\item
the corresponding quadrature rule 
$\mathcal{Q}^{\star}(\cdot) := \mathcal{Q}( \cdot, \boldsymbol{\rho}^{\star} )$
satisfies
\begin{equation}
\label{eq:accuracy_constraints}
\big|  
\mathcal{Q}^{\star}(  \eta(\cdot; \phi_j, \mu)   )  
\, - \, 
\mathcal{Q}^{\rm hf}( \eta(\cdot; \phi_j, \mu) ) 
 \big| 
 \leq \delta,
 \qquad
j=1,\ldots,J_{\rm es},
\end{equation}
for all $\mu$ in the training set $\Xi^{\rm train,eq} = \{ \mu^{\ell}  \}_{\ell=1}^{n_{\rm train}^{\rm eq}}$, and 
\begin{equation}
\label{eq:identity_constraints}
\big|  
\mathcal{Q}^{\star}(  1  )  
\, - \, 
\mathcal{Q}^{\rm hf}( 1 ) 
 \big| 
 \leq \delta.
\end{equation}
\end{enumerate}
Given a (approximate) solution $\boldsymbol{\rho}^{\star}$, we then extract the strictly non-null quadrature weights 
$\{  \rho_q^{\rm eq}, x_q^{\rm eq} \}_{q=1}^{Q_{\rm eq}} =
\{
\{  \rho_i^{\star}, x_i^{\rm hf} \}_{i}: \,
i \in \{  
k: \rho_k^{\star} \neq 0
\}
\}
$.

{ 
The first requirement corresponds to minimizing the number of non-null weights $Q_{\rm eq}$:
recalling \eqref{eq:EQ_ES_estimate}, minimizing $Q_{\rm eq}$ is equivalent to minimizing the online costs for a given   choice of the empirical test space. Condition \eqref{eq:accuracy_constraints} controls the accuracy of the dual norm estimate, as discussed in the error analysis.
On the other hand, as explained in \cite{yanoalp}, condition \eqref{eq:identity_constraints} is empirically found to improve the accuracy of the EQ procedure when the integral is close to zero due to the cancellation of the integrand in different parts of the domain. Finally, we remark that in 
\cite{patera2017lp,yanoalp} the authors propose to add the non-negativity constraint 
\begin{equation}
\label{eq:positivity_constraint}
\rho_i^{\star} \geq 0 ,\quad
i=1,\ldots, \mathcal{N}_{\rm q}.
\end{equation}
As discussed later in this section, the non-negativity  constraint reduces by half the size of the problem  that is  practically solved during the offline stage for two of the EQ methods ($\ell^1$-EQ and MIO-EQ) employed in this work; furthermore, we observe that the non-negativity   of the weights is used in \cite{yanodiscontinuous} to prove a stability result for a  Galerkin ROM.} We here consider both the case of non-negative weights and the case of real-valued weights. We anticipate that for the latter case we are able to prove a theoretical result that motivates the approach.

These desiderata can be translated in the following minimization statement:
\begin{subequations}
\label{eq:benchmark_optimization_no_positivity}
\begin{equation}
\min_{\boldsymbol{\rho} \in \mathbb{R}^{\mathcal{N}_{\rm q}}   }
\,
\|  \boldsymbol{\rho}  \|_0,
\quad
{\rm s.t.} \; 
\|  {\mathbb{G}}  \boldsymbol{\rho} -  {\mathbf{y}}^{\rm hf}   \|_{\infty} \leq \delta    
\end{equation}
where $\| \cdot    \|_{\infty} $ denotes the $\infty$ norm,
$\|  \mathbf{v}  \|_{\infty} = \max_k |v_k |$,
$\|  \cdot \|_0$ denotes the $\ell^0$ "norm"\footnote{
$\|  \cdot \|_0$ is not a norm since it does not satisfy the homogeneity property; nevertheless, it is   called norm in the vast majority of the statistics and optimization  literature.
}
$\|  \boldsymbol{\rho} \|_0 = \# \{ \rho_i \neq 0: \, i=1,\ldots, \mathcal{N}_{\rm q} \}$, and 
$\mathbb{G} \in \mathbb{R}^{K \times  \mathcal{N}_{\rm q}}$ and 
$\mathbf{y}^{\rm hf} \in \mathbb{R}^{K}$, $K=n_{\rm train}^{\rm eq} J_{\rm es} + 1$,
are defined as
\begin{equation}
\mathbb{G} =
\left[
\begin{array}{ccc}
\eta(x_1^{\rm hf}; \phi_1, \mu^1 ), & \ldots &  \eta(x_{\mathcal{N}_{\rm q}}^{\rm hf}; \phi_1, \mu^1 ) \\ 
& \vdots & \\
\eta(x_1^{\rm hf}; \phi_{J_{\rm es}}, \mu^{n_{\rm train}^{\rm eq}} ), & \ldots &  \eta(x_{\mathcal{N}_{\rm q}}^{\rm hf}; \phi_{J_{\rm es}}, \mu^{n_{\rm train}^{\rm eq}} ) \\[2mm]
1 & \ldots & 1 \\
\end{array}
\right],
\end{equation}
\begin{equation}
\mathbf{y}^{\rm hf} = \left[
\mathcal{Q}^{\rm hf} \left( \eta(\cdot; \phi_1, \mu^1 )  \right),  \ldots  \, , \,
\mathcal{Q}^{\rm hf} \left( \eta(\cdot; \phi_{J_{\rm es}}, \mu^{n_{\rm train}^{\rm eq}} )  \right) , 
\mathcal{Q}^{\rm hf} \left( 1   \right)    \right]
\end{equation}
\end{subequations}
Alternatively, if we choose to include the non-negativity  constraint, we obtain
\begin{equation}
\label{eq:benchmark_optimization}
\min_{\boldsymbol{\rho} \in \mathbb{R}^{\mathcal{N}_{\rm q}}   }
\,
\|  \boldsymbol{\rho}  \|_0,
\quad
{\rm s.t.} \;  \left\{
\begin{array}{l}
\displaystyle{
\|  {\mathbb{G}}  \boldsymbol{\rho} -  {\mathbf{y}}^{\rm hf}   \|_{\infty} \leq \delta    
} \\[2mm]
\boldsymbol{\rho} \geq \mathbf{0} \\
\end{array}
\right.
\end{equation}

Problems   \eqref{eq:benchmark_optimization_no_positivity}  and  \eqref{eq:benchmark_optimization} 
 can be interpreted as sparse representation problems where the input data
 --- the high-fidelity integrals $\mathbf{y}^{\rm hf}$ ---
 are noise-free. We emphasize that there are important differences between the two problems considered here and the sparse representation problems typically considered in the statistics literature, particularly in compressed sensing (CS, \cite{donoho2006compressed}).
 CS relies on the assumption that the original signal is sparse, and that the coherence among different columns of $\mathbb{G}$ is small (see, e.g., \cite{bruckstein2009sparse} for a thorough discussion). In our setting, these conditions are not expected to hold due to the smoothness in space  of the elements of the manifold and to the deterministic nature of the problem. As a result, techniques developed and analyzed in the CS literature might be highly suboptimal in our context.
 After the seminal work by Bertsimas et al. \cite{bertsimas2016best}, Hastie et al. \cite{hastie2017extended} presented  detailed empirical comparisons for several state-of-the-art approaches  for  datasets characterized by a wide spectrum of Signal-to-Noise Ratios.

As stated in the introduction, we here resort to three EQ approaches to approximate 
\eqref{eq:benchmark_optimization_no_positivity} and \eqref{eq:benchmark_optimization}. While $\ell^1$-EQ and EIM-EQ have been first presented in \cite{patera2017lp} and \cite{antil2013two}, 
MIO-EQ is new in this context. In the next three sections, we briefly illustrate  the three EQ techniques.

\begin{remark}
\textbf{Dependence on the basis $\{ \phi_j \}_j$.}
Conditions \eqref{eq:accuracy_constraints} depend on the choice of the basis of $\mathcal{X}_{J_{\rm es}}$. In particular, given
$\tilde{\phi} = \sum_{j=1}^{J_{\rm es}} a_j \phi_j$, if we define 
${\eta}^{j, \ell} =   \eta( \cdot; {\phi}_j, \mu^{\ell} )
$, we obtain
$$
\begin{array}{l}
\displaystyle{
\big|  
\mathcal{Q} \left( \eta(\cdot; \tilde{\phi},  \mu^{\ell} )
 ; \boldsymbol{\rho}^{\star}  \right) 
\, - \, \mathcal{Q}^{\rm hf} \left(
\eta(\cdot; \tilde{\phi},  \mu^{\ell} )
 \right) 
 \big| 
=
\big|
\sum_{j=1}^{J_{\rm es}} \,
\left(
\mathcal{Q}\left( {\eta}^{j, \ell} ;
\boldsymbol{\rho}^{\star} \right)
\, - \,
\mathcal{Q}^{\rm hf}\left(  {\eta}^{j, \ell} \right)
\right)
  \, a_j
\big|
}
\\[3mm]
\displaystyle{
\leq
\delta \|  \mathbf{a} \|_1
\leq
\delta  \| \mathbf{1} \|_2
\|  \mathbf{a} \|_2
=
\delta  \sqrt{J_{\rm es}}  \|  \tilde{\phi} \|_{\mathcal{X}}
}
\\
\end{array}
$$
{
Note that in the first inequality we used \eqref{eq:accuracy_constraints}, while in the second inequality we used Cauchy-Schwarz inequality.
}
\end{remark}

\subsubsection{$\ell^1$ relaxation ($\ell^1$-EQ)}
\label{sec:ell1_EQ}
Following \cite{patera2017lp},
we consider the convex relaxation of \eqref{eq:benchmark_optimization}:
$$
\min_{\boldsymbol{\rho} \in \mathbb{R}^{\mathcal{N}_{\rm q}}   }
\,
\|  \boldsymbol{\rho}  \|_1,
\quad
{\rm s.t.} \; 
\left\{
\begin{array}{l}
\|  {\mathbb{G}}  \boldsymbol{\rho} -  {\mathbf{y}}^{\rm hf}   \|_{\infty} \leq \delta \\
  \boldsymbol{\rho}   \geq 0    \\
\end{array}
\right.
$$
which can be restated  as a linear programming problem: 
\begin{subequations}
\label{eq:ell1_minimization}
\begin{equation}
\min_{\boldsymbol{\rho} \in \mathbb{R}^{\mathcal{N}_{\rm q}}   }
\,
\mathbf{1}^T  \boldsymbol{\rho},
\quad
{\rm s.t.} \; 
\left\{
\begin{array}{l}
\mathbb{A}  \,  \boldsymbol{\rho} \leq \mathbf{b} \\
 \boldsymbol{\rho} \geq  \mathbf{0} \\
\end{array}
\right.
\end{equation}
 where
\begin{equation}
\mathbb{A}  =  
\left[
\begin{array}{c}
{\mathbb{G}}  \\
 - {\mathbb{G}} \\
\end{array}
\right],
\quad
\mathbf{b} =
\left[
\begin{array}{c}
 {\mathbf{y}}^{\rm hf} + \delta
\\
- {\mathbf{y}}^{\rm hf}  + \delta
\\
\end{array}
\right] 
\end{equation}
\end{subequations}

Proceeding in a similar way, we obtain the  
$\ell^1$-convexification of 
 \eqref{eq:benchmark_optimization_no_positivity}:
\begin{equation}
\label{eq:ell1_minimization_aux}
\min_{\boldsymbol{\rho} \in \mathbb{R}^{\mathcal{N}_{\rm q}}   }
\,
\|  \boldsymbol{\rho}  \|_1,
\quad
{\rm s.t.} \; 
{\mathbb{A}}  \boldsymbol{\rho} \leq  \mathbf{b}.
\end{equation}
{ 
If $(\boldsymbol{\rho}^{1,\star},   \boldsymbol{\rho}^{2,\star})$
is the solution to the linear programming problem
\begin{equation}
\label{eq:ell1_minimization_no_positivity}
\min_{\boldsymbol{\rho}^1,\boldsymbol{\rho}^2  \in \mathbb{R}^{\mathcal{N}_{\rm q}}   }
\,
\mathbf{1}^T  \left( \boldsymbol{\rho}^1 +  \boldsymbol{\rho}^2 \right),
\quad
{\rm s.t.} \; 
\left\{
\begin{array}{l}
\mathbb{A}  \, \left( \boldsymbol{\rho}^1  -   \boldsymbol{\rho}^2 \right) \leq \mathbf{b} \\
 \boldsymbol{\rho}^1,  \boldsymbol{\rho}^2  \geq  \mathbf{0}, \\
\end{array}
\right.
\end{equation}
then, 
$\boldsymbol{\rho}^{\star} = \boldsymbol{\rho}^{1,\star} -  \boldsymbol{\rho}^{2,\star}$ solves \eqref{eq:ell1_minimization_aux}:
as a result, \eqref{eq:ell1_minimization_no_positivity} can be employed to find solutions to \eqref{eq:ell1_minimization_aux}.
To prove the latter statement, we first observe that 
if  $(\boldsymbol{\rho}^{1,\star},   \boldsymbol{\rho}^{2,\star})$ solves \eqref{eq:ell1_minimization_no_positivity}, then
$\rho_i^{1,\star} \rho_i^{2,\star} = 0$ for $i=1,\ldots,\mathcal{N}_{\rm q}$; therefore, the vector
$\boldsymbol{\rho}^{\star} = \boldsymbol{\rho}^{1,\star} -  \boldsymbol{\rho}^{2,\star}$ satisfies
the constraints in \eqref{eq:ell1_minimization_aux}, and
$$
\begin{array}{ll}
\left(\boldsymbol{\rho}^{\star}\right)^+ := \max\{ \boldsymbol{\rho}^{\star}, \mathbf{0}  \} = \boldsymbol{\rho}^{1,\star}, &
\left(\boldsymbol{\rho}^{\star}\right)^- := -\min\{ \boldsymbol{\rho}^{\star}, \mathbf{0}  \} = \boldsymbol{\rho}^{2,\star}, \\[3mm]
\|  \boldsymbol{\rho}^{\star}  \|_1 = 
\mathbf{1}^T  \left( \boldsymbol{\rho}^{1,\star}  +  \boldsymbol{\rho}^{2,\star}  \right). &
\\
\end{array}
$$
If  $\boldsymbol{\rho} \in \mathbb{R}^{\mathcal{N}_{\rm q}}$ satisfies the constraints in \eqref{eq:ell1_minimization_aux}, we find that 
$(\boldsymbol{\rho}^{+} = \max\{ \boldsymbol{\rho}, \mathbf{0}  \}, \boldsymbol{\rho}^-
= - \min\{ \boldsymbol{\rho}, \mathbf{0}  \})$ satisfies the constraints in \eqref{eq:ell1_minimization_no_positivity} and
$$
\|  \boldsymbol{\rho} \|_1
=
\mathbf{1}^T  \left( \boldsymbol{\rho}^{+}  +  \boldsymbol{\rho}^{-}  \right)
\geq
\mathbf{1}^T  \left( \boldsymbol{\rho}^{1,\star}  +  \boldsymbol{\rho}^{2,\star}  \right)
=
\|  \boldsymbol{\rho}^{\star} \|_1,
$$
which is the thesis.}

Problems \eqref{eq:ell1_minimization} and \eqref{eq:ell1_minimization_no_positivity} can be solved 
using the dual simplex method. We observe that these problems require  the storage of a dense matrix of size 
$2 K   \times \mathcal{N}_{\rm q}$ and
$2 K   \times 2\mathcal{N}_{\rm q}$, respectively: even in 2D, this might be extremely demanding. 
Note that the linear programming problem \eqref{eq:ell1_minimization_no_positivity} has twice as many unknowns as \eqref{eq:ell1_minimization}.
In section \ref{sec:divide_conquer}, we illustrate a divide-and-conquer approach, which does not require the assembling of the matrix $\mathbb{G}$.


\subsubsection{Quadrature rule using EIM (EIM-EQ) }

A second approach (\cite{antil2013two})
consists in exploiting the EIM Greedy algorithm. Given $\Xi^{\rm train,eq} = \{ \mu^{\ell} \}_{\ell=1}^{n_{\rm train}^{\rm eq}}$ and $\{ \phi_j  \}_{j=1}^{J_{\rm es}}$,  we define 
$\eta^{\ell, j} :=  \eta(  \cdot; \phi_j, \mu^{\ell} )$ for $\ell=1,\ldots, n_{\rm train}^{\rm eq}$ and $j=1,\ldots,J_{\rm es}$.
Then, (i) we resort to a compression strategy to build an approximation space $\mathcal{Z}_{Q_{\rm eq}} = {\rm span } \{  \zeta_q  \}_{q=1}^{Q_{\rm eq}}$ for  
$\{  \eta^{\ell, j} \}_{\ell,j}$, 
(ii) we use the EIM Greedy algorithm to identify a set of quadrature points $\{ x_q^{\rm eq} \}$ based on $ \{  \zeta_q  \}_{q=1}^{Q_{\rm eq}}$, and 
(iii) we construct the quadrature weights.

In \cite{antil2013two}, the authors resort to a strong-Greedy procedure to determine the approximation space $\mathcal{Z}_{Q_{\rm eq}}$; in this work, we resort to POD based on the $L^2(\Omega)$ inner product.
On the other hand, the application of EIM and the subsequent construction of the quadrature points is detailed in Appendix \ref{sec:EIM}.
We remark that this approach does not in general lead to positive weights: as a result, the resulting quadrature rule should be interpreted as an approximation to problem \eqref{eq:benchmark_optimization_no_positivity}.

\subsubsection{Solution to  \eqref{eq:benchmark_optimization} using MIO (MIO-EQ)}
\label{sec:MIO_EQ}
We might also exploit  Mixed Integer Optimization (MIO) algorithms to directly solve \eqref{eq:benchmark_optimization}. With this in mind, we observe that \eqref{eq:benchmark_optimization} can be restated as
\begin{equation}
\label{eq:MIO_opt}
\min_{ 
\boldsymbol{\rho} \in \mathbb{R}^{\mathcal{N}_{\rm q}},
\mathbf{z} \in \{ 0, 1 \}^{\mathcal{N}_{\rm q}}
     } \,
\mathbf{1}^T \mathbf{z}
     \quad
{\rm s.t.} \; \;
\left\{
\begin{array}{l}
\mathbb{A} \boldsymbol{\rho} \leq \mathbf{b} \\
     \mathbf{0} \leq \boldsymbol{\rho} \leq |\Omega| \mathbf{z}
\\       
\end{array}
    \right.
\end{equation}
where $\mathbb{A}, \mathbf{b} $ are defined in \eqref{eq:ell1_minimization}. 

Problem \eqref{eq:MIO_opt} corresponds to a linear mixed integer optimization problem;
it is well-known that finding the optimal solution to \eqref{eq:MIO_opt} is in general a NP-hard problem. However, thanks to recent advances in discrete optimization, nearly-optimal solutions to the problem can be found within a reasonable time-frame.
We refer to \cite{bertsimas2005optimization,bertsimas2016best} for further discussions.
We here rely on the Matlab routine \texttt{intlinprog} to estimate the  solution to \eqref{eq:MIO_opt}.

Direct solution to  \eqref{eq:benchmark_optimization_no_positivity} requires the solution to the linear mixed integer optimization problem\footnote{Given the solution $(\boldsymbol{\rho}^1,\boldsymbol{\rho}^2, \mathbf{z}^1,\mathbf{z}^2)$    to \eqref{eq:MIO_opt_no_positivity}, it is possible to verify that 
$\boldsymbol{\rho}^{\star} = \boldsymbol{\rho}^1 - \boldsymbol{\rho}^2$ solves \eqref{eq:benchmark_optimization_no_positivity}. The proof follows from  the fact that 
$\rho_i^1 \rho_i^2 = z_i^1 z_i^2 = 0$ for $i=1,\ldots,\mathcal{N}_{\rm q}$. We omit the details.}
\begin{equation}
\label{eq:MIO_opt_no_positivity}
\min_{ 
\boldsymbol{\rho}^1,\boldsymbol{\rho}^2 \in \mathbb{R}^{\mathcal{N}_{\rm q}},
\mathbf{z}^1,\mathbf{z}^2  \in \{ 0, 1 \}^{\mathcal{N}_{\rm q}}
     } \,
\mathbf{1}^T ( \mathbf{z}^1 + \mathbf{z}^2)
     \quad
{\rm s.t.} \; \;
\left\{
\begin{array}{l}
\mathbb{A} 
(\boldsymbol{\rho}^1  \, - \, \boldsymbol{\rho}^2)
 \leq \mathbf{b} \\
     \mathbf{0} \leq \boldsymbol{\rho}^1 \leq C  \mathbf{z}^1 \\      
      \mathbf{0} \leq \boldsymbol{\rho}^2 \leq C  \mathbf{z}^2 \\        
\end{array}
    \right.
\end{equation}
where $C$ is chosen to be sufficiently larger than $|\Omega|$.
As for $\ell^1$-EQ, we note that \eqref{eq:MIO_opt_no_positivity} has twice as many unknowns  as \eqref{eq:MIO_opt};
it is thus considerably more difficult to solve.

\subsubsection{A divide-and-conquer approach for $\ell^1$-EQ and MIO-EQ}
\label{sec:divide_conquer}
In order to deal with large-scale problems, we propose a divide-and-conquer approach for $\ell^1$-EQ and MIO-EQ. Towards this end,   we define the triangulation of $\Omega$,
$\mathcal{T}^{\rm hf} = \{  {\texttt{D}}_j \}_{j=1}^{n_{\rm elem}}$ where 
${\texttt{D}}_j$ denotes the $j$-th element of the mesh and $n_{\rm elem}$ denotes the number of elements in the mesh.
Then, we introduce the partition of $\mathcal{T}^{\rm hf}$  as the set of indices
$\mathcal{J}_1,\ldots,\mathcal{J}_{N_{\rm part}} \subset \{ 1,\ldots, n_{\rm elem}\}$ such that
$\bigcup_{\ell=1}^{ N_{\rm part}   } \mathcal{J}_{\ell} = \{1,\ldots,  n_{\rm elem} \}$ and
$\mathcal{J}_{\ell} \cap \mathcal{J}_{\ell'} = \emptyset$ for $\ell \neq \ell'$.
If we assume that all quadrature points lie in the interior of the mesh elements\footnote{This condition is satisfied by standard Finite Element/Spectral Element discretizations.}, we find that  the global quadrature rule $\{ x_i^{\rm hf}, \rho_i^{\rm hf}  \}_{i=1}^{\mathcal{N}_{\rm q}}$ induces local quadrature rules
on the  subdomains $\Omega_{\ell} = \bigcup_{j \in \mathcal{J}_{\ell}}   {\texttt{D}}_j$; we denote by  
$\{   x_{i,\ell}^{\rm hf}, \rho_{i,\ell}^{\rm hf}  \}_{i=1,\ldots,\mathcal{N}_{\rm q}^{(\ell)}}$
the local quadrature points and weights associated with the $\ell$-th subdomain;
we further denote by  $\mathcal{Q}^{\rm hf, (\ell)}(\cdot)$ the high-fidelity quadrature rule on  $\Omega_{\ell}$.

Algorithm \ref{DC_procedure} outlines the divide-and-conquer  computational strategy for
\eqref{eq:ell1_minimization}; similar strategies can be derived for \eqref{eq:ell1_minimization_no_positivity}, \eqref{eq:MIO_opt},  \eqref{eq:MIO_opt_no_positivity}.
We observe that the local problems can be solved in parallel, 
and the full matrix $\mathbb{G}$ is not  assembled during the procedure.
Furthermore, we remark that for large-scale problems it might be convenient to consider recursive divide-and-conquer approaches based on several layers;  the extension is completely standard and is here omitted.
Finally, we remark that, thanks to the choice of the tolerance in \eqref{eq:DC_local}, the admissible set associated with  \eqref{eq:DC_global} is not empty, as rigorously shown in the next Proposition.

\begin{Proposition}
\label{th:well_posedness}
The admissible set associated with problem \eqref{eq:DC_global} is not empty for any choice of $\delta>0$.
\end{Proposition}
\begin{proof}
Since 
$\boldsymbol{\rho}^{(\ell)} = \boldsymbol{\rho}^{\rm hf, (\ell)}$ is   admissible for \eqref{eq:DC_local},   the admissible set associated with \eqref{eq:DC_local} is not empty.
Furthermore, 
any solution $\boldsymbol{\rho}^{\star, (\ell)}$ to \eqref{eq:DC_local} is uniformly bounded: we have indeed
$\| \boldsymbol{\rho}^{\star, (\ell)} \|_1 \leq 
\|  \boldsymbol{\rho}^{\rm hf, (\ell)}  \|_1=:C$.
Then, since the set 
$$
\left\{
\boldsymbol{\rho} \, :  \,
\|  \mathbb{G}^{(\ell)}  \boldsymbol{\rho} -  \mathbf{y}^{\rm hf,(\ell)}   \|_{\infty} \leq 
\frac{\delta}{N_{\rm part}}, \; \;
 \boldsymbol{\rho}  \geq \mathbf{0} , \; \;
 \|  \boldsymbol{\rho}  \|_1 \leq C 
\right\}
$$
is compact and $\boldsymbol{\rho}  \mapsto \| \boldsymbol{\rho} \|_1$ is continuous, the existence of a solution to \eqref{eq:DC_local} follows from the Weierstrass theorem.

Let $\boldsymbol{\rho}^{(\ell)} \in \mathbb{R}^{\mathcal{N}_{\rm q}^{(\ell)}}$ be a solution  to \eqref{eq:DC_local} for $\ell=1,\ldots, N_{\rm part}$, and let 
$\mathcal{I}_{(\ell)} \subset \{1,\ldots, \mathcal{N}_{\rm q}\}$ be the indices associated with the quadrature points in $\Omega_{(\ell)}$.
We define $\boldsymbol{\rho}^{\star} \in \mathbb{R}^{\mathcal{N}_{\rm q}}$ such that
$\boldsymbol{\rho}^{\star}(  \mathcal{I}_{(\ell)}  ) = \boldsymbol{\rho}^{(\ell)}$  for  $\ell=1,\ldots, N_{\rm part}$.

Clearly, we have  $\boldsymbol{\rho}^{\star} \geq \mathbf{0}$. Furthermore, we find
$$
\|  \mathbb{G} \boldsymbol{\rho}^{\star}   - \mathbf{y}^{\rm hf}  \|_{\infty}
=
\|
\sum_{\ell=1}^{  N_{\rm part}  } \left(
 \mathbb{G}^{(\ell)} \boldsymbol{\rho}^{(\ell)}   - \mathbf{y}^{\rm hf, (\ell)}
\right)
\|_{\infty}
\leq
N_{\rm part}
\frac{\delta}{N_{\rm part}} \leq \delta.
$$
This implies that $\boldsymbol{\rho}^{\star}$ is an admissible solution to \eqref{eq:DC_global}. 
\end{proof}

\begin{algorithm}[H]                      
\caption{
Divide-and-conquer approach for EQ training}     
\label{DC_procedure}                           
\textbf{Divide}
\begin{algorithmic}[1]
\State
Define $
\mathcal{J}_{1}, \ldots, \mathcal{J}_{N_{\rm part}} \subset \{  1,\ldots, n_{\rm elem} \}$
\medskip

\State
Compute 
the local quadrature rules on $\Omega_1,\ldots,\Omega_{N_{\rm part}}$ by solving
\begin{subequations}
\label{eq:DC_local}
\begin{equation}
\min_{\boldsymbol{\rho} \in \mathbb{R}^{\mathcal{N}_{\rm q}^{(\ell)}}   }
\,
\|  \boldsymbol{\rho}  \|_1,
\quad
{\rm s.t.} \; 
\left\{
\begin{array}{l}
\|  \mathbb{G}^{(\ell)}  \boldsymbol{\rho} -  \mathbf{y}^{\rm hf,(\ell)}   \|_{\infty} \leq 
\frac{\delta}{N_{\rm part}}
 \\
 \boldsymbol{\rho} \geq \mathbf{0} \\
\end{array}
\right.
\end{equation}
where 
\begin{equation}
\mathbb{G}^{(\ell)} =
\left[
\begin{array}{ccc}
\eta(x_{1,\ell}^{\rm hf}; \phi_1, \mu^1 ), & \ldots &  
\eta(x_{\mathcal{N}_{\rm q}^{(\ell)}}^{\rm hf}; \phi_1, \mu^1 ) \\ 
& \vdots & \\
\eta(x_{1, \ell}^{\rm hf}; \phi_{J_{\rm es}}, \mu^{n_{\rm train}^{\rm eq}} ), & \ldots &  
\eta(x_{\mathcal{N}_{\rm q}^{(\ell)}}^{\rm hf}; \phi_{J_{\rm es}},, \mu^{n_{\rm train}^{\rm eq}} ) 
\\[2mm]
1 & \ldots & 1 \\
\end{array}
\right],
\end{equation}
and 
\begin{equation}
\mathbf{y}^{\rm hf,(\ell)} = \left[
\mathcal{Q}^{\rm hf,(\ell)} \left( \eta(\cdot; \phi_1, \mu^1 )  \right),  \ldots  \, , \,
\mathcal{Q}^{\rm hf,(\ell)} \left( \eta(\cdot; \phi_{J_{\rm es}}, \mu^{n_{\rm train}^{\rm eq}} )  \right) , 
\mathcal{Q}^{\rm hf,(\ell)} \left( 1   \right)    \right]
\end{equation}
\end{subequations}
\medskip

\State
Define the set of  indices  $\mathcal{I}_{(\ell), \rm loc} \subset \{1,\ldots, \mathcal{N}_{\rm q} \}$ associated with the nonzero elements of the optimal solutions $\boldsymbol{\rho}^{(\ell)}$  to \eqref{eq:DC_local}, and set 
$\mathcal{I}_{\rm loc} = \bigcup_{\ell}  \mathcal{I}_{(\ell), \rm loc}$.
\end{algorithmic}

\textbf{Conquer}
\begin{algorithmic}[1]
\State
 Solve
\begin{equation}
\label{eq:DC_global}
\min_{\boldsymbol{\rho} \in \mathbb{R}^{\mathcal{N}_{\rm q}}   }
\,
\|  \boldsymbol{\rho}  \|_1,
\quad
{\rm s.t.} \; 
\left\{
\begin{array}{l}
\|  {\mathbb{G}}  \boldsymbol{\rho} -  {\mathbf{y}}^{\rm hf}   \|_{\infty} \leq \delta \\
  \boldsymbol{\rho}   \geq 0, \quad
  \rho_i = 0 \; {\rm if} \, i \notin \mathcal{I}_{\rm loc}  \\
  \end{array}
\right.
\end{equation}
where ${\mathbb{G}},   {\mathbf{y}}^{\rm hf} $ are defined in \eqref{eq:benchmark_optimization}.
\end{algorithmic}
\end{algorithm}

\subsection{Summary of the EQ+ES offline/online procedure}
 \label{sec:full_procedure} 

 Algorithm \ref{full_procedure} summarizes the offline/online computational procedure. 
 {  As regards the offline cost of the ES procedure, computation of the  Riesz  elements  scales with 
  $\mathcal{O}(n_{\rm train}^{\rm es}  C_{\rm riesz} )$,  while the cost of computing the POD space --- provided that $n_{\rm train}^{\rm es} \ll \mathcal{N}$ --- scales with $\mathcal{O}(    (n_{\rm train}^{\rm es})^2 \mathcal{N}    )$. Offline memory cost is $\mathcal{O}(  n_{\rm train}^{\rm es}  \mathcal{N}    )$: note that the cost of POD can be significantly reduced by resorting to hierarchical (\cite{himpe2016hierarchical})
or stochastic (\cite{balabanov2018randomized}) approaches. 
As regards the offline cost of the EQ procedure, 
memory cost  of the three EQ strategies discussed above
is  $\mathcal{O}( n_{\rm train}^{\rm eq} \, J_{\rm es} \mathcal{N}_{\rm q} )$:
as $n_{\rm train}^{\rm eq} \, J_{\rm es} \mathcal{N}_{\rm q} $ increases, offline memory costs become prohibitive.
For EIM-EQ, memory costs --- which are associated with the application of POD --- can be  reduced by resorting to hierarchical or stochastic strategies (see in particular the approach in \cite{antil2013two});
on the other hand, we might resort to  the divide-and-conquer approach discussed  in section  \ref{sec:divide_conquer}  to reduce the costs of   $\ell^1$-EQ  and MIO-EQ.  
We are not able to provide general estimates for the offline computational costs associated with  the algorithms in sections \ref{sec:ell1_EQ} and \ref{sec:MIO_EQ}: in section \ref{sec:numerics},  we provide results for the model problems considered. Finally, we observe that storage of $\{ F(x_q^{\rm eq}; \phi_j) \}_{q,j}$ requires the storage of  $\mathcal{C}_{\rm on}= D J_{\rm es} Q_{\rm eq}$ floating points; similarly, computation of $\mathbb{H}(\mu)$ in \eqref{eq:EQ_ES_estimate} can be performed through $\mathcal{O}(\mathcal{C}_{\rm on})$ operations. }

\begin{algorithm}[H]                      
\caption{
Offline/online procedure for dual norm calculations}     
\label{full_procedure}                           
\textbf{Offline stage}
\begin{algorithmic}[1]
\State
Sample $\mu^1,\ldots,\mu^{n_{\rm train}^{\rm es}} \overset{\rm iid}{\sim} {\rm Uniform}(\mathcal{P})$,
and compute $\{  \xi^{\ell} = \xi_{\mu^{\ell}}    \}_{\ell=1}^{n_{\rm train}^{\rm es}}$ 
\medskip

\State
Compute 
$\mathcal{X}_{{J_{\rm es}}} = {\rm span} \{ \phi_j \}_{j=1}^{J_{\rm es}}$ using 
POD.
\medskip

\State
Compute the  quadrature rule
$\{  \rho_q^{\rm eq}, x_q^{\rm eq} \}_{q=1}^{Q_{\rm eq}}$
using $\ell^1$-EQ, EIM-EQ or  MIO-EQ (cf. section \ref{sec:empirical_quadrature}).
\medskip

\State
Store the evaluations  of $\{ F(\cdot;  \phi_j ) \}_j$  in   $\{ x_q^{\rm eq} \}_{q=1}^{Q_{\rm eq}}$.
 \end{algorithmic}

\textbf{Online stage}

\begin{algorithmic}[1]
\State
Compute the matrix $( \mathbb{H}(\mu) )_{q,j} = \eta(x_q^{\rm eq}; \phi_j, \mu)$ in \eqref{eq:EQ_ES_estimate}.
\medskip

\State
Compute $L_{{J_{\rm es}},{Q_{\rm eq}}}(\mu)$ using  
\eqref{eq:EQ_ES_estimate}.
\end{algorithmic}
\end{algorithm}

\subsection{\emph{A priori} error analysis}
\label{sec:a_priori_error}
Given the quadrature rule $\{ x_q^{\rm eq}, \rho_q^{\rm eq}  \}_{q=1}^{Q_{\rm eq}}$, we define the maximum quadrature error:
\begin{equation}
\label{eq:maximum_quadrature_error}
\delta_{Q_{\rm eq}}^{\rm eq} := \max_{\mu \in \mathcal{P},  j=1,\ldots,J_{\rm es}} \,
\left|
\mathcal{Q}^{\rm eq}(  \eta(\cdot; \phi_j, \mu)  ) \, - \,
\mathcal{Q}^{\rm hf}(  \eta(\cdot; \phi_j, \mu)  ) 
\right|
\end{equation}
For the $\ell^1$-EQ and MIO-EQ procedures presented in section \ref{sec:empirical_quadrature},  
the maximum quadrature error $\delta_{Q_{\rm eq}}^{\rm eq}$ is  enforced to be below the target tolerance $\delta$  for all parameters in the training set  $\Xi^{\rm train,eq} = \{  \mu^{\ell} \}_{\ell=1}^{n_{\rm train}^{\rm eq}}$.
{ Note that for 
$\mu \in \mathcal{P} \setminus \Xi^{\rm train,eq} $
the quadrature error $\delta_{Q_{\rm eq}}^{\rm eq}$  might  exceed $\delta$;}
however, we can exploit \cite[Lemma 2.2]{patera2017lp} to conclude that 
$\lim_{   n_{\rm train}^{\rm eq} \to \infty  }\delta_{Q_{\rm eq}}^{\rm eq}  \leq  \delta$,  provided that $\Upsilon$ is Lipschitz-continuous in $\mu$.
Furthermore, given the reduced space $\mathcal{X}_{J_{\rm es}}$, we define the discretization error 
\begin{equation}
\label{eq:discretization_error}
\epsilon_{J_{\rm es}}^{\rm discr}
=
\max_{\mu \in \mathcal{P}}
\,
\|   \Pi_{  \mathcal{X}_{J_{\rm es}}^{\perp}  } \xi_{\mu}      \|_{\mathcal{X}}.
\end{equation}
We observe that  $\epsilon_{J_{\rm es}}^{\rm discr}$ can be estimated using the error
 indicator $E_{J_{\rm es},n_{\rm train}, n_{\rm test}}^{(\infty)}$
defined in \eqref{eq:error_indicator_max}.

Proposition \ref{th:a_priori} shows the  \emph{a priori} error bound for the estimation error
$\big| L_{J_{\rm es},Q_{\rm eq}}(\mu) - L(\mu) \big|$. 
We observe that the overall error depends on the sum of the quadrature error  $\delta_{Q_{\rm eq}}^{\rm eq}$  and  of the discretization error $\epsilon_{J_{\rm es}}^{\rm discr}$.

\begin{Proposition}
\label{th:a_priori}
Given the quadrature rule $\{ x_q^{\rm eq}, \rho_q^{\rm eq}  \}_{q=1}^{Q_{\rm eq}}$, and the empirical test space $\mathcal{X}_{J_{\rm es}}$,
the following bound holds for any $\mu \in \mathcal{P}$:
\begin{equation}
\label{eq:a_priori_error}
\big|
L_{J_{\rm es},Q_{\rm eq}}(\mu)
- L(\mu) 
\big|
\leq
\sqrt{J_{\rm es}} \delta_{Q_{\rm eq}}^{\rm eq} \, 
+
\frac{    \left( \epsilon_{J_{\rm es}}^{\rm discr}  \right)^2  }{ L(\mu) +  L_{J_{\rm es}}(\mu)},
\end{equation}
where $L_{J_{\rm es}}(\mu) :=  \|  \mathcal{L}_{\mu}  \|_{\mathcal{X}_{J_{\rm es}}'}$.
\end{Proposition}

\begin{proof}
Recalling the Riesz representer theorem, we have that
$\mathcal{L}_{\mu}(v) = ( \xi_{\mu}, v)_{\mathcal{X}}$ for all $v \in \mathcal{X}$; as a result,
$$
\left( L_{J_{\rm es}}(\mu) \right)^2 = 
\|   \mathcal{L}_{\mu}  \|_{\mathcal{X}_{J_{\rm es}}'}^2
=
\sup_{v \in \mathcal{X}_{J_{\rm es}}} \,
\frac{
( \xi_{\mu}, v  )_{\mathcal{X}}^2
}{\| v \|_{\mathcal{X}}^2  }
\,
=
\,
\| 
\Pi_{\mathcal{X}_{J_{\rm es}}} \xi_{\mu}  \|_{\mathcal{X}}^2
=
\left( L(\mu) \right)^2
-
\| \Pi_{\mathcal{X}_{J_{\rm es}}^{\perp}   } \xi_{\mu}  \|_{\mathcal{X}}^2,
$$
where in the last equality we used the projection theorem. Exploiting the identity
$(a-b)(a+b) = a^2-b^2$ we find
\begin{equation}
\label{eq:aux_estimate_proof}
 L(\mu)  
-
 L_{J_{\rm es}}(\mu)  
=
\frac{\|   \Pi_{  \mathcal{X}_{J_{\rm es}}^{\perp}  } \xi_{\mu}      \|_{\mathcal{X}}^2   }{
 L(\mu)  +  L_{J_{\rm es}}(\mu)  
} .
\end{equation}
On the other hand, exploiting  inverse   triangle   inequality and the definition of 
$\delta_{Q_{\rm eq}}^{\rm eq}$, we find
\begin{equation}
\label{eq:aux_estimate_proof_2}
\begin{array}{rl}
\displaystyle{
\big|  L_{J_{\rm es}}(\mu) -  L_{J_{\rm es},Q_{\rm eq}}(\mu) \big|
=}
&
\displaystyle{
\big| 
\sqrt{ \sum_j \,  \left( \mathcal{Q}^{\rm eq}(  \eta(\cdot; \phi_j, \mu)  ) \right)^2} 
\, - \,
\sqrt{ \sum_j \,  \left( \mathcal{Q}^{\rm hf}(  \eta(\cdot; \phi_j, \mu)  ) \right)^2} 
\big| 
} \\[3mm]
\leq
&
\displaystyle{
\sqrt{ \sum_j \,  \left( \mathcal{Q}^{\rm eq}(  \eta(\cdot; \phi_j, \mu)   \, - \,
                                   \mathcal{Q}^{\rm hf}(  \eta(\cdot; \phi_j, \mu)  ) \right)^2} 
} \\[3mm]
\leq
&
\displaystyle{
\sqrt{
\sum_j \left( \delta_{Q_{\rm eq}}^{\rm eq}  \right)^2
}
=
\sqrt{ J_{\rm es}} \, \delta_{Q_{\rm eq}}^{\rm eq}.
}
\\
\end{array}
\end{equation}
Thesis  follows by observing that 
$
\big|  L(\mu) -  L_{J_{\rm es},Q_{\rm eq}}(\mu) \big|
\leq
\big|  L_{J_{\rm es}}(\mu) -  L(\mu) \big| +
\big|  L_{J_{\rm es}}(\mu) -  L_{J_{\rm es},Q_{\rm eq}}(\mu) \big|$ and then using   \eqref{eq:aux_estimate_proof} and \eqref{eq:aux_estimate_proof_2}.
\qed
\end{proof}

\section{Approximation-then-integration approaches for dual-norm calculations}
\label{sec:theoretical_comparison_ITI_vs_EQ_ES}
We illustrate below how to apply Approximation-Then-Integration (ATI) approaches to dual norm calculations. The aim of this section is to illustrate the key differences between the EQ+ES method presented in section \ref{sec:methodology} and ATI state-of-the-art techniques, and to provide insights about the potential benefits and drawbacks of the proposed method. 

\subsection{Review of ATI-based approaches for dual norm calculations}
\label{sec:ITI_benchmark}

We briefly recall  the standard ATI-based procedure for dual-norm calculations.
We state upfront that our objective is to provide a representative example of ATI approach that will be compared with the EQ+ES approach proposed in this paper;
a thorough discussion of the available ATI approaches for the problem at hand is beyond the scope of this work.
Given $\mathcal{L}_{\mu}$ in \eqref{eq:functional_definition},   an  interpolation/approximation approach (e.g., Gappy POD, EIM,...) 
 is employed
to obtain a surrogate  of $\Upsilon$,
\begin{subequations}
\begin{equation}
\label{eq:psi_phi_surrogates}
\Upsilon_{M,\mu}(x)
=
\sum_{m=1}^{M} \, 
\left( \boldsymbol{\Theta}_M(\mu) \right)_m
 \, \zeta_m(x),
\end{equation}
where 
$ \boldsymbol{\Theta}_M: \mathcal{P} \to \mathbb{R}^M$ is a 
given function of the parameters, which can be computed in $\mathcal{O}(M^2)$ operations;
then,  the parametrically-affine surrogate of $\mathcal{L}_{\mu}$  is defined as
\begin{equation}
\label{eq:surrogate_calLM}
\mathcal{L}_{M,\mu}(v) = 
\sum_{m=1}^M \, \left( \boldsymbol{\Theta}_M(\mu) \right)_m \, \mathcal{L}_m(v), 
\end{equation}
where 
 $\mathcal{L}_m(v) = \int_{\Omega} \, \zeta_m(x)  \cdot F(x;  v) \, dx$ for $m=1,\ldots,M$.
 \end{subequations}
  Since the Riesz operator is linear, we have that
  \begin{equation}
  \label{eq:integration_then_interpolation_estimate}
  \left( L_M(\mu) \right)^2 :=
  \|  \mathcal{L}_{M,\mu} \|_{\mathcal{X}'}^2
  =
  \sum_{m,m'=1}^M
  \,
 \left( \boldsymbol{\Theta}_M(\mu) \right)_m
 \,
 \left( \boldsymbol{\Theta}_M(\mu) \right)_{m'}
 \,
\mathbb{A}_{m,m'}^{\rm off},
  \end{equation}
  where 
 $ \mathbb{A}_{m,m'}^{\rm off} := \left( \xi^m, \,
\xi^{m'}   \right)_{\mathcal{X}}$ 
and $\xi^m =  R_{\mathcal{X}} \mathcal{L}_m$, $m=1,\ldots,M$.

Identity \eqref{eq:integration_then_interpolation_estimate}  allows an efficient  offline/online computational decomposition for the estimation of $L(\mu)$.
\begin{itemize}
\item
\emph{Offline stage:} (performed once)
\begin{enumerate}
\item
find the surrogate $\mathcal{L}_{M,\mu}$ in \eqref{eq:surrogate_calLM},
\item
compute $\xi^m= R_{\mathcal{X}} \mathcal{L}_m$ for $m=1,\ldots,M$, and
\item
compute $\mathbb{A}^{\rm off} \in \mathbb{R}^{M \times M}$ in \eqref{eq:integration_then_interpolation_estimate}.
\end{enumerate}
\item
\emph{Online stage:} (performed for any new $\mu \in \mathcal{P}$)
\begin{enumerate}
\item
 evaluate ${\boldsymbol{\Theta}}_M(\mu)$, 
\item
return
$L_M(\mu)
=
\sqrt{ {\boldsymbol{\Theta}}_M(\mu)^T
\,  {\mathbb{A}}^{\rm off} 
\, {\boldsymbol{\Theta}}_M(\mu)
}.
$
\end{enumerate}
\end{itemize}
We conclude this section by stating an \emph{a priori} result and two remarks.

\begin{Proposition}
\label{th:a_priori_EIM}
Let $C_F:= \sup_{v \in \mathcal{X}} \, \frac{\int_{\Omega} \, \|  F(x; v)  \|_2^2 \, dx }{\| v \|_{\mathcal{X}}}$. Then,
$$
\big|
L(\mu)  \, - \, L_M(\mu)
\big|
\leq
C_F \, \| \Upsilon_{\mu} -  \Upsilon_{M,\mu} \|_{L^2(\Omega)}.
$$
\end{Proposition}
\begin{proof}
Applying the inverse triangle inequality and Cauchy-Schwarz inequality, we find
$$
\begin{array}{rl}
\displaystyle{
\big|
L(\mu)  \, - \, L_M(\mu)
\big|
\leq
} &
\displaystyle{
\| \mathcal{L}_{\mu} - \mathcal{L}_{M,\mu}  \|_{\mathcal{X}'} \,
=
\sup_{v \in \mathcal{X}} \,
\frac{\int_{\Omega} (\Upsilon_{\mu} - \Upsilon_{M,\mu}) \cdot F(v)  \, dx }{\|  v \|_{\mathcal{X}}} 
}
\\[3mm]
\leq
&
\displaystyle{
C_F \, \| \Upsilon_{\mu} -  \Upsilon_{M,\mu} \|_{L^2(\Omega)},
}
\end{array}
$$
which is the thesis.
\end{proof}

\begin{remark}
\label{remark_cost_ATI}
\textbf{Computational cost.}
The offline cost of a typical ATI procedure 
--- such as the one employed in the numerical results and detailed in Appendix \ref{sec:EIM} --- 
 scales with $\mathcal{O}(M C_{\rm riesz}  + M^2 \mathcal{N})$ plus the cost of defining the surrogate of $\Upsilon_{\mu}$. If we resort to POD (as in Appendix \ref{sec:EIM}), given $\{  \Upsilon^k \}_{k=1}^{n_{\rm train}}$, this cost scales with $\mathcal{O}(n_{\rm train}^2 \mathcal{N})$, provided that $n_{\rm train} \ll \mathcal{N}$. Note that our cost estimate does not include the cost of generating the snapshots $\{  \Upsilon^k = \Upsilon_{\mu^k} \}_{k=1}^{n_{\rm train}}$.
On the other hand, the online cost scales with $\mathcal{O}(M^2)$.
\end{remark}

\begin{remark}
\label{remark_ATI_ES}
\textbf{ATI+ES.}
Given the surrogate $\mathcal{L}_{M,\mu}$ in \eqref{eq:surrogate_calLM}, we might consider the approximation
$$
L_{J_{\rm es}, M}(\mu) 
\, = \, 
 \sup_{v \in \mathcal{X}_{J_{\rm es}}} \,
\frac{\mathcal{L}_{M,\mu}(v)}{\| v \|_{\mathcal{X}}}
\, = \,
\mathbb{H}^{\rm ati}(\mu) \, \boldsymbol{\Theta}_M(\mu),
$$
where $\left( \mathbb{H}^{\rm ati}(\mu) \right)_{j,m} = \mathcal{L}_m(\phi_j)$. Here, the space $\mathcal{X}_{J_{\rm es}}$ should be designed to approximate the manifold $\mathcal{M}_{\mathcal{L},M}:=
\{ R_{\mathcal{X}} \mathcal{L}_{M,\mu}: \mu \in \mathcal{P}  \}$.
Note that if we choose $\mathcal{X}_{J_{\rm es}=M} = {\rm span}\{
 R_{\mathcal{X}} \mathcal{L}_m
\}_{m=1}^M$, we have $L_{J_{\rm es}, M}(\mu)  = L_{M}(\mu) $.
In  section \ref{sec:numerics}, we investigate whether it is beneficial to consider $J_{\rm es} < M$.
\end{remark}

\subsection{Discussion}

The construction of the affine surrogate of $\Upsilon$ in \eqref{eq:psi_phi_surrogates} involves (i) the definition of an approximation space $\mathcal{Z}_M = {\rm span} \{  \zeta_m \}_{m=1}^M \subset [L^2(\Omega)]^D$, and
(ii) the definition of an interpolation/approximation procedure to efficiently compute the parameter-dependent coefficients $\boldsymbol{\Theta}_M(\mu)$ such that
$\|  \Upsilon_{\mu} - \Upsilon_{M,\mu}  \|_{L^2(\Omega)}
\approx \inf_{\zeta \in \mathcal{Z}_M} \| \Upsilon_{\mu}  - \zeta\|_{L^2(\Omega)}$.
\begin{itemize}
\item
As opposed to the  EQ+ES approach  where the estimation error  is the sum of two contributions associated with two subsequent approximations,
the only source of error in 
$\big|
L(\mu) - L_M(\mu)
\big|$ 
is the substitution
$\mathcal{L}_{\mu} \to  \mathcal{L}_{M,\mu}$.
\item
While the empirical test space $\mathcal{X}_{J_{\rm es}}$ should approximate the manifold of Riesz representers
$\mathcal{M}_{\mathcal{L}} := \{ \xi_{\mu}: \mu \in \mathcal{P}\} \subset \mathcal{X}$,
the space $\mathcal{Z}_M$ should be tailored to approximate the manifold $\{ \Upsilon_{\mu}: \mu \in \mathcal{P}  \} \subset [L^2(\Omega)]^D$;
 therefore, for $\mathcal{X} \neq L^2(\Omega)$, we do not expect the spaces $\mathcal{Z}_M$ and $\mathcal{X}_{J_{\rm es}}$ to be related.
\item
Although small approximation errors in $\Upsilon_{\mu}$ lead to small errors in dual norm prediction (cf. Proposition \ref{th:a_priori_EIM}),
the objectives of function approximation and dual norm prediction are arguably quite different: 
we thus expect --- and we empirically demonstrate in the numerical sections ---
that integration-only strategies, which bypass the task of approximating $\Upsilon_{\mu}$,  might be preferable 
when approximating $\Upsilon_{\mu}$  is significantly more challenging than   predicting the dual norm of $\mathcal{L}_{\mu}$. 
\item
{
If we neglect the cost of computing $\{\Upsilon^k = \Upsilon_{\mu^k}  \}_{k=1}^{n_{\rm train}}$,
we observed in Remark \ref{remark_cost_ATI} that 
the   offline cost of the ATI procedure scales with $\mathcal{O}(  M C_{\rm riesz} + (M^2 + n_{\rm train}) \mathcal{N} )$:
for $ n_{\rm train} = n_{\rm train}^{\rm es} \gg M$, 
this cost  is significantly lower than the cost of building the empirical test space $\mathcal{X}_{J_{\rm es}}$, 
$\mathcal{O}( n_{\rm train}^{\rm es} C_{\rm riesz}  + (M^2 + 
( n_{\rm train}^{\rm es} )^2 \mathcal{N})  )$.
However, for several problems, including the ones considered in the numerical section, computation of $\Upsilon_{\mu}$ involves the solution to a PDE: as a result, we expect that in many cases the cost associated with the construction of the empirical test space is negligible compared to the overall offline cost.
}
\end{itemize}

Proposition \ref{th:relationship_MJQ} relates the  number of quadrature points that are needed to achieve a target accuracy to the magnitude of the other discretization parameters $M$ and $J_{\rm es}$.
We postpone the proof of Proposition \ref{th:relationship_MJQ}  to Appendix \ref{sec:proof_tricky}.

\begin{Proposition}
\label{th:relationship_MJQ}
Let
$ \mathcal{L}_{M,\mu}(v) = \int_{\Omega} \, \Upsilon_{M,\mu} (x) \cdot F(x; v)  \, dx$,
$\Upsilon_{M,\mu} = \sum_{m=1}^{M} \, \Theta_m(\mu) \zeta_m$, 
satisfy  
\begin{subequations}
\label{eq:calLM_MJQ}
\begin{equation}
\label{eq:calLM_MJQ_a}
\big|  
\mathcal{L}_{M,\mu}(\phi_j) -   \mathcal{L}_{\mu}(\phi_j) 
\big| \leq 
\delta_{\rm ati}
\quad
\forall \, \mu \in \mathcal{P},
\quad
j=1,\ldots,J_{\rm es},
\end{equation}
for some tolerance $\delta_{\rm ati}>0$.
Then, if we introduce the interpolation error
\begin{equation}
\label{eq:calLM_MJQ_b}
\epsilon_{\rm ati} := \sup_{x \in \Omega, \; j=1,\ldots,J_{\rm es}, \mu \in \mathcal{P} } \, 
\big|
\left( \Upsilon_{\mu}(x) - \Upsilon_{M,\mu}(x)
\right)   \cdot F(x; \phi_j)
\big|
\end{equation}
we find that  
any solution $\boldsymbol{\rho}^{\rm opt}$ to 
\eqref{eq:benchmark_optimization_no_positivity} 
with $\delta = \delta_{\rm ati} + C_{M,J_{\rm es}} (M J_{\rm es} +  1) \epsilon_{\rm ati}$
satisfies
$\|  \boldsymbol{\rho}^{\rm opt} \|_0 \leq 
M  J_{\rm es} + 1$, where $C_{M,J_{\rm es}}$ depends on $\Upsilon_{M,\mu}$ and $\{ \phi_j  \}_j$.
\end{subequations}
\end{Proposition}

Proposition \ref{th:relationship_MJQ} suggests that the number of empirical quadrature points $Q_{\rm eq}$ should depend linearly on $J_{\rm es}$: this implies that EQ+ES is likely to become increasingly suboptimal compared to ATI approaches as $J_{\rm es}$ increases.
However,  as discussed above, since ATI approaches do not directly tackle the problem of interest, there is in practice no guarantee that computable surrogates of $\mathcal{L}$ are quasi-optimal  for a given tolerance $\delta$. 

We also observe that  if $\mathcal{L}_{\mu}$ is parametrically-affine (i.e.,
$\mathcal{L}_{\mu} = \mathcal{L}_{M, \mu}$ for some $M>0$), then
\eqref{eq:calLM_MJQ_a} and \eqref{eq:calLM_MJQ_b}  
hold with 
$\delta_{\rm ati} = \epsilon_{\rm ati} = 0$. 
As a result, Proposition \ref{th:relationship_MJQ}
shows that, for any $\delta>0$ and any choice of the training set $\mu^1,\ldots, \mu^{n_{\rm train}^{\rm eq}}$,
any solution $\boldsymbol{\rho}^{\rm opt}$ to 
\eqref{eq:benchmark_optimization_no_positivity} 
satisfies $\|  \boldsymbol{\rho}^{\rm opt} \|_0 \leq  M  J_{\rm es} + 1$.

%
%
%
%
%
%
%
%

\section{Numerical results}
\label{sec:numerics}
\subsection{Comparison between EQ+ES and an EIM-based ATI approach}
\label{sec:numerical_comparison}
	
We consider the problem of estimating the dual norm of the $\mathcal{X} = H^1(\Omega)$ functional 
\begin{subequations}
\label{eq:synthetic_model_problem}
\begin{equation}
\mathcal{L}_{\mu}(v)
= \int_{\Omega} \, \Phi( u(x; \mu)   ) \, v(x) \, dx.
\end{equation}
Here,  $\Omega = (0,3)^2$, $\mathcal{P} = [0.7,1.3]^8$, and
$u:\Omega \times \mathcal{P} \to \mathbb{R}$ is the solution to the thermal block problem (\cite[section 6.1.1]{rozza2008reduced})
\begin{equation}
 \left\{
\begin{array}{ll}
-\nabla \cdot 
\left(
\kappa(\mu) \nabla u(\mu)
\right)
=0
& 
\mbox{in}  \; \Omega \\[3mm]
\kappa(\mu) 
\frac{\partial }{ \partial n} u(\mu)
= g
& 
\mbox{on}  \; \Gamma_{1} \cup 
\Gamma_{2} 
\cup
\Gamma_{3}  \\[3mm]
u (\mu)
= 0
& 
\mbox{on}  \; \Gamma_4 \\
\end{array}
\right.
\end{equation}
where  $\Omega = \bigcup_{i=1}^9 \, \Omega_i$, and 
\begin{equation}
\kappa(x, \mu)  
=
\left\{
\begin{array}{ll}
1 & \mbox{in} \, \Omega_1, \\
\mu_i & \mbox{in} \, \Omega_{i+1}, \; i=1,\ldots,8; \\
\end{array}
\right.
\qquad
g(x)
=
\left\{
\begin{array}{ll}
1 & \mbox{on} \, \Gamma_1, \\
0 & \mbox{on} \, \Gamma_2, \\
1 - 2 x_1 & \mbox{on} \, \Gamma_3. \\
\end{array}
\right.
\end{equation}
Furthermore, we endow $\mathcal{X}$ with the inner product
$$
(w,v) = \mathcal{Q}^{\rm hf}(\nabla w \cdot \nabla v + w v).
$$
Figure \ref{fig:plot_domain}(a) shows the computational domain, and the partition $\{ \Omega_i \}_{i=1}^9$;
while Figure \ref{fig:plot_domain}(b) shows the behavior of the solution $u$ for a given value of $\mu \in \mathcal{P}$.
We rely on a $P=3$ Finite Element (FE) discretization ($\mathcal{N}=8281$, $\mathcal{N}_{\rm q} = 34200$).
Simulations are performed in Matlab $2017$b on a Desktop computer (RAM 16Gb, 800 Mhz, Processor Intel Xeon 3.60GHz, 8 cores).

We here consider two choices for $\Phi$:
\begin{equation}
\Phi_1(u) =  \log \left( 1 + e^{u + 4} \right), \quad
\Phi_2(u) = \max\{ u+4, 0\}.
\end{equation}
In statistics and Machine Learning (see, e.g., \cite{james2013introduction}), 
$\Phi_1$ is known as logistic loss, while $\Phi_2$ is known as Hinge loss;
as shown in Figure \ref{fig:plot_domain}(c),  $\Phi_1$ is a smooth approximation of $\Phi_2$.
 Our choice is motivated by the need to investigate performance for both smooth fields and relatively rough fields:
 we have indeed that  $\Phi_1 \in C^{\infty}(\mathbb{R})$, while 
 $\Phi_2 \in {\rm Lipschitz}(\mathbb{R})$. 
 \end{subequations}
	
\begin{figure}[h!]
\centering
\subfloat[] 
{  \includegraphics[width=0.33\textwidth]
 {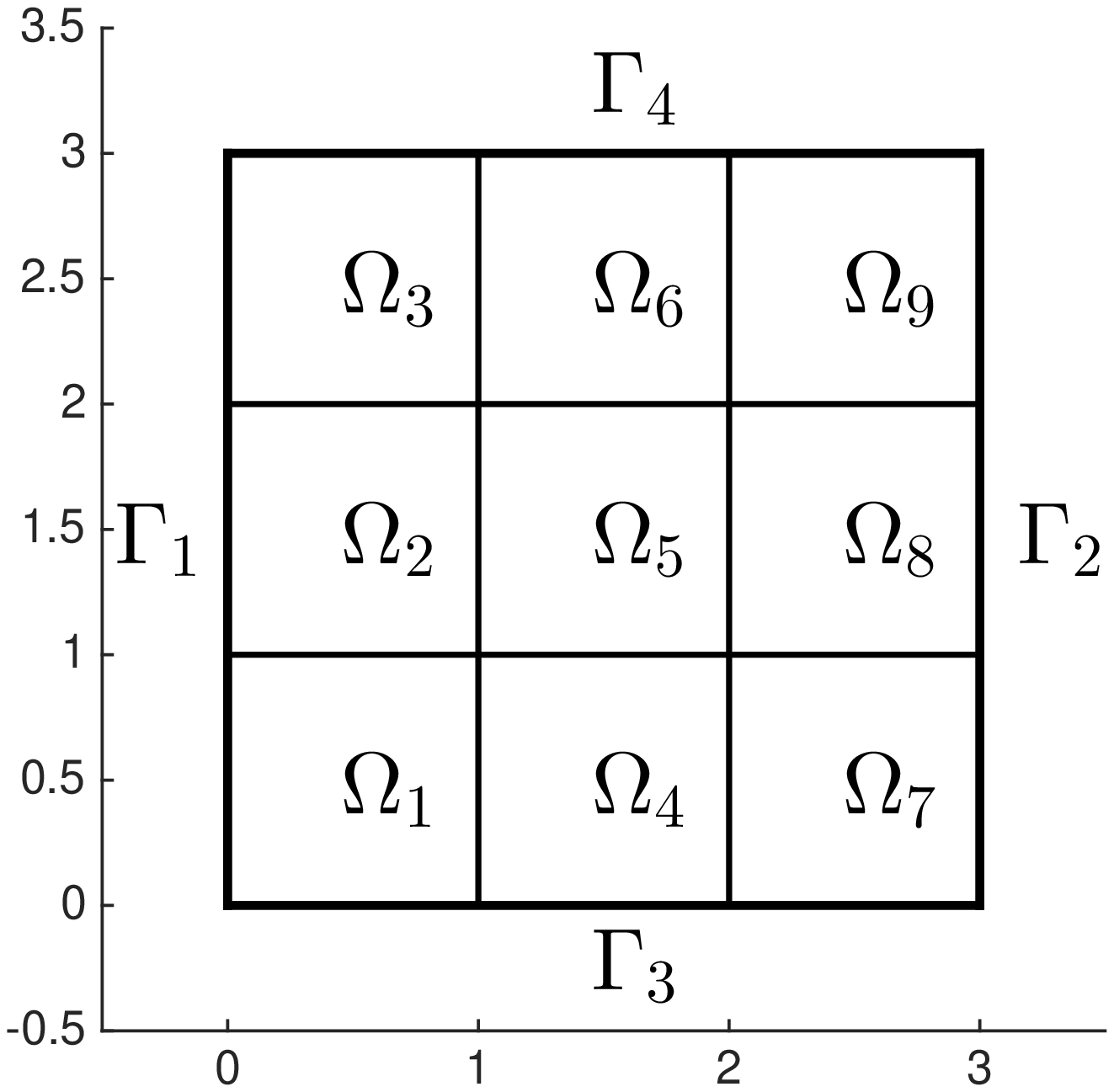}}
  ~~
 \subfloat[] 
{  \includegraphics[width=0.33\textwidth]
 {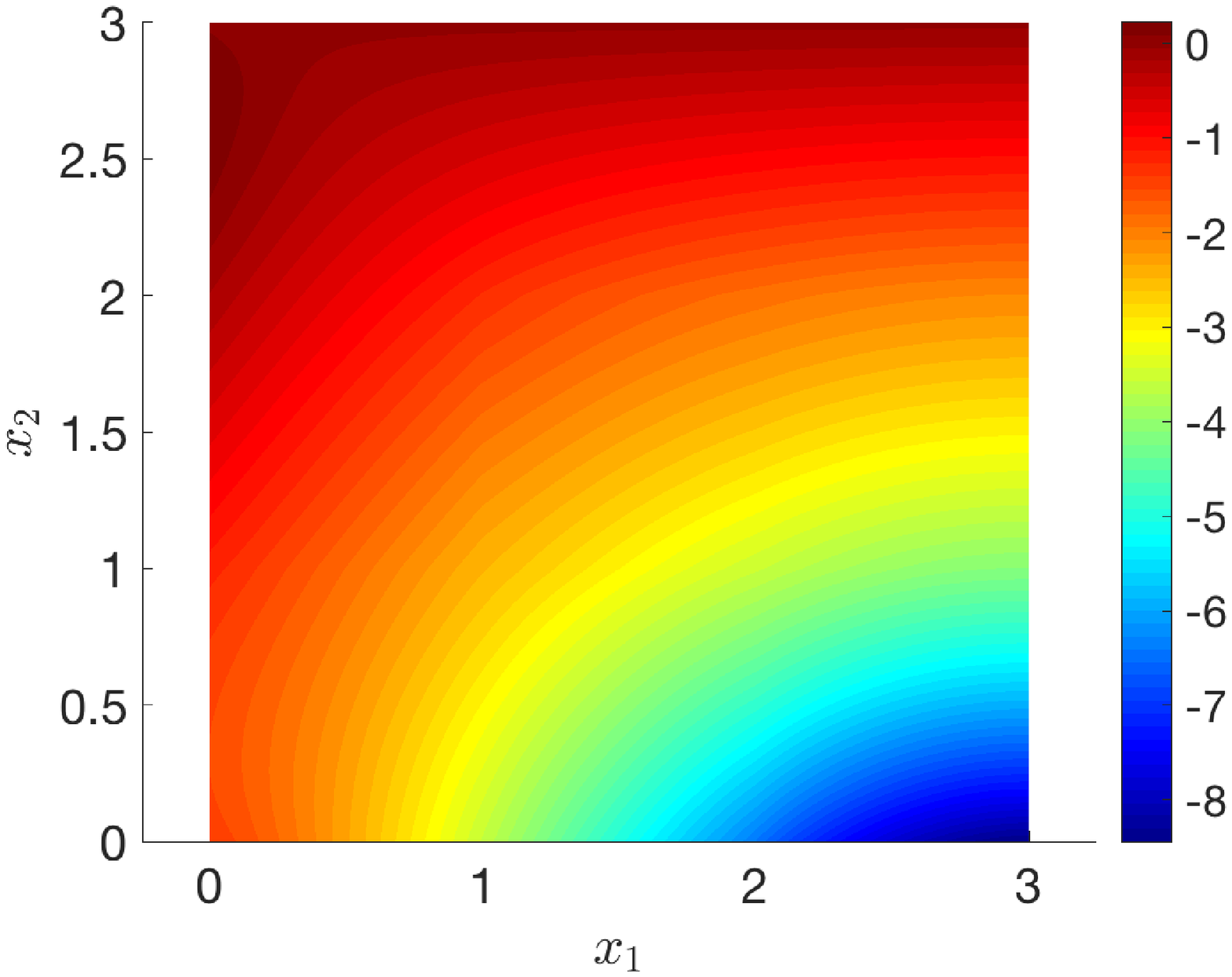}}
 ~~
 \subfloat[] 
 {  \includegraphics[width=0.33\textwidth]
 {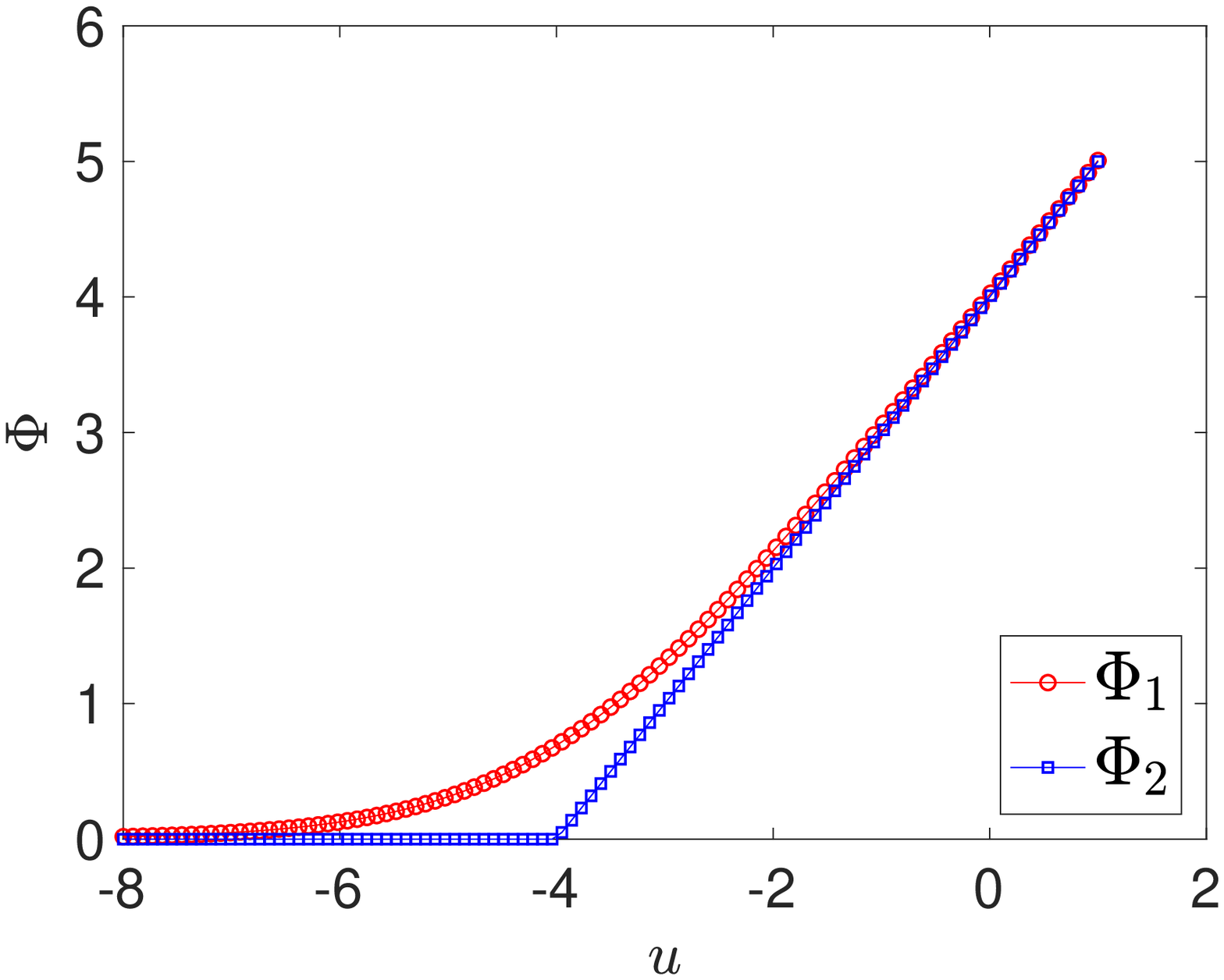}}
 
 \caption{thermal block problem. (a) computational domain. (b)
 behavior of $u$ for
 $\mu = [1.08,  0.79, 1.02,$   $1.24,  0.73,$    $1.23,  1.01, 0.84]$.
 (c) behavior of $\varphi_1$ and $\varphi_2$.}
 \label{fig:plot_domain}
  \end{figure}  

We present results for five approaches: an EIM-based ATI approach,
an EIM-based ATI+ES approach (see Remark \ref{remark_ATI_ES}),
 $\ell^1$-EQ+ES, 
EIM-EQ+ES and   MIO-EQ+ES. 
The empirical test space is generated using  the snapshot set
$\{  \Phi( u(\cdot; \mu^{\ell}) ) \}_{\ell=1}^{n_{\rm train}^{\rm es}}$
where  $\mu^1,\ldots, \mu^{n_{\rm train}^{\rm es}}$  $\overset{\rm iid}{\sim}$ ${\rm Uniform}(\mathcal{P})$, $n_{\rm train}^{\rm es} = 200$;
similarly, the approximation space associated with EIM is generated using the same snapshot set (see Algorithm \ref{EIM} in Appendix \ref{sec:EIM} for further details).
To generate the EQ rule, we impose the accuracy constraints in 
$\Xi^{\rm train, eq} = \{ \mu^{\ell} \}_{\ell=1}^{n_{\rm train}^{\rm eq}}$ 
with $n_{\rm train}^{\rm eq} = 50$; furthermore, we use the divide-and-conquer approach   discussed in section \ref{sec:divide_conquer} with $N_{\rm part}=40$: to speed up computations, local sparse representation problems (see \eqref{eq:DC_local}) are solved using $\ell^1$ for both 
$\ell^1$-EQ+ES and MIO-EQ+ES.
Moreover, we impose the threshold $T_{\rm max}=1800 [\rm s]$ for the maximum run time of MIO-EQ+ES.
For 
$\ell^1$-EQ+ES, we rely on the Matlab routine \texttt{linprog} to solve the LP problem with  default initial condition;
for MIO-EQ+ES, we rely on \texttt{intlinprog} and we consider the 
$\ell^1$-EQ+ES solution as initial condition for the optimizer.
On the other hand, performance is measured using 
$\{  \Phi( u(\cdot; \tilde{\mu}^{\ell}) ) \}_{\ell=1}^{n_{\rm test}}$,
where $\tilde{\mu}^1,\ldots, \tilde{\mu}^{n_{\rm test}}$ $\overset{\rm iid}{\sim}$ ${\rm Uniform}(\mathcal{P})$,
$n_{\rm test} = 100$.

Figure \ref{fig:Jconvergence} shows the behavior of the maximum out-of-sample error $\max_k \, L(\tilde{\mu}^k) - L_{J_{\rm es}}(\tilde{\mu}^k)$ and compares it with the squared best-fit error 
$\max_k \|  \Pi_{\mathcal{X}^{\perp}_{J_{\rm es}} } \xi_{\mu} \|_{\mathcal{X}}^2$, for the two choices of $\Phi$ considered. 
We observe that
$ L(\tilde{\mu}^k) - L_{J_{\rm es}}(\tilde{\mu}^k) \sim C \|  \Pi_{\mathcal{X}_{J_{\rm es}}^{\perp}} \xi_{\mu} \|_{\mathcal{X}}^2$: this is in good agreement with 
Eq. \eqref{eq:aux_estimate_proof} of
Proposition \ref{th:a_priori}.
We further observe that convergence with ${J_{\rm es}}$ is extremely   rapid for both $\Phi=\Phi_1$ and   $\Phi=\Phi_2$.

 \begin{figure}[h!]
\centering
\subfloat[$\Phi=\Phi_1$] {\includegraphics[width=0.48\textwidth]
 {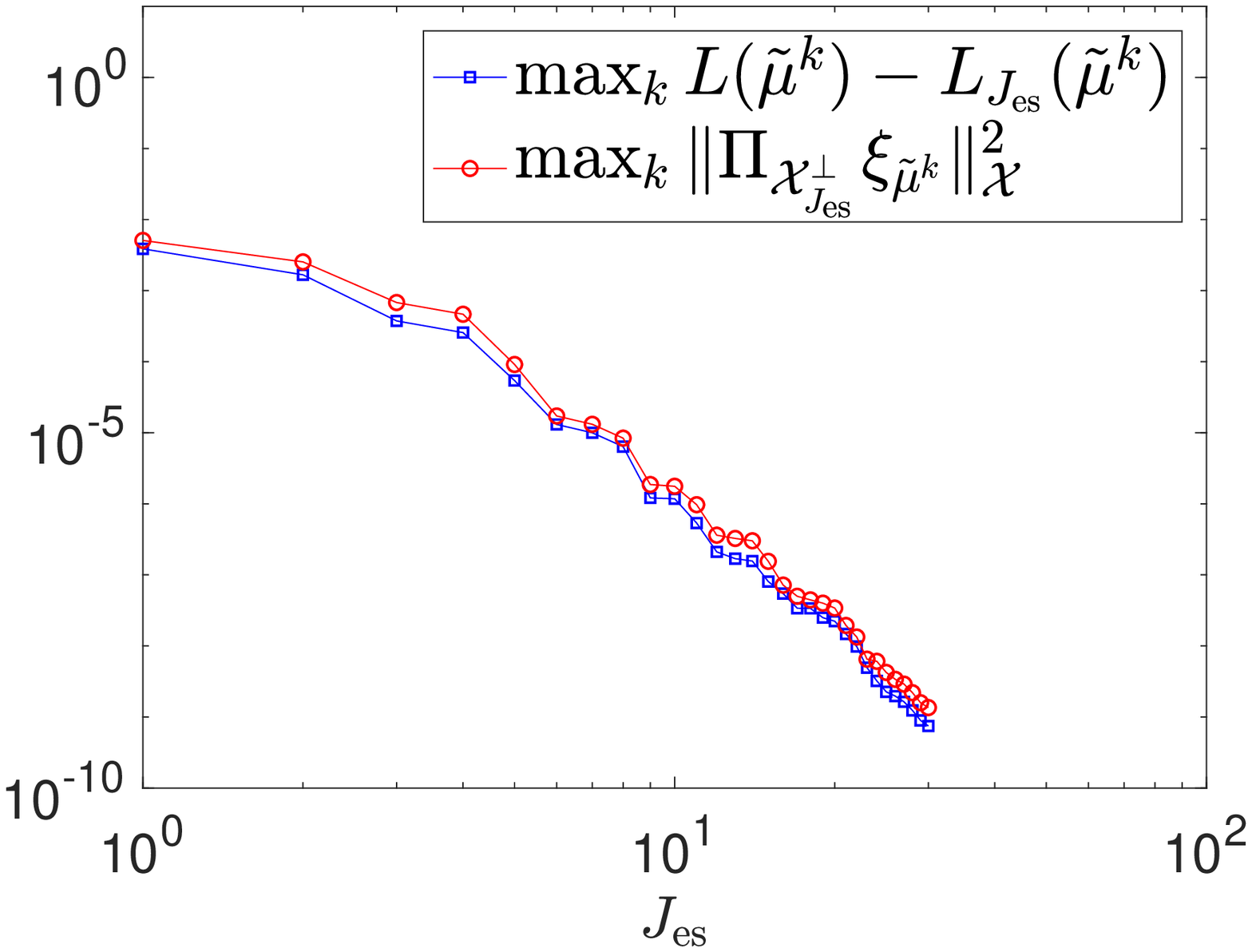}}
 ~
\subfloat[$\Phi=\Phi_2$] {\includegraphics[width=0.48\textwidth]
 {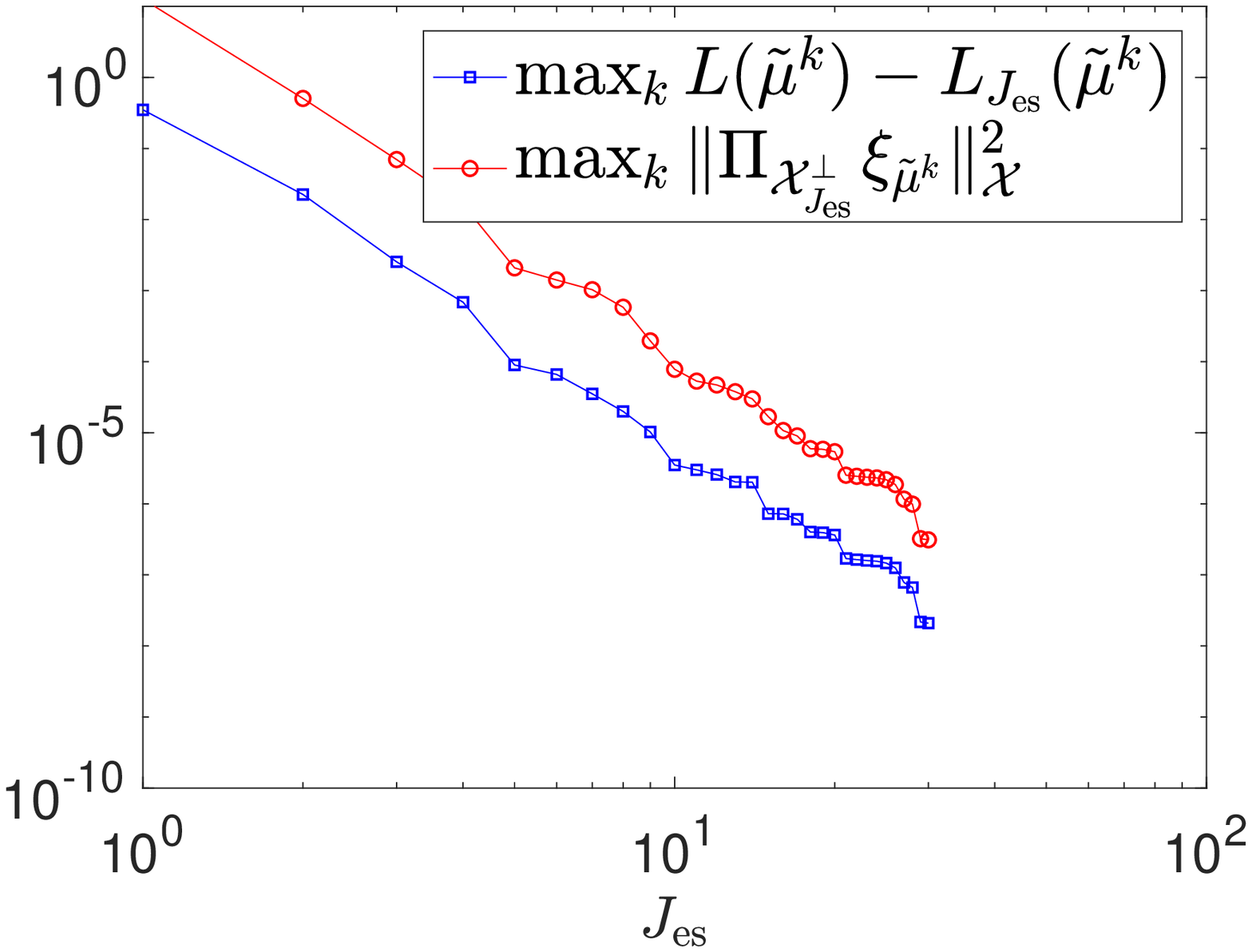}}
  
\caption{ 
behavior of  
$\max_k \, L(\tilde{\mu}^k) - L_{J_{\rm es}}(\tilde{\mu}^k)$ and  
$\max_k \|  \Pi_{\mathcal{X}_{J_{\rm es}}^{\perp}}  \xi_{\mu} \|_{\mathcal{X}}^2$
with $J_{\rm es}$, for two choices of $\Phi$.
}
 \label{fig:Jconvergence}
\end{figure} 

Figures \ref{fig:compare_pos}   show results for the five procedures:
for $\ell^1$-EQ+ES and MIO-EQ+ES, we impose the non-negativity constraint. 
Here,  $\mathcal{C}_{\rm on}$ denotes the number of floating points loaded during the online stage for the different methods: note that $\mathcal{C}_{\rm on} =M^2$ for ATI, 
$\mathcal{C}_{\rm on} =M J_{\rm es}$ for ATI+ES,  and 
$\mathcal{C}_{\rm on} = J_{\rm es} Q_{\rm eq}$ for EQ+ES.
On the other hand,  $E_{\rm test}^{(\infty)}$ is the maximum prediction error over the test set:
\begin{equation}
\label{eq:E_test_inf}
E_{\rm test}^{(\infty)} := \max_{k=1,\ldots,n_{\rm test}} \, 
\big|  L(\tilde{\mu}^k) \,  - \, \widehat{L}(\tilde{\mu}^k)  \big| ,
\end{equation}
where $\widehat{L}(\cdot)$ denotes the predicted dual norm.
For ATI, we show results for $M=1,2,\ldots,120$; for EQ+ES we show results for several prescribed tolerances --- $\delta = [10^{-2},10^{-3}, 10^{-4}, 10^{-5}, 10^{-6}]$ for $\Phi= \Phi_1$ and
$\delta = [10^{-2},10^{-3}, 5 \cdot 10^{-4}, 10^{-4}]$ for $\Phi= \Phi_2$ --- 
and two values of $J_{\rm es}$, $J_{\rm es}=10,15$. 
We recall (cf.  Remark \ref{remark_ATI_ES}) that for $J_{\rm es}\geq M$ ATI+ES
is equivalent to ATI; therefore, we set $J_{\rm es,M}= {\rm min}(M, J_{\rm es})$.

Some comments are in order.
We observe that in all our examples ATI+ES is superior to the standard ATI approach:
for $J_{\rm es} \gtrsim 10$, discretization error associated with the empirical test space is negligible compared to the error $| L(\mu) - L_M(\mu) |$. 
This also explains why increasing $J_{\rm es}$ from $10$ to $15$ does not lead to any significant improvement in accuracy.
We further observe that ATI+ES significantly outperforms the three EQ+ES procedures considered  
for $\Phi=\Phi_1$ (smooth case), while 
ATI and ATI+ES approaches are less accurate  than $\ell^1$-EQ+ES and MIO-EQ+ES for most choices of the discretization parameters for  $\Phi=\Phi_2$ (rough case).  These  results suggest that EQ procedures might guarantee better performance for irregular parametric functions $\Phi$, particularly for small tolerances.
Finally, we note that $\ell^1$-EQ+ES and MIO-EQ+ES lead to similar performance, for both choices of $\Phi$ and for all choices of the discretization parameters,
while EIM-EQ+ES is less accurate for the rough test case.

\begin{figure}[h!]
\centering
\subfloat[$\Phi=\Phi_1$, $J_{\rm es}=10$] {\includegraphics[width=0.48\textwidth]
 {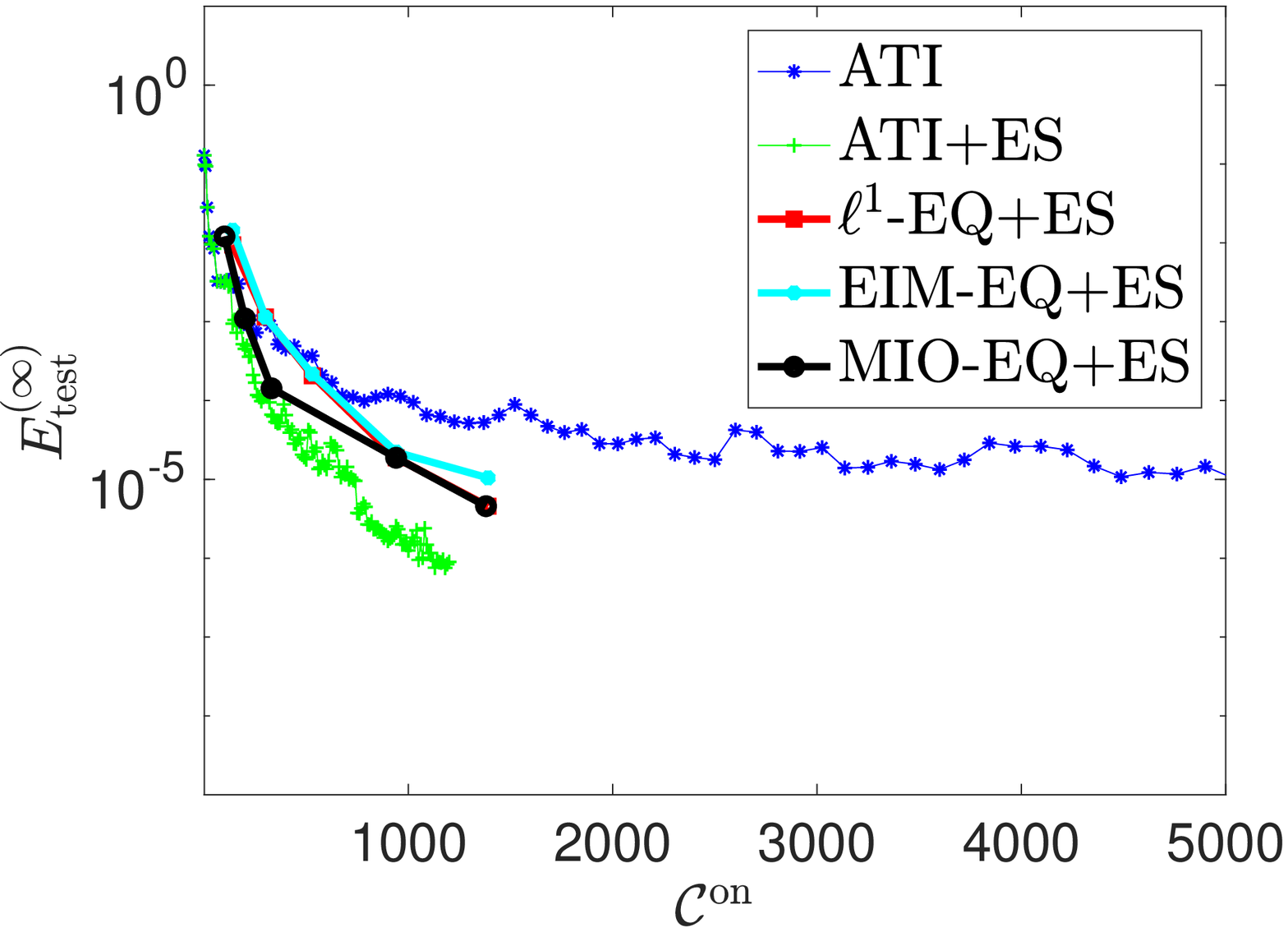}}
 ~
\subfloat[$\Phi=\Phi_1$, $J_{\rm es}=15$] {\includegraphics[width=0.48\textwidth]
 {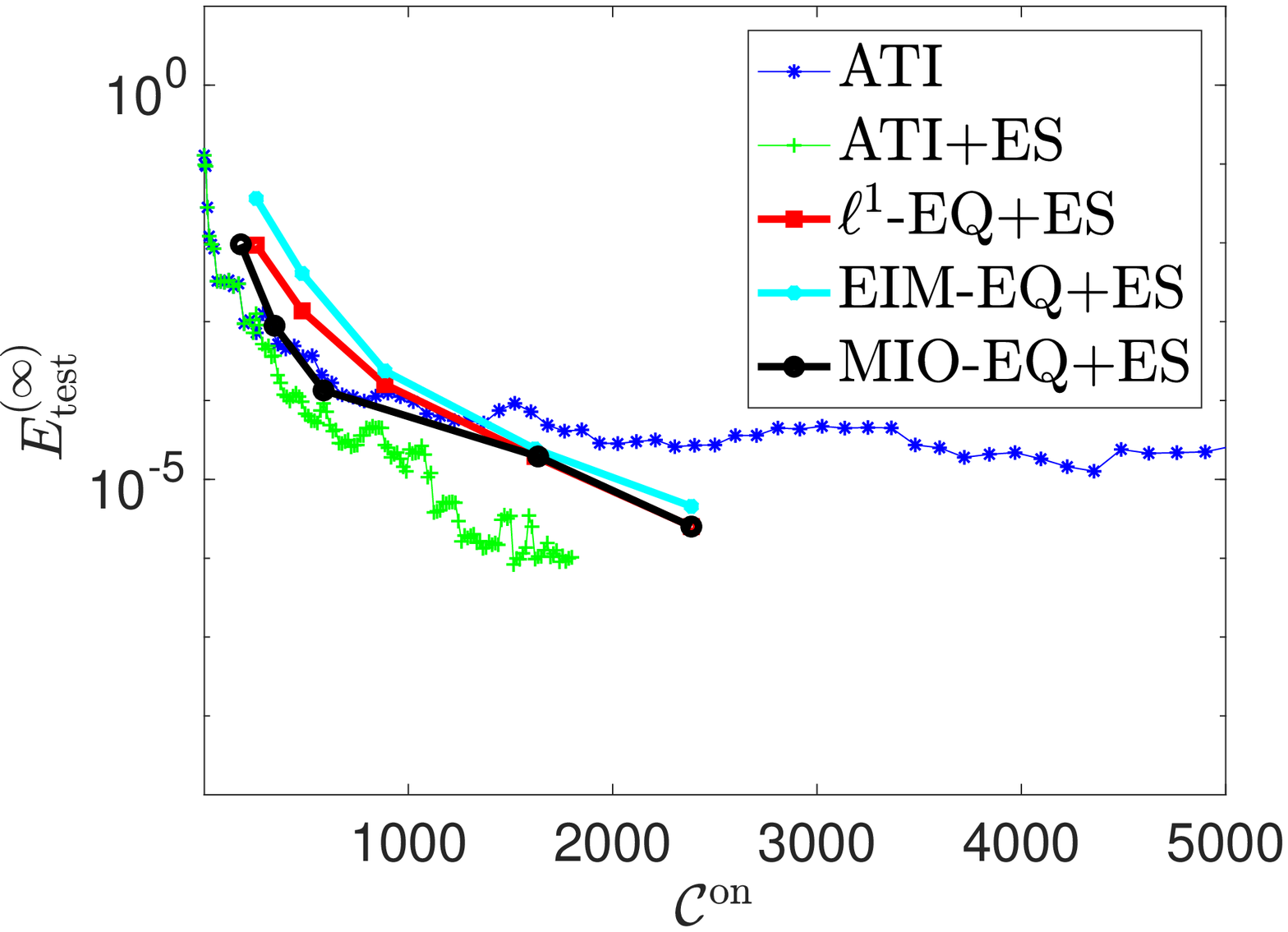}}
  
\subfloat[$\Phi=\Phi_2$, $J_{\rm es}=10$] {\includegraphics[width=0.48\textwidth]
 {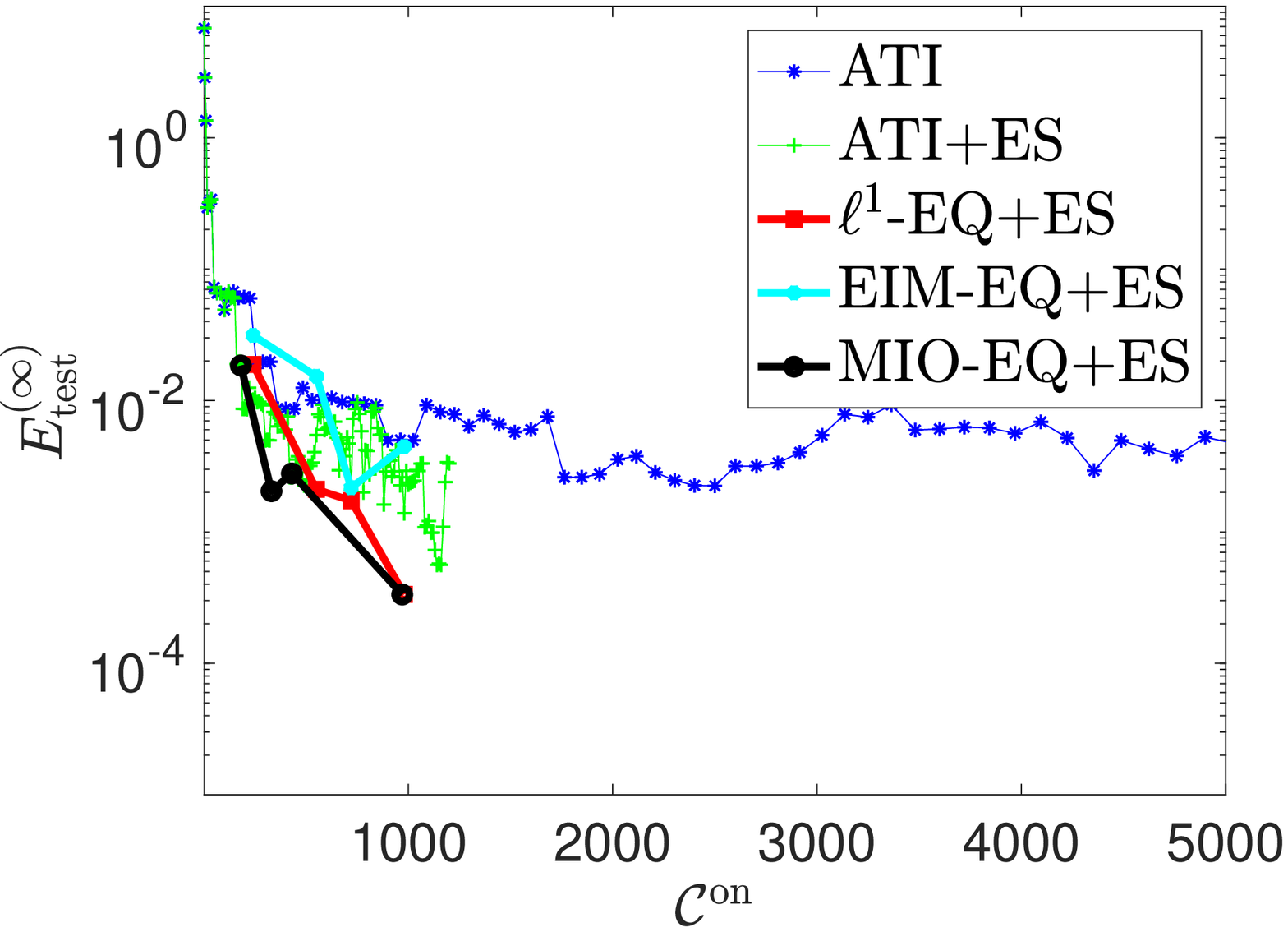}}
 ~
\subfloat[$\Phi=\Phi_2$, $J_{\rm es}=15$] {\includegraphics[width=0.48\textwidth]
 {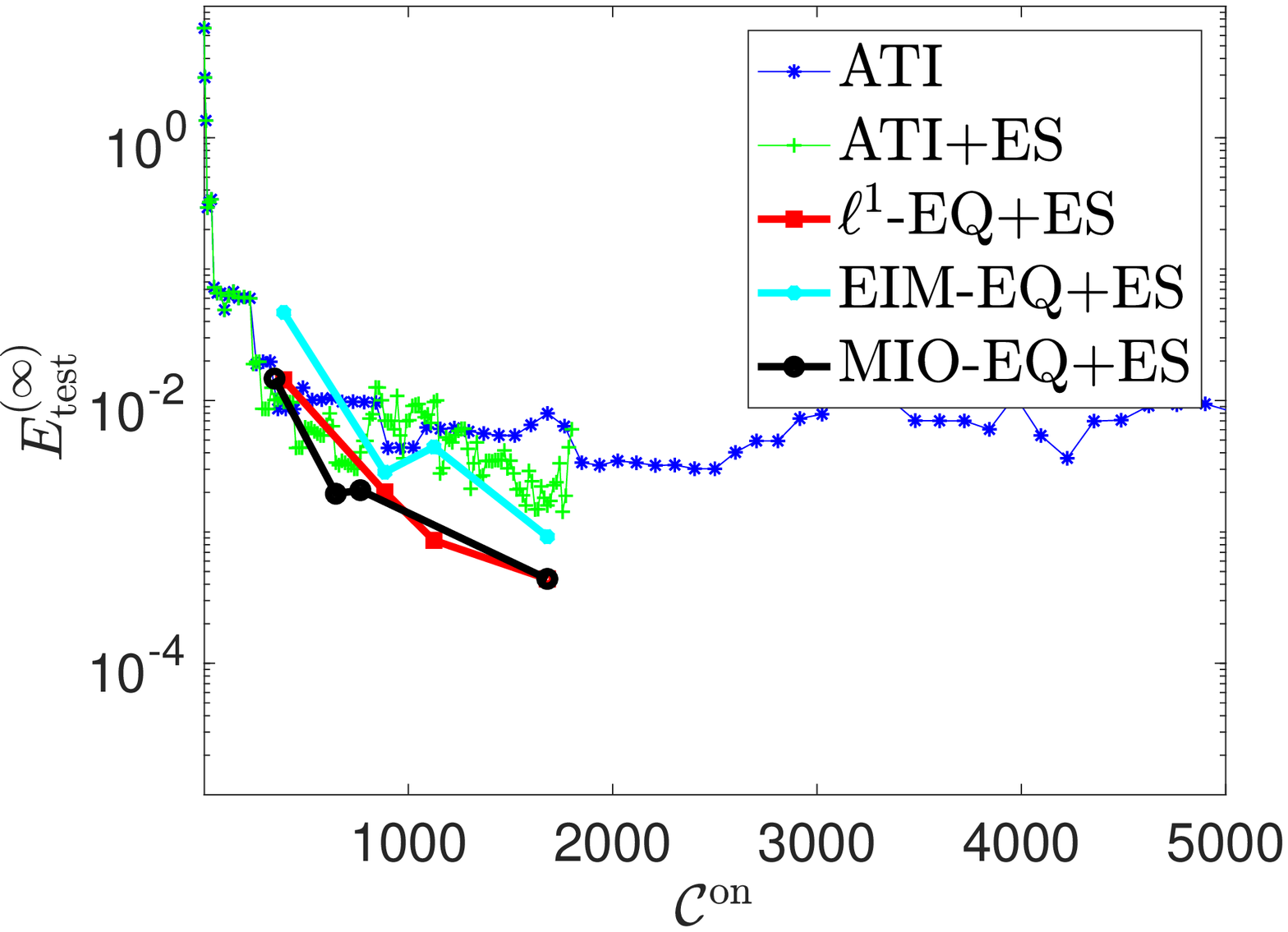}}
  
\caption{ 
performance of several dual norm prediction routines.
Behavior of $E_{\rm test}^{(\infty)}$ \eqref{eq:E_test_inf} with respect to $\mathcal{C}^{\rm on}$, for $\Phi=\Phi_1,\Phi_2$, and $J_{\rm es}=10,15$.
}
 \label{fig:compare_pos}
\end{figure} 

In Table  \ref{tab:cost_table},
we report representative offline costs of dual norm estimation procedures; to facilitate interpretation, we separate sampling costs associated with the computation of $\{ \Phi(u(\cdot; \mu^{\ell}))  \}_{\ell=1}^{n_{\rm train}^{\rm es}=n_{\rm train}}$ --- which are shared by all methods ---
from the other offline costs. 
ATI and ATI+ES are less expensive than $\ell^1$-EQ,  EIM-EQ+ES and MIO-EQ+ES; however,   due to the overhead associated with the sampling cost, costs of ATI, ATI+ES, $\ell^1$-EQ,  EIM-EQ+ES are of the same order magnitude. On the other hand, MIO-EQ+ES is considerably more expensive.

\begin{table}
\caption{representative costs of dual norm estimation procedures; we separate sampling costs from the other costs.}
\begin{center}
\begin{tabular}{|l|l|}
\hline
Method
 &
elapsed cost  [s]
 \\ [2mm]
\hline
ATI ($M=120$)
 &
 $10.40$ + $0.60$
 \\ [2mm]
\hline
ATI+ES ($M=120, J_{\rm es}=10$)
&
 $10.40$ + $1.08$
\\  [2mm]
\hline
$\ell^1$ EQ+ES ($\delta=10^{-6}, J_{\rm es}=10$, pos. weights)
&
 $10.40$ + $26.12$
\\  [2mm]
\hline
MIO EQ+ES ($\delta=10^{-6}, J_{\rm es}=10$, pos. weights)
&
  $10.40$ + $1800$
\\
\hline
EIM EQ+ES ($Q_{\rm eq}=200, J_{\rm es}=10$)
&
  $10.40$ + $4.85$
\\
\hline
\end{tabular}
\end{center}
\label{tab:cost_table}
\end{table}

In Figure \ref{fig:compare_real}, we show results for $\ell^1$-EQ+ES and MIO-EQ+ES for both non-negative weights and for real-valued weights. We observe that considering real-valued weights leads to a slight improvement in performance, particularly for $\Phi=\Phi_2$.

\begin{figure}[h!]
\centering
\subfloat[$\Phi=\Phi_1$,   $\ell^1$-EQ+ES] {\includegraphics[width=0.48\textwidth]
 {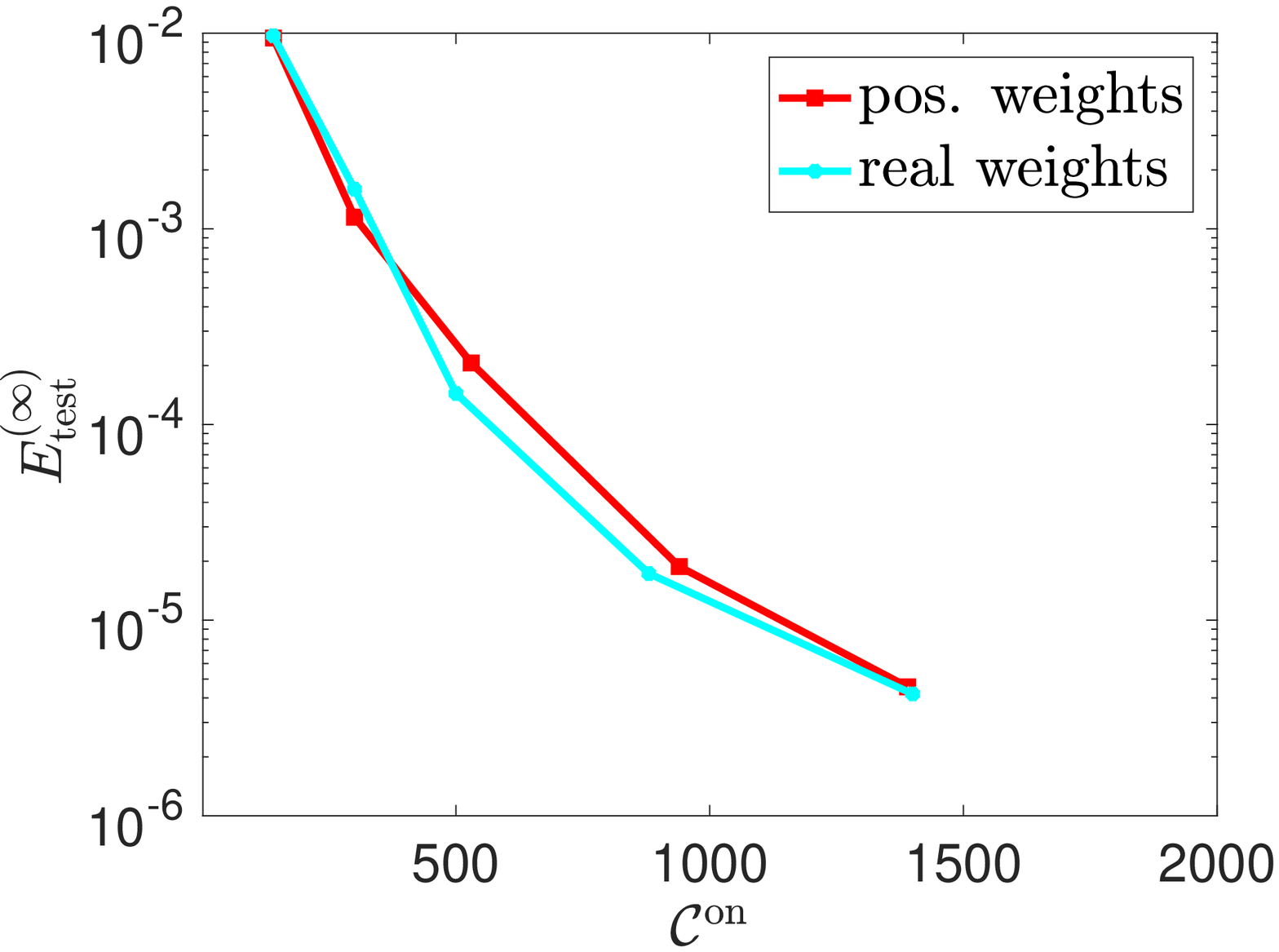}}
 ~
\subfloat[$\Phi=\Phi_1$,   MIO-EQ+ES] {\includegraphics[width=0.48\textwidth]
 {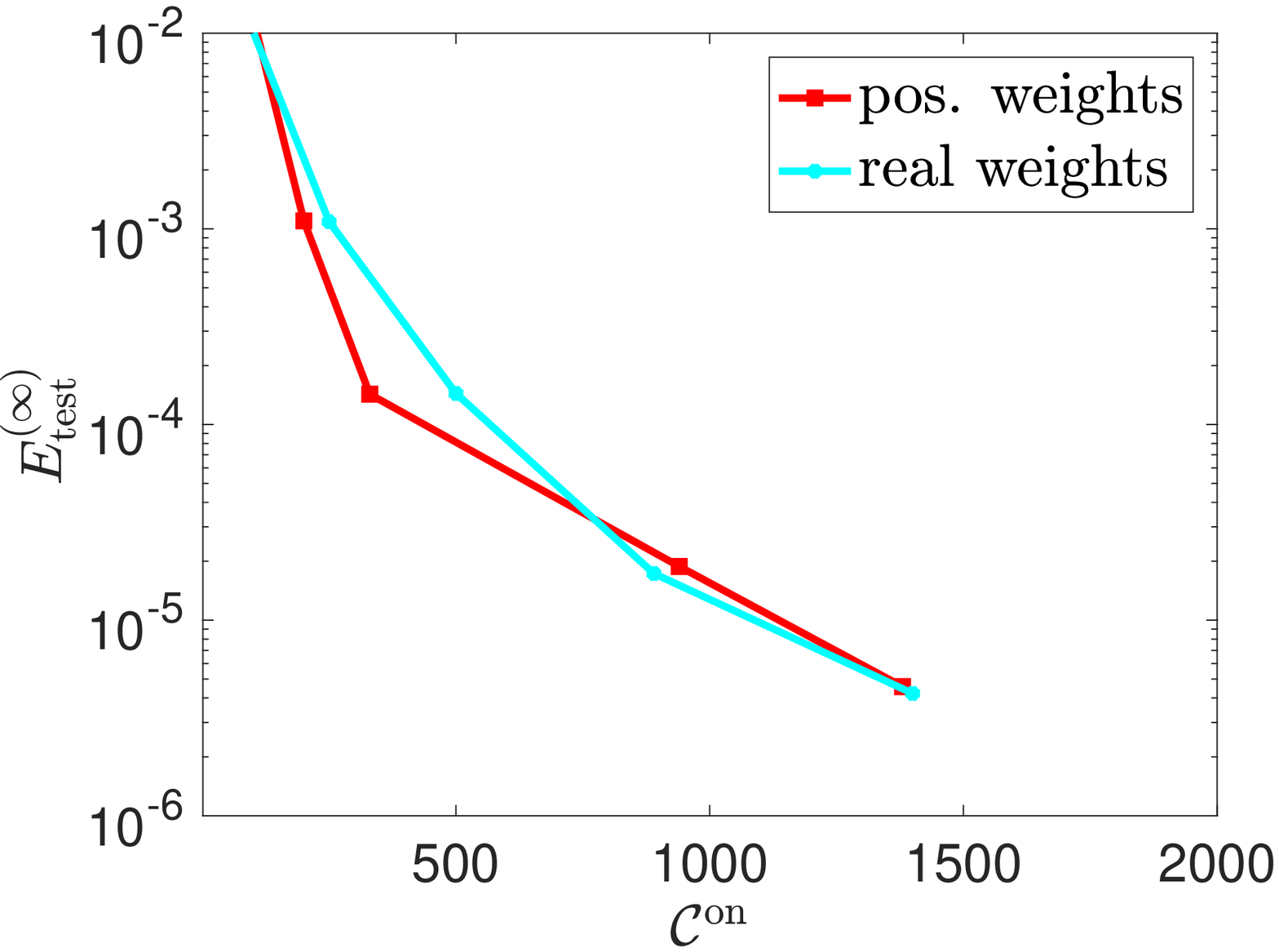}}
 
\subfloat[$\Phi=\Phi_2$,   $\ell^1$-EQ+ES] {\includegraphics[width=0.48\textwidth]
 {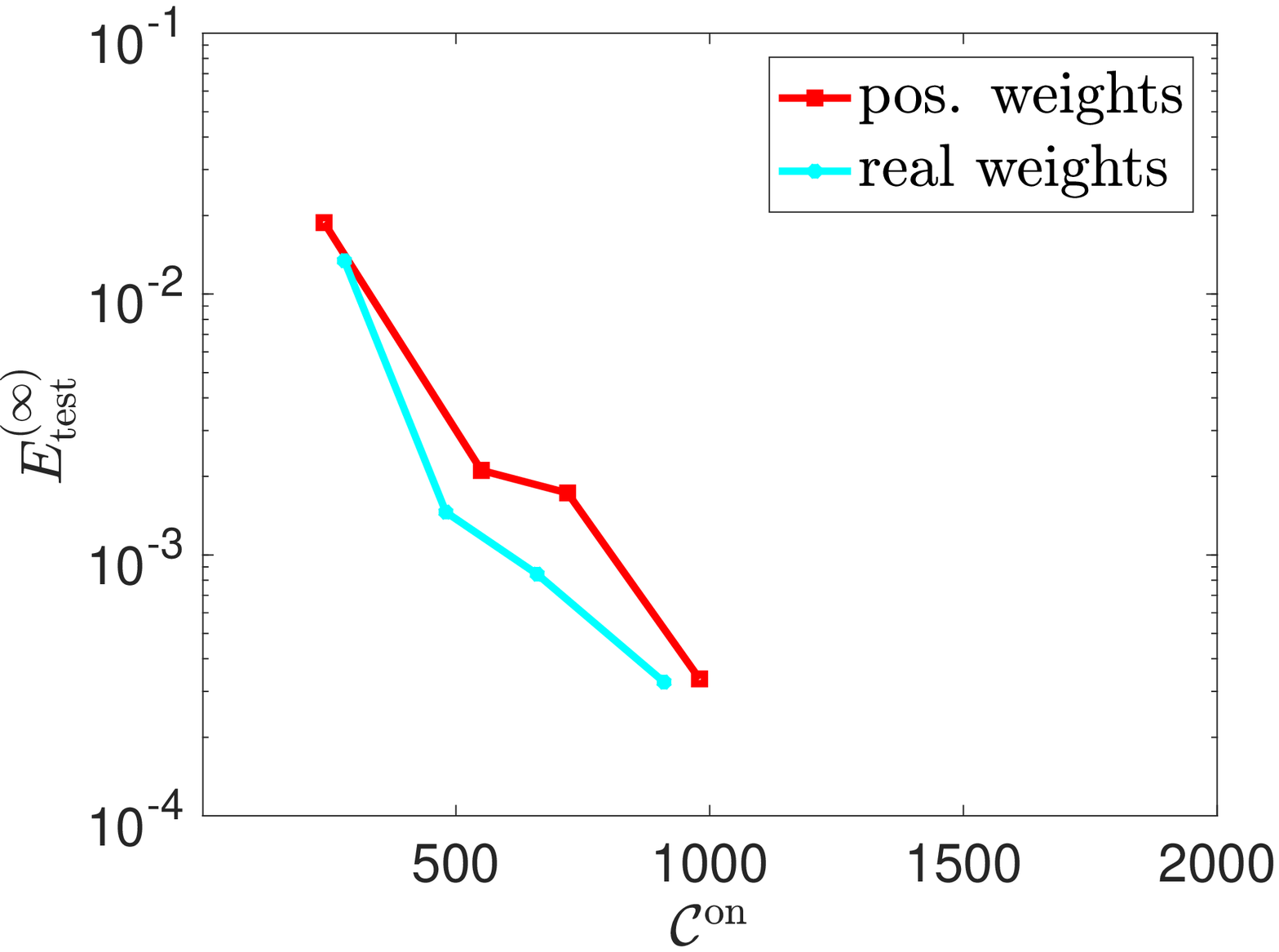}}
 ~
\subfloat[$\Phi=\Phi_2$,  MIO-EQ+ES] {\includegraphics[width=0.48\textwidth]
 {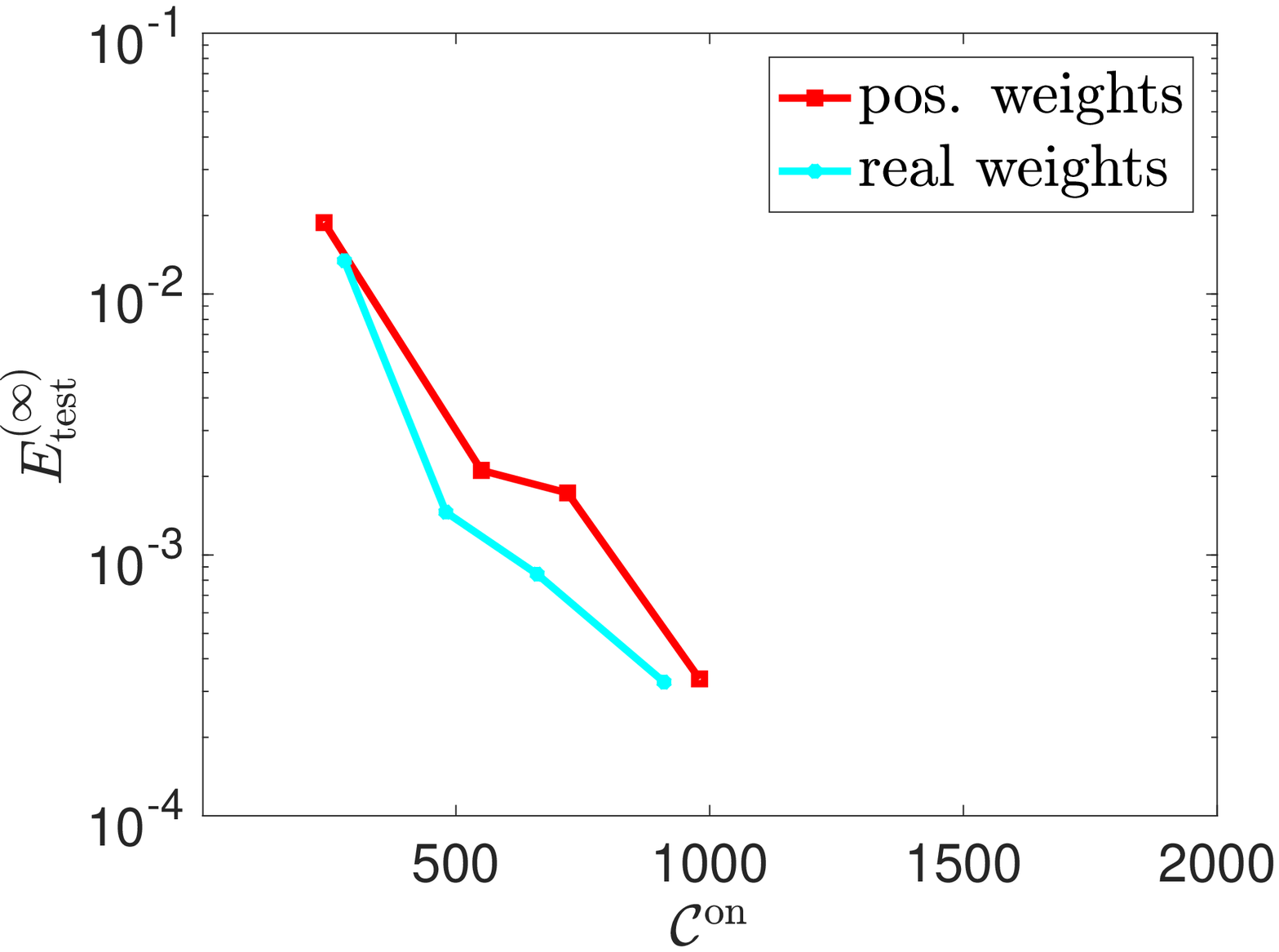}}
  
\caption{ 
influence of the non-negativity constraints  \eqref{eq:positivity_constraint} for $\ell^1$-EQ+ES and MIO-EQ+ES.
Behavior of   $E_{\rm test}^{(\infty)}$ with respect to  $\mathcal{C}^{\rm on}$, for $\Phi=\Phi_1,\Phi_2$ and $J_{\rm es}=10$.
}
 \label{fig:compare_real}
\end{figure}

In Figure \ref{fig:JQ_analysis}, we investigate performance of 
$\ell^1$-EQ+ES with positive weights for $\Phi=\Phi_1,\Phi_2$, for several values of $J_{\rm es}$.
We observe that for small values of $J_{\rm es}$, the "$J_{\rm es}$-error" associated with ES dominates;
as $J_{\rm es}$ increases, $E_{\rm test}^{(\infty)}$ reaches a threshold that depends on the value of the quadrature tolerance $\delta$. We further observe that $Q_{\rm eq}$ grows linearly with $J_{\rm eq}$: this is in good agreement with the result in Proposition \ref{th:relationship_MJQ}. 

\begin{figure}[h!]
\centering
\subfloat[$\Phi=\Phi_1$] {\includegraphics[width=0.48\textwidth]
 {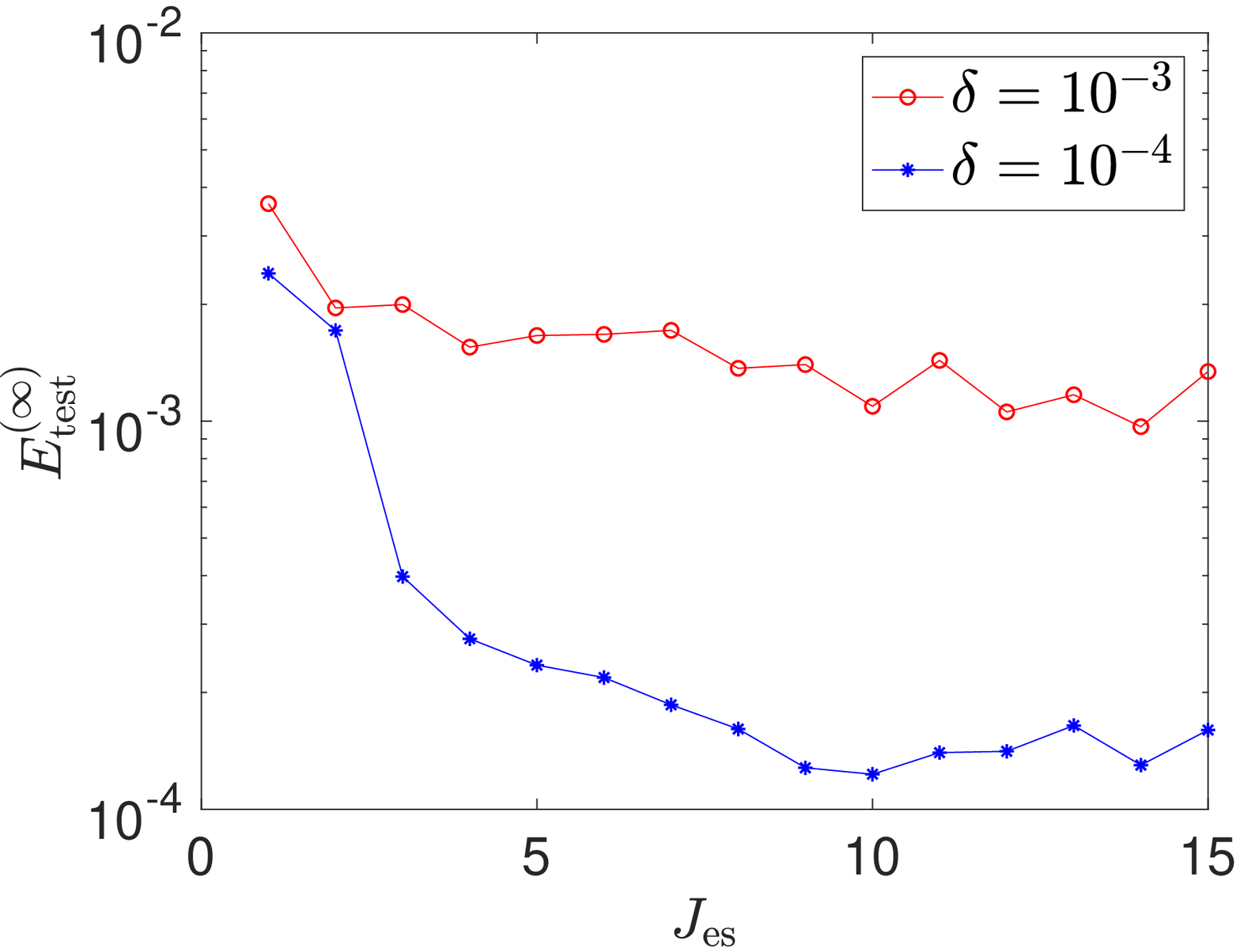}}
 ~
\subfloat[$\Phi=\Phi_1$] {\includegraphics[width=0.48\textwidth]
 {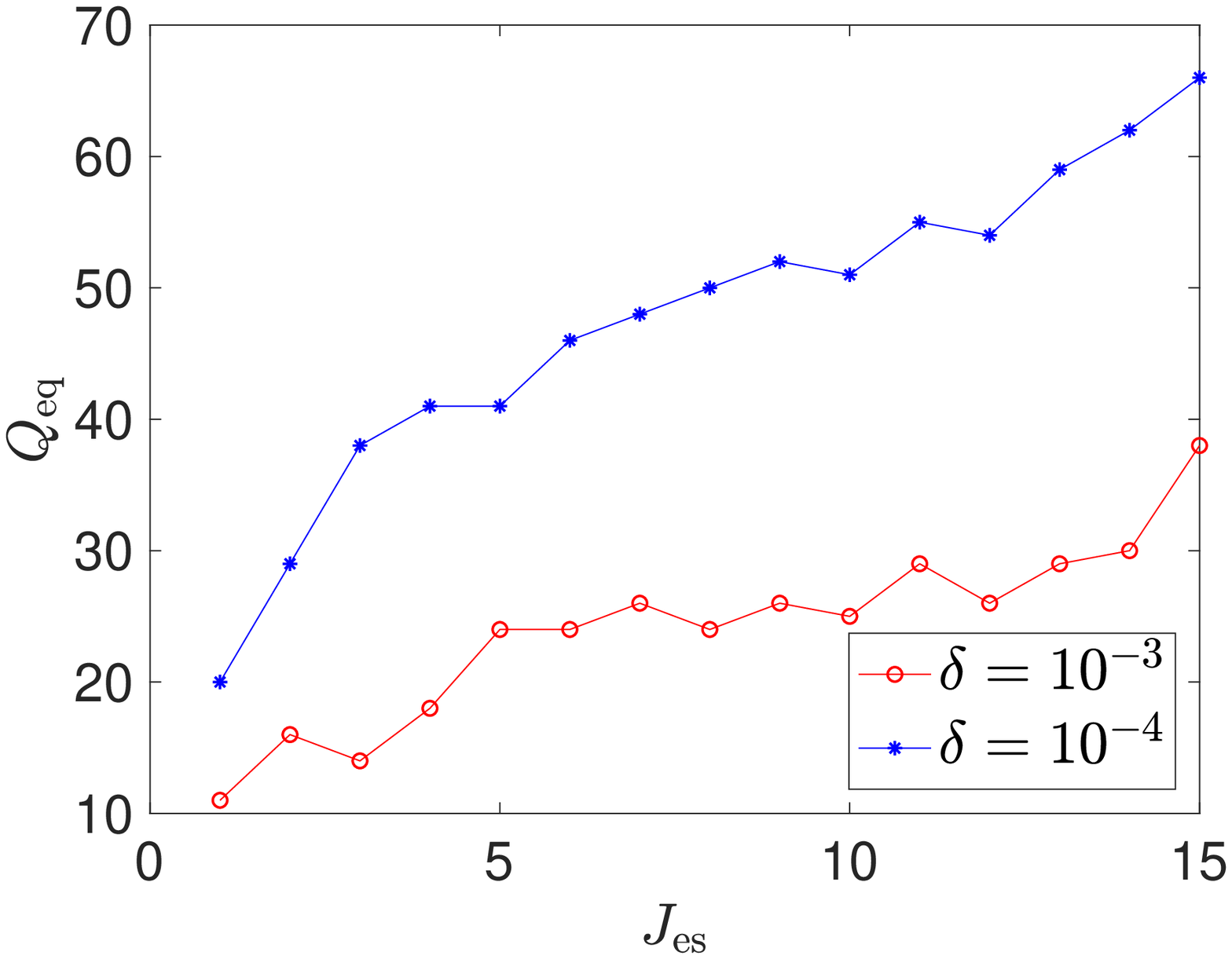}}
 
\subfloat[$\Phi=\Phi_2$] {\includegraphics[width=0.48\textwidth]
 {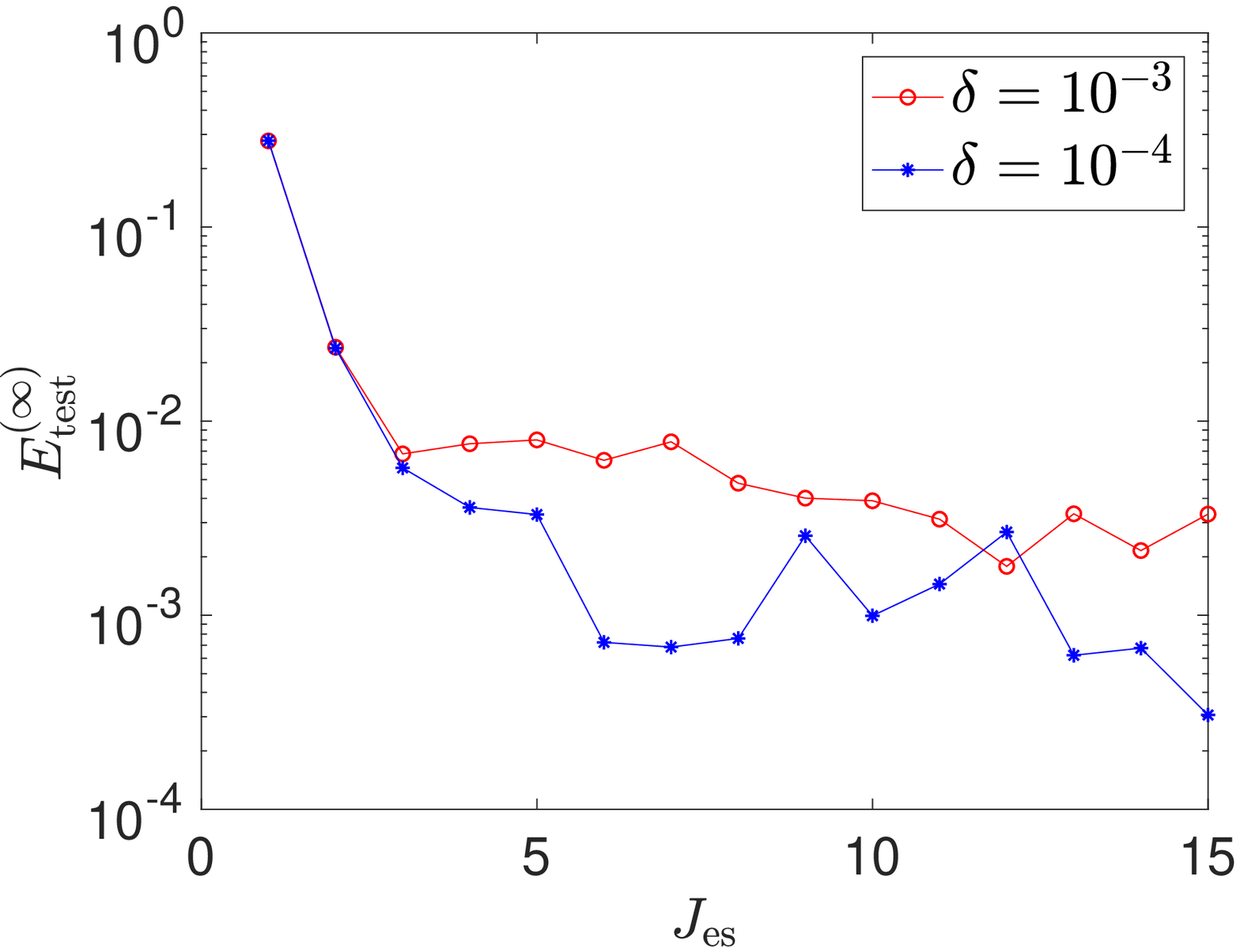}}
 ~
\subfloat[$\Phi=\Phi_2$] {\includegraphics[width=0.48\textwidth]
 {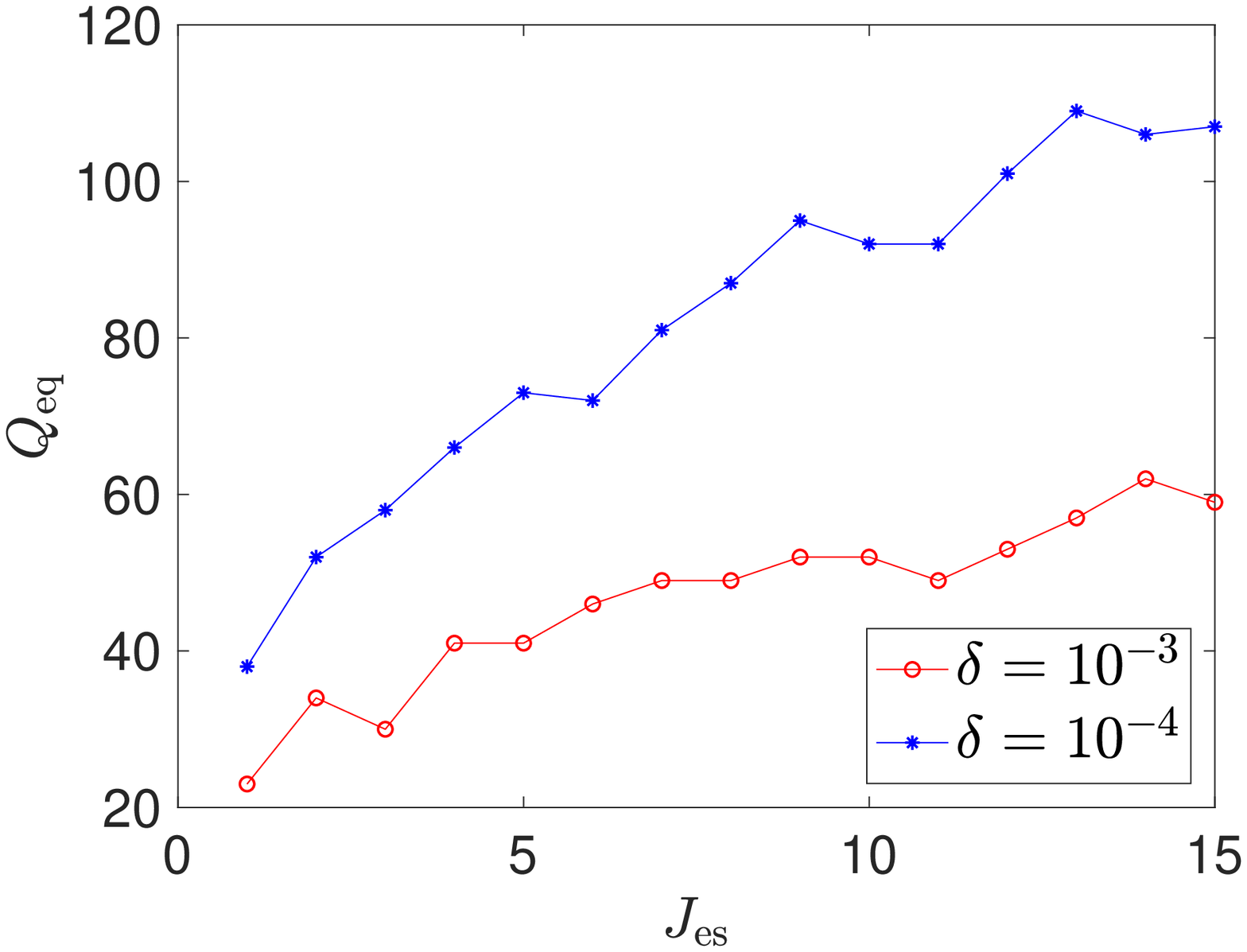}} 
   
\caption{ 
performance of $\ell^1$-EQ+ES for $\Phi=\Phi_1,\Phi_2$,
for $J_{\rm es}=1,\ldots,15$.
}
 \label{fig:JQ_analysis}
\end{figure} 

Figure \ref{fig:compare_vis} shows the interpolation  points selected by EIM, and the quadrature points obtained by applying MIO-EQ with $\delta=10^{-4}$ and real-valued weights.
Interestingly, we observe that the qualitative pattern of the points selected by the two procedures is extremely similar.

\begin{figure}[h!]
\centering
\subfloat[$\Phi=\Phi_1$] {\includegraphics[width=0.48\textwidth]
 {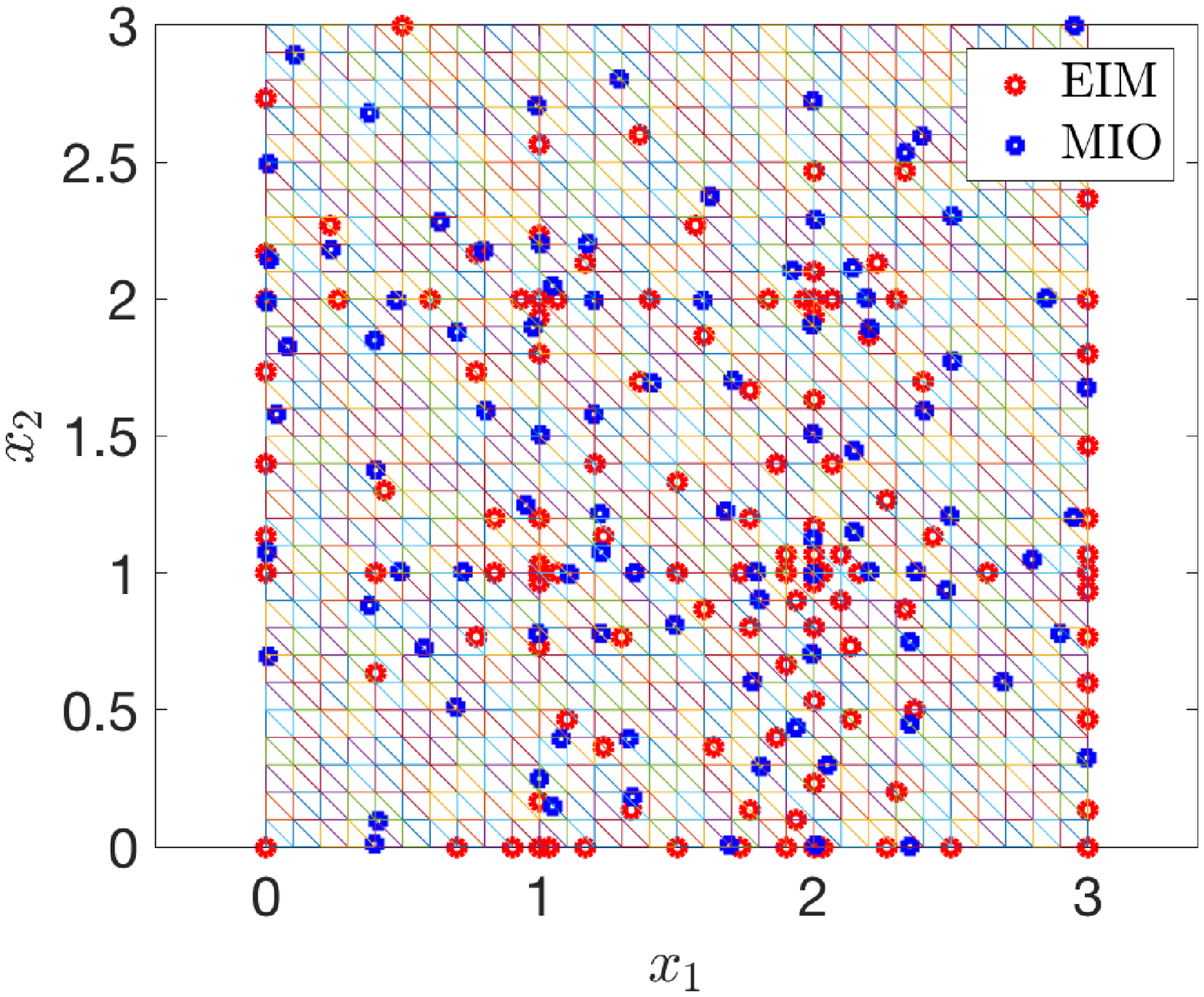}}
 ~
\subfloat[$\Phi=\Phi_2$] {\includegraphics[width=0.48\textwidth]
 {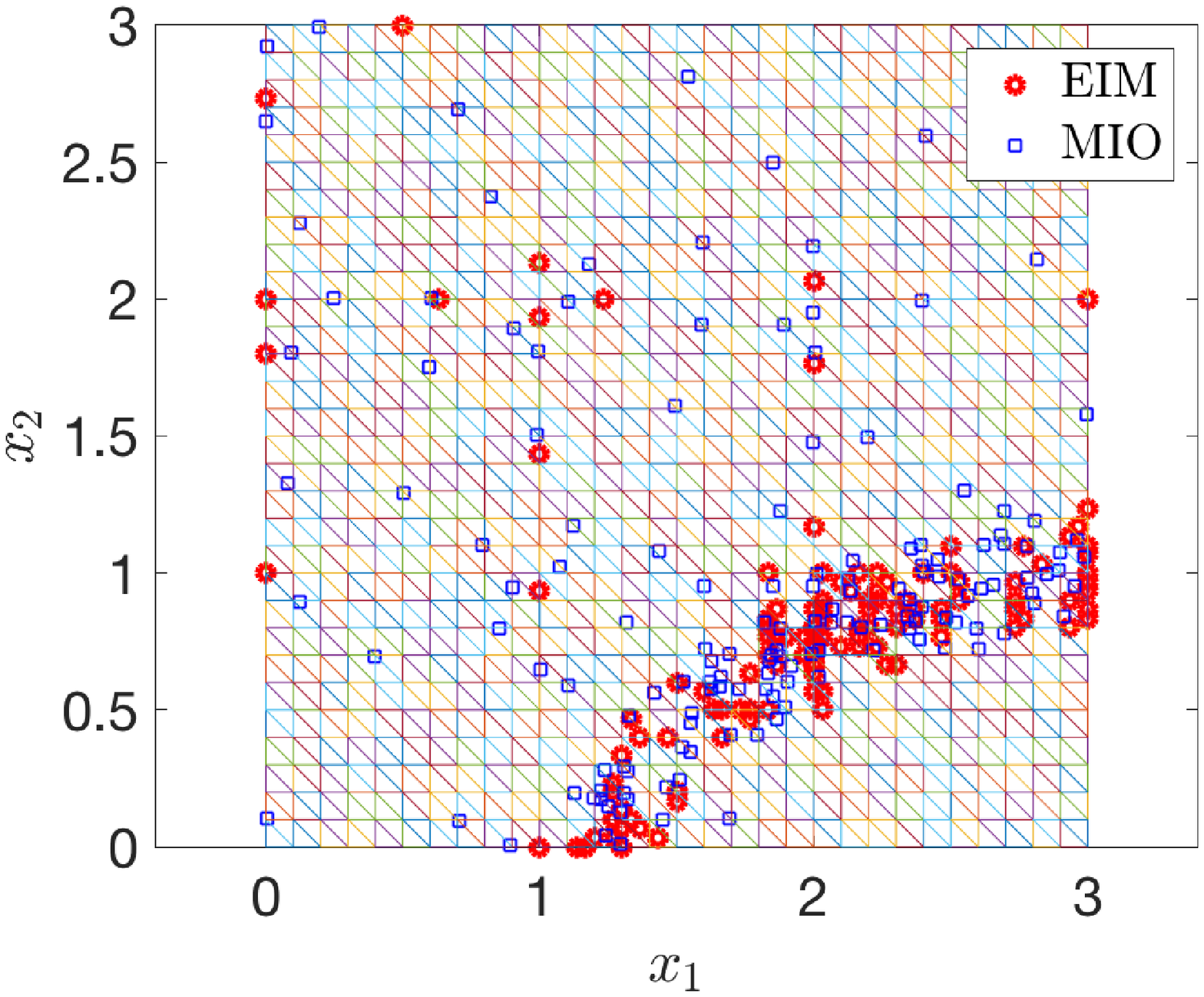}}
  
\caption{ 
EIM  interpolation points and MIO-EQ quadrature points.
}
 \label{fig:compare_vis}
\end{figure} 

\subsection{Application to residual calculations}
\label{sec:residual_navier_stokes}

\subsubsection{Problem statement}
\label{sec:NS_statement}

We apply the ATI+ES and  EQ+ES approaches  to the computation of the time-averaged residual indicator proposed in \cite{fick2017reduced} for the unsteady incompressible Navier-Stokes equations. 
We refer to \cite{fick2017reduced} for all the details concerning the  definition 
of the model problem
(a two-dimensional lid-driven cavity flow problem over a range of Reynolds numbers), and the Reduced Order Model (ROM) employed;
here, we only introduce quantities that are directly related to the residual indicator.
Given $\Omega = (-1,1)^2$ and  the time grid $\{ t^j = j \Delta t  \}_{j=0}^J$, 
we define the space $V_{\rm div} = \{ v \in [H_0^1(\Omega)]^2: \, \nabla \cdot v \equiv 0  \}$. Then,  for any sequence 
$\{  w^j \}_{j=0}^J \subset V_{\rm div}$ we define the time-averaged residual
\begin{subequations}
\label{eq:time_averaged_R}
\begin{equation}
\langle R \rangle
\left(
\{  w^j \}_{j=0}^J,
v;
{\rm Re}
\right)
=
\frac{\Delta t}{T-T_0}
\,
\sum_{j=J_0}^{J-1} \,
e( w^j, w^{j+1}, v; {\rm Re} ),
\end{equation}
where 
${\rm Re} \in \mathcal{P}=[15000,25000]$ denotes the Reynolds number, 
$T=t^J, T_0=t^{J_0}$ and
$e: V_{\rm div} \times V_{\rm div}\times V_{\rm div} \to \mathbb{R}$ is the residual associated with the discretized Navier-Stokes equations at time $t^j$ 
\begin{equation}
\begin{array}{rl}
e( w^j, w^{j+1}, v; {\rm Re} ) 
=
&
\int_{\Omega} \,
\left(
\frac{w^{j+1} - w^j}{\Delta t}
\, + \,
(w^j + R_g) \cdot \nabla (w^{j+1} + R_g)
\right) \cdot v
\\[3mm]
&
\displaystyle{
+
\frac{1}{\rm Re} \, \nabla \left( w^{j+1} + R_g \right) : \nabla v 
\,
dx.
}
\\
\end{array}
\end{equation}
Finally, $R_g \in [H^1(\Omega)]^2$, $\nabla \cdot R_g \equiv  0$ is a suitable lift associated with the Dirichlet boundary condition
$g \in [H^{1/2}(\partial \Omega)]^2$.
\end{subequations}

Our goal is to compute the dual norm of the residual 
$\langle R \rangle$,
\begin{equation}
\label{eq:dual_residual}
\Delta^{\rm u}
\left(
\{  w^j \}_{j=0}^J;
{\rm Re}
\right)
:=
\|   
\langle R \rangle
\left( \{  w^j \}_{j=0}^J,
\cdot;  {\rm Re} \right)
    \|_{V_{\rm div}'}
\end{equation}
for a given ROM 
${\rm Re} \in \mathcal{P} \mapsto \{  \hat{u}^j({\rm Re})    \}_{j}$
satisfying 
$\hat{u}^j(x; {\rm Re}) = \sum_{n=1}^N \, a_n^j({\rm Re}) $ $\zeta_n^{\rm rom}(x) $.
For this class of ROMs, we introduce the parameterized functional 
$\widetilde{  \langle R \rangle  }$ 
associated with $ \langle R 	\rangle$,
\begin{subequations}
\label{eq:tildeR}
\begin{equation}
\widetilde{  \langle R \rangle  } \, (v;  {\rm Re})
=
 \langle R 	\rangle 
 \left(
\{  \hat{u}^j({\rm Re})    \}_{j}, \, v;
{\rm Re}
 \right)
= 
\int_{\Omega} \,
\Upsilon(x; {\rm Re}) \, \cdot \, 
F(x; v) \, dx
\end{equation}
where
$F(\cdot; v) = [(\nabla v_1)_1, (\nabla v_1)_2,(\nabla v_2)_1, (\nabla v_2)_2, v_1,v_2]$ and
$\Upsilon = [\Psi_{1,1}, \ldots \Psi_{2,2}, \Phi_1,\Phi_2]$, with 
\begin{equation}
\begin{array}{rl}
\Psi(\cdot; {\rm Re} ) =  &
\displaystyle{
\frac{1}{\rm Re} \, \sum_{n=1}^N \, a_n^+ \, \nabla \zeta_n^{\rm rom} + \nabla R_g,
}
\\[3mm]
\Phi(\cdot; {\rm Re} ) =  &
\displaystyle{
\sum_{n=1}^N \, 
\left(
\frac{a_n^J - a_n^{J_0}}{T-T_0}
\right)
\,
\zeta_n^{\rm rom}
\, + \, 
a_n^+ \, \left( R_g \cdot \nabla \right) \zeta_n^{\rm rom} 
\, + \, 
a_n^- \, \left( \zeta_n^{\rm rom} \cdot \nabla \right) R_g 
 }
\\[3mm]
&
\displaystyle{
+ \,
\sum_{m,n=1}^N \, \bar{c}_{m,n}  \left( \zeta_n^{\rm rom} \cdot \nabla \right) \zeta_m^{\rm rom}
\, + \,
 \left( R_g  \cdot \nabla \right) R_g
},
\\
\end{array}
\end{equation}
and
\begin{equation}
a_n^+ = \frac{\Delta t}{T - T_0} \sum_{j=J_0+1}^J \, a_n^k,
\quad
a_n^- = \frac{\Delta t}{T - T_0} \sum_{j=J_0}^{J-1} \, a_n^j,
\quad
\bar{c}_{m,n}
=
\frac{\Delta t}{T - T_0} \sum_{j=J_0}^{J-1} \,  a_m^{j+1}  \,
\, a_n^j.
\end{equation}
Note that the functional $\widetilde{\langle R \rangle}$ is 
parametrically affine; however, the number of expansion's terms $M_{\rm R}$ is
equal to $N^2 + 3N + 2$, and is thus  extremely large for practical values of $N$.
\end{subequations}

The  functional $\widetilde{\langle R \rangle} $   \eqref{eq:tildeR} is of the form studied in this paper; for this reason, we can apply the 
techniques presented in sections \ref{sec:methodology} and
\ref{sec:theoretical_comparison_ITI_vs_EQ_ES}
 to estimate its dual norm. 
We consider $T=10^3$, $T_0= 500$, $\Delta t = 5 \cdot 10^{-3}$, 
and we consider the constrained Galerkin ROM proposed in \cite{fick2017reduced} anchored in ${\rm Re}=20000$,
for two values of the ROM dimension $N$.
The high-fidelity discretization is based on a P=8 spectral element discretization with $\mathcal{N}=25538$ degrees of freedom and 
$\mathcal{N}_{\rm hf}=36864$ quadrature points.

\subsubsection{Numerical results}
\label{sec:numerics_NS}

We consider EIM-based ATI(+ES) and $\ell^1$-EQ+ES to approximate the dual norm of $\widetilde{\langle R \rangle} $. To generate the empirical test space, we use $n_{\rm train}^{\rm es}=150$ uniformly-sampled Reynolds numbers ${\rm Re}^1,\ldots,{\rm Re}^{\rm n_{\rm train}}$ in $\mathcal{P}$. Then, to generate the EQ rule, we consider the tolerance $\delta=10^{-7}$, we impose the accuracy constraints for $n_{\rm train}^{\rm eq}=50$ parameters, and we use the divide-and-conquer strategy discussed in section \ref{sec:divide_conquer} with $N_{\rm part}=32$.
On the other hand, to generate the ATI approximation we employ the same training set used for the generation of the empirical test space. Numerical results are presented for $M=50$ and $M=100$,  and $J_{\rm es}=50$.
Note that for $M= J_{\rm es}$ we have an exact ATI approximation.
 To assess performance, 
we consider $n_{\rm test} = 11$ equispaced parameters.

Figure \ref{fig:NS_EJ} shows the behavior of 
$E_{J_{\rm es}}^{\infty, \rm rel}
=
\max_{\rm Re} \, \frac{\|  \Pi_{\mathcal{X}_{J_{\rm es}^{\perp}}} \xi_{\rm Re}      \|_{\mathcal{X}}}{\| \xi_{\rm Re} \|_{\mathcal{X}}}
$ over the training set and over the test set, for two values of $N$.
Note that for $J_{\rm es} \gtrsim 50$ the relative error is below $\mathcal{O}(10^{-1})$.

\begin{figure}[h!]
\centering
\subfloat[$N=60$] {\includegraphics[width=0.48\textwidth]
 {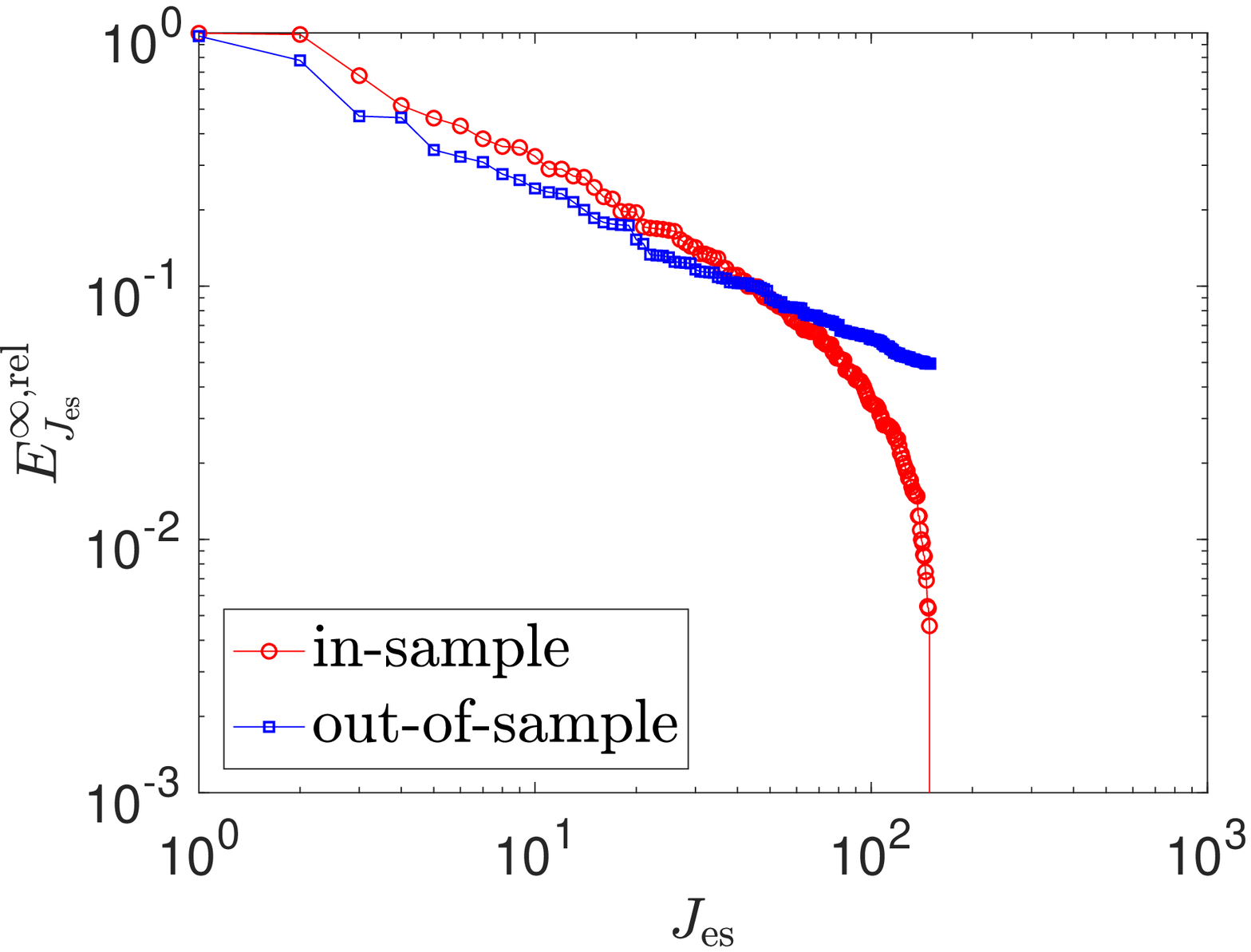}}
 ~
\subfloat[$N=80$] {\includegraphics[width=0.48\textwidth]
 {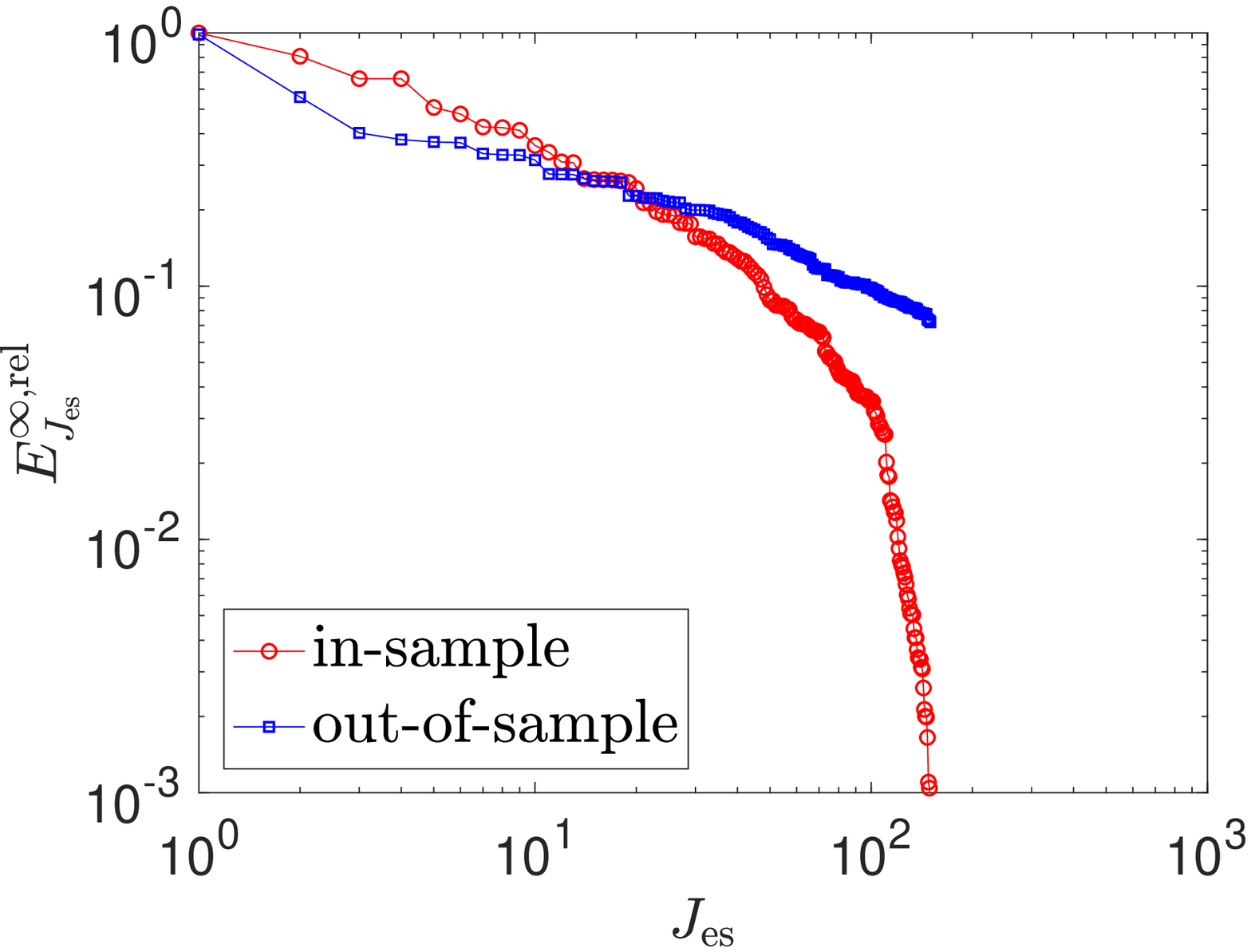}}
  
\caption{ 
behavior of  the maximum relative error $E_{J_{\rm es}}^{\infty, \rm rel}$ with respect to $J_{\rm es}$.
}
 \label{fig:NS_EJ}
\end{figure}

Figure \ref{fig:NS_pred} shows the behavior of the truth and estimated error indicator 
$\Delta^{\rm u}({\rm Re})$ over the test set, for two values of the ROM dimension $N$, $N=60,80$.
We observe that both ATI+ES and $\ell^1$-EQ+ES lead to similar performance in terms of accuracy.
$\ell^1$-EQ+ES returns a quadrature rule with 
$Q_{\rm eq}=717$ points for $N=60$ and
$Q_{\rm eq}=720$ points for $N=80$: $\ell^1$-EQ+ES  thus requires the offline computation of $n_{\rm train}^{\rm es}=150$ Riesz elements and the online storage cost of 
$\mathcal{C}_{\rm on} = D J_{\rm es} Q_{\rm eq} \approx 2.2 \cdot 10^5$ floating points\footnote{
Computational cost associated with the construction of the EQ rule is here negligible compared to the other offline costs.
}. On the other hand, ATI(+ES) requires the computation of $n_{\rm train}^{\rm es}=150$  Riesz elements and the online storage of $\mathcal{C}_{\rm on}= M J_{\rm es} D = 0.3 \cdot 10^5$ floating points: memory costs of ATI+ES for this test case are significantly lower than the costs of  $\ell^1$-EQ+ES.

\begin{figure}[h!]
\centering
\subfloat[$N=60$] {\includegraphics[width=0.48\textwidth]
 {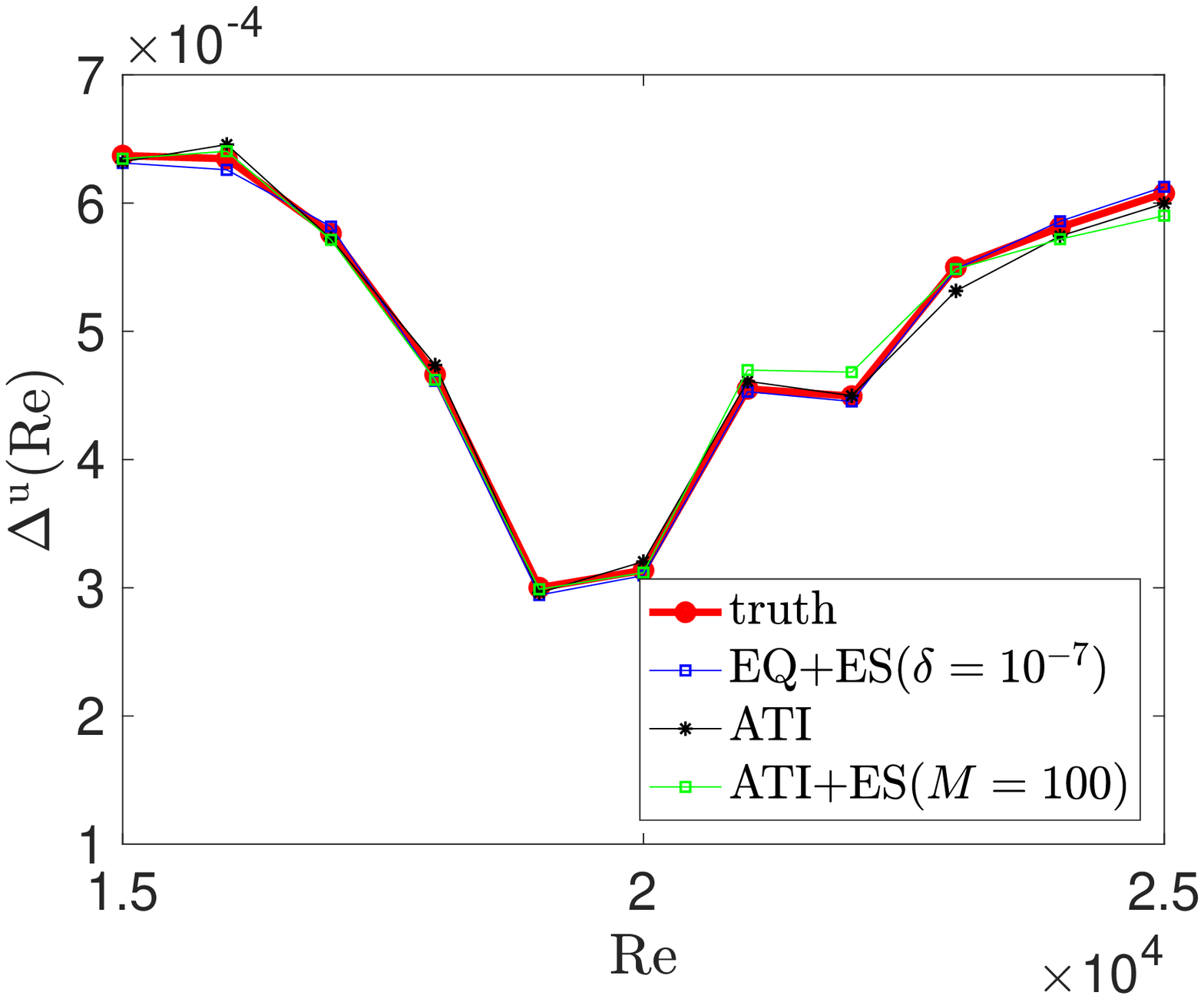}}
 ~
\subfloat[$N=80$] {\includegraphics[width=0.48\textwidth]
 {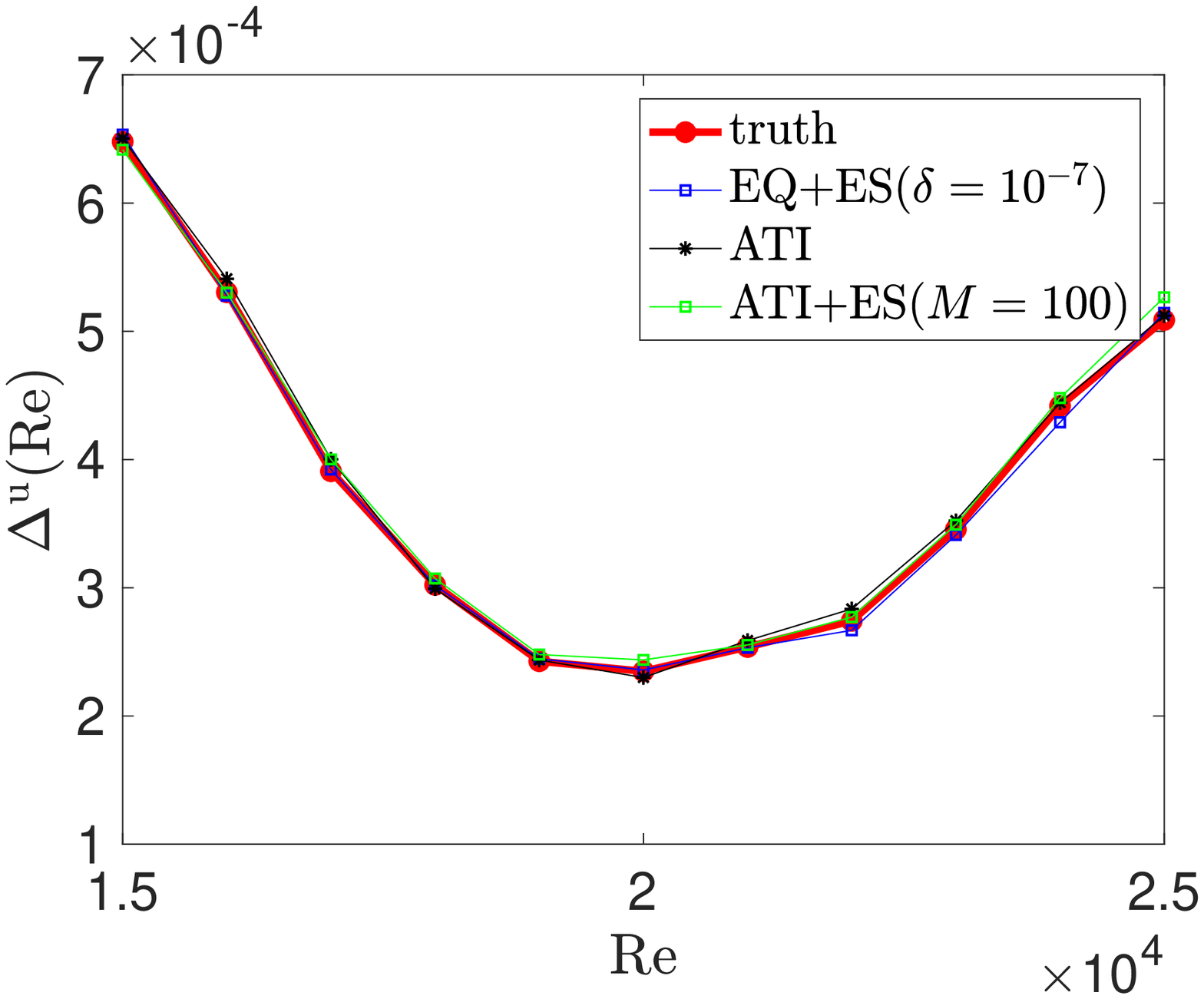}}
  
\caption{ 
behavior of   truth and estimated error indicator 
$\Delta^{\rm u}({\rm Re})$ for $n_{\rm test} = 11$ out-of-sample Reynolds numbers.
}
 \label{fig:NS_pred}
\end{figure} 

We conclude by commenting on the 
computational savings of the here-presented approaches compared to the approach employed in \cite{fick2017reduced}.
In \cite{fick2017reduced}, the authors exploit the fact that
$\widetilde{\langle R \rangle}$  \eqref{eq:tildeR} is parametrically affine to compute the truth dual norm.
For $N=80$, the procedure in \cite{fick2017reduced} requires 
the offline computation and the storage of $M_{\rm R}=6642$ Riesz elements and the online storage of $\mathcal{C}_{\rm on}=
M_{\rm R}^2  = 4.4 \cdot 10^7$ floating points.
Therefore, both $\ell^1$-EQ+ES and ATI+ES dramatically reduce offline and online memory costs.
As regards the online computational cost, computation of the residual indicator involves the computation of the  $\mathcal{O}(N^2)$ coefficients in \eqref{eq:tildeR}, which requires  $\mathcal{O}(N^2 (J-J_0) )$ flops.
Since $J-J_0 \approx 10^5 \gg N^2$, the online cost is dominated by the computation of the parameter-dependent coefficients in \eqref{eq:tildeR} for both approaches: as a result, the benefit of EQ+ES and ATI+ES in terms of online computational savings is extremely modest for the problem at hand.

\section{Conclusions}
\label{sec:conclusion}

 In this paper, we developed and analyzed an offline/online computational procedure for computing the dual norm of parameterized linear functionals. 
The key elements are an Empirical test Space (ES) built using POD, which reduces the dimensionality of the optimization problem associated with the computation of the dual norm, and an Empirical Quadrature  (EQ) procedure based on 
an $\ell^1$ relaxation or on MIO, which allows efficient calculations in an offline/online setting.

We presented theoretical  and numerical results to justify our approach. 
In particular, our results suggest that resorting to ES for proper choices of $J_{\rm es}$ might significantly reduce computational and memory costs without affecting accuracy. 
Furthermore, for the problem at hand, ATI was clearly superior for smooth integrands, while EQ strategies were able to achieve better accuracies for a non-differentiable integrand. 
Finally,  the performance of $\ell^1$-EQ+ES was inferior to that of an ATI+ES approach for 
 the estimation of the time-averaged residual indicator proposed in  \cite{fick2017reduced}.

We believe that several aspects of the proposed approach deserve further investigations.
First,  we wish to extend the analysis to quadrature rules with positive weights, and we wish to study the performance of the divide-and-conquer approach presented in section \ref{sec:empirical_quadrature}.
Second, we envision that our approach could also be employed to reduce memory and computational costs associated with minimum residual ROMs (\cite{carlberg2013gnat,yano2014space}) for nonlinear problems.
 
%
%
\appendix
 
\section{Notation}
\label{sec:notation}

\textbf{High-fidelity discretization} 
\medskip

\begin{tabular}{|l|l|}
\hline
$\mathcal{X} = {\rm span}\{ \varphi_i \}_{i=1}^{\mathcal{N}}$ &
ambient space defined over $\Omega \subset \mathbb{R}^d$ \\
\hline
$\mathcal{X}'$ & dual space \\
\hline
$R_{\mathcal{X}}: \mathcal{X}' \to \mathcal{X}$ & Riesz operator \\
\hline
$\Pi_{\mathcal{W}}: \mathcal{X} \to \mathcal{W}$ & orthogonal projector operator onto the linear space $\mathcal{W} \subset \mathcal{X}$ \\[1mm]
\hline
$\mathcal{Q}^{\rm hf}(v) = \sum_{i=1}^{\mathcal{N}_{\rm q}} \, \rho_i^{\rm hf} \, v(x_i^{\rm hf})$ &  high-fidelity quadrature rule \\[2mm]
\hline
$\mathbf{v} \in \mathbb{R}^{\mathcal{N}}$ 
&
vector of coefficients such that $v = \sum_i  v_i \, \varphi_i$
\\ \hline
$\boldsymbol{\mathcal{L}} \in \mathbb{R}^{\mathcal{N}}$
&
 vector  $\boldsymbol{\mathcal{L}}   =[
\mathcal{L}(\varphi_1),\ldots, \mathcal{L}(\varphi_{\mathcal{N}})]$
for any $\mathcal{L} \in \mathcal{X}'$
\\
\hline
$\mathbb{X}  \in \mathbb{R}^{\mathcal{N} \times \mathcal{N}}$
&
matrix such that  $\mathbb{X}_{i,j}   = (\varphi_j, \varphi_i)_{\mathcal{X}}$
\\
\hline
$C_{\rm riesz} = \mathcal{O}(\mathcal{N}^s)$ 
&
 cost to compute $R_{\mathcal{X}} \mathcal{L}$ for a given 
$\mathcal{L} \in \mathcal{X}'$
\\
\hline
\end{tabular}
\medskip

\noindent
\textbf{Parameterized functional} 
\medskip

\begin{tabular}{|l|l|}
\hline
$\mu \in \mathcal{P} \subset \mathbb{R}^P$
&
vector of parameters 
\\ [1mm]
\hline
$F(\cdot; v): \Omega   \to \mathbb{R}^D$ 
&
linear function of $v$ and its derivatives  \\ [1mm]
&
(e.g., $F(x; v) = [v(x), \nabla v(x)]$) \\[2mm]
\hline
$\Upsilon_{\mu}: \Omega   \to \mathbb{R}^D$
&
parameterized function \\ [2mm]
 \hline
$\mathcal{L}_{\mu}(v) = \mathcal{Q}^{\rm hf}( \eta(\cdot; v, \mu) )$
  & parameterized functional with 
  $ \eta(x; v, \mu) = \Upsilon_{\mu}(x) \cdot F(x; v)$
  \\ [1mm]
\hline
$L(\mu) = \|  \mathcal{L}_{\mu}  \|_{\mathcal{X}'}$
  &  dual norm
  \\ [1mm]
  \hline
$ \xi_{\mu} = R_{\mathcal{X}} \mathcal{L}_{\mu} $
  &  Riesz element of $\mathcal{L}_{\mu}$
  \\ [1mm]
\hline
$\mathcal{M}_{\mathcal{L}} = \{ 
\xi_{\mu}: \mu \in \mathcal{P} \}$
  &  dual norm
  \\ [1mm]
\hline
\end{tabular}
\medskip

\noindent
\textbf{EQ+ES  discretization} 
\medskip

\begin{tabular}{|l|l|}
\hline
$\mathcal{X}_{J_{\rm es}} = {\rm span}\{ \phi_j \}_{j=1}^{J_{\rm es}}$ &
empirical test space \\
\hline
$\mathcal{Q}^{\rm eq}(v) = \sum_{q=1}^{Q_{\rm eq}} \, \rho_q^{\rm eq} \, v(x_q^{\rm eq})$ &  empirical  quadrature rule \\[2mm]
\hline
$ L_{J_{\rm es}, Q_{\rm eq}}(\mu)
= \sup_{\phi \in \mathcal{X}_{J_{\rm es}}} \,
\frac{\mathcal{Q}^{\rm eq}(\eta(\cdot; \phi, \mu))}{\| \phi \|_{\mathcal{X}}}$
&
EQ+ES dual norm estimate \\ [2mm]
\hline
 $L_{J_{\rm es}}(\mu) = \| \mathcal{L}_{\mu} \|_{\mathcal{X}_{J_{\rm es}}}$
 &
ES dual norm estimate \\ [2mm]
\hline
$\Xi^{\rm train, es} = \{ \mu^{\ell}  \}_{\ell=1}^{n_{\rm train}^{\rm es}} \subset \mathcal{P}$
&
parameter training set for $\mathcal{X}_{J_{\rm es}} $ generation
\\  [2mm]
\hline
$\Xi^{\rm train, eq} = \{ \mu^{\ell}  \}_{\ell=1}^{n_{\rm train}^{\rm eq}} \subset \mathcal{P}$
&
parameter training set associated with \eqref{eq:accuracy_constraints}
\\  [2mm]
\hline
$K = n_{\rm train}^{\rm eq} J_{\rm eq} + 1$
&
number of rows in $\mathbb{G}$ in \eqref{eq:benchmark_optimization_no_positivity}
\\
\hline
\end{tabular}
\medskip

\noindent
\textbf{ATI discretization} 
\medskip

\begin{tabular}{|l|l|}
 \hline 
$\Upsilon_{M,\mu}(x)
= \sum_{m=1}^M \, 
\left( \boldsymbol{\Theta}_M(\mu) \right)_m \zeta_m(x)
$
 &
 $M$-term affine approximation of $\Upsilon_{\mu}$ \\[1mm]
 \hline 
$
\mathcal{L}_{M,\mu}(v) =
\mathcal{Q}^{\rm hf}
( \eta_M(\cdot; v, \mu)   )$
 &
 $M$-term affine approximation of  $\mathcal{L}_{\mu}$ \\[1mm]
 &
 $\eta_M(x; v, \mu)  := \Upsilon_{M,\mu}(x) \cdot F(x; v)$
  \\[1mm]
 \hline 
 $L_M(\mu) = \| \mathcal{L}_{M,\mu}  \|_{\mathcal{X}'}$
 &
 ATI dual norm estimate
 \\[1mm]
 \hline
  $L_{M,J_{\rm es}}(\mu) = \| \mathcal{L}_{M,\mu}  \|_{\mathcal{X}_{J_{\rm es}}'}$
 &
 ATI+ES  dual norm estimate
 \\[1mm]
 \hline
 \end{tabular}

\section{Empirical Interpolation Method}
\label{sec:EIM}

\subsection{Review of the interpolation procedure for scalar fields}

We review the Empirical Interpolation Method  (EIM),  and we discuss  its application to empirical quadrature and its extension to the approximation of vector-valued fields. Given 
the Hilbert space $\mathcal{Y}$ defined over $\Omega$, 
the $M$-dimensional linear space $\mathcal{Z}_M = {\rm span} \{ \psi_m \}_{m=1}^M \subset \mathcal{Y}$ and the points $\{  x_m^{\rm i} \}_{m=1}^M \subset \overline{\Omega}$, we define the interpolation operator  $\mathcal{I}_M: \mathcal{Y} \to \mathcal{Z}_M$ such that $\mathcal{I}_M(v)(x_m^{\rm i}) = v(x_m^{\rm i})$ for $m=1,\ldots,M$ for all $v \in \mathcal{Y}$. Given the manifold $\mathcal{F} \subset \mathcal{Y}$ and an integer $M>0$, the objective of EIM is to determine an approximation space  $\mathcal{Z}_M$ and $M$
points $\{  x_m^{\rm i} \}_{m=1}^M$ such that $\mathcal{I}_M(f)$ accurately approximates $f$ for all $f \in \mathcal{F}$.

Algorithm \ref{EIM} summarizes the EIM procedure as implemented in our code.
The algorithm takes as input snapshots of the manifold
$\{ f^k  \}_{k=1}^{n_{\rm train}} \subset \mathcal{F}$ and returns the functions $\{ \psi_m  \}_{m=1}^M$,  the interpolation points $\{  x_m^{\rm i} \}_{m=1}^M$ and the matrix $\mathbb{B} \in \mathbb{R}^{M \times M}$ such that
$\mathbb{B}_{m,m'} = \psi_m(x_{m'}^{\rm i})$.
It is possible to show that the matrix $\mathbb{B}$ is lower-triangular: for this reason, online computations  can be performed in  $\mathcal{O}(M^2)$ flops.
Note that in the original EIM paper the authors resort to a strong Greedy procedure to generate $\mathcal{Z}_M$, while here (as in several other works including \cite{chaturantabut2010nonlinear}) we resort to POD.
A thorough comparison between the two compression strategies is beyond the scope of the present work.

 \begin{algorithm}[H]                      
\caption{
Empirical Interpolation Method.}     
\label{EIM}                           


 \small
\begin{flushleft}
\begin{tabular}{| l | l |  }
\hline
\textbf{Inputs:} & $\{f^k \}_{k=1}^{n_{\rm train}}$, $M$ \\[2mm]
\hline
\textbf{Outputs:}
&
$\{ \psi_m \}_{m=1}^M, \mathbb{B} \in \mathbb{R}^{M  \times M}, \{ x_m^{\rm i}  \}_{m=1}^M$
\\[2mm]
  \hline
\end{tabular}
\end{flushleft}  
 \normalsize 

\begin{algorithmic}[1]
 \State
 Build the POD space $\zeta_1,\ldots,\zeta_M$ based on the snapshot set $\{f^k \}_{k=1}^{n_{\rm train}}$.
 \smallskip
  
 \State
 $x_1^{\rm i} := {\rm arg} \max_{x \in \overline{\Omega}} \, |\psi_1(x)|$,
 $\psi_1 :=   \frac{1}{\zeta_1(x_1^{\rm i})} \, \zeta_1$, $\left(\mathbb{B} \right)_{1,1}= 1$
  \smallskip
 
\For{$m=2,\ldots,M$}

\State
$r_m = \zeta_m - \mathcal{I}_{m-1} \zeta_m$
  \smallskip

\State
$x_m^{\rm i} := {\rm arg} \max_{x \in \overline{\Omega}} \, |r_m(x)|$,
$\psi_m = \frac{1}{r_m(x_m^{\rm i}) } \, r_m$,
$\left(\mathbb{B} \right)_{m,m'}= \psi_m(x_{m'}^{\rm i})$.
 
 \EndFor
   \smallskip

\end{algorithmic}

\end{algorithm}

\subsection{Application to empirical quadrature}
As shown in  \cite{antil2013two}, EIM naturally induces a specialized quadrature rule for elements of $\mathcal{F}$. To show this, we  consider the approximation:
$$
\mathcal{Q}^{\rm hf} \left( f \right)
\approx
\mathcal{Q}^{\rm hf} \left( \mathcal{I}_M(f) \right)
=
\sum_{m, m'=1}^M
\mathcal{Q}^{\rm hf} \left(  \psi_m
 \right)
 \mathbb{B}_{m,m'}^{-1} \, f(x_{m'}^{\rm i})
= 
\sum_{m'=1}^M
\rho_{m'}^{\rm eq} \, f(x_{m'}^{\rm i})
$$
where $\rho_{m'}^{\rm eq}  = \sum_{m=1}^M  \mathbb{B}_{m,m'}^{-1}    \mathcal{Q}^{\rm hf} \left(  \psi_m  \right)$. 
Note that in the first equality we used
$
 \mathcal{I}_M(f)(x) =
 \sum_{m, m'=1}^M
 \,   \mathbb{B}_{m,m'}^{-1}   \,
  \, f(x_{m'}^{\rm i})       \psi_m(x)
$.

\subsection{Extension to vector-valued fields}
The EIM procedure can be extended to vector-valued fields.
We present below the non-interpolatory extension of EIM employed in section \ref{sec:residual_navier_stokes}.
We refer to \cite{tonn2011reduced,negri2015efficient} for two alternatives applicable to vector-valued fields.
Given the space $\mathcal{Z}_M = {\rm span} \{ \zeta_m \}_{m=1}^M \subset \mathcal{Y}$ and 
the points $\{  x_m^{\rm i} \}_{m=1}^M \subset \overline{\Omega}$, 
we define the least-squares approximation operator $\mathcal{I}_M: \mathcal{Y} \to \mathcal{Z}_M$ such that for all $v \in \mathcal{Y}$
$$
\mathcal{I}_M(v) := {\rm arg} \min_{\zeta \in \mathcal{Z}_M} \, \sum_{m=1}^M \, \| v(x_m^{\rm i}) - \zeta(x_m^{\rm i})   \|_2^2.
$$
It is possible to show that $\mathcal{I}_M$ is well-defined if and only if the matrix $\mathbb{B} \in \mathbb{R}^{MD \times M}$,
\begin{subequations}
\label{eq:EIM_vec}
\begin{equation}
\mathbb{B} = 
\left[
\begin{array}{ccc}
\zeta_1(x_1^{\rm i}), & \ldots, & \zeta_M(x_1^{\rm i}) \\
& \vdots & \\
\zeta_1(x_M^{\rm i}), & \ldots, & \zeta_M(x_M^{\rm i}) \\
\end{array}
\right]
\end{equation}
is full-rank. In this case, we find that   $\mathcal{I}_M$ can be efficiently computed as
\begin{equation}
\mathcal{I}_M(v) = \sum_{m=1}^M \,
\left(  \boldsymbol{\alpha}(v) \right)_m \, \zeta_m,
\quad
\boldsymbol{\alpha}(v) = 
\mathbb{B}^{\dagger}
\left[
\begin{array}{c}
v(x_1^{\rm i}) \\
\vdots \\
v(x_M^{\rm i}) \\
 \end{array}
\right]
\end{equation}
for any $v \in \mathcal{Y}$, 
where $\mathbb{B}^{\dagger} = (\mathbb{B}^T \mathbb{B})^{-1} \mathbb{B}^T$ denotes the Moore-Penrose pseudo-inverse of $\mathbb{B}$. 
\end{subequations} 

Algorithm \ref{EIM_vec} summarizes the procedure employed to compute $\mathcal{Z}_M$, $\{ x_m^{\rm i}  \}_{m=1}^M$ and the matrix $\mathbb{B}$.
We observe that for scalar fields the procedure reduces to the one outlined in Algorithm \ref{EIM}. 
We further observe that online computational cost scales with $\mathcal{O}(D M^2)$, provided that 
$\mathbb{B}^{\dagger}$ is computed offline.

 \begin{algorithm}[H]                      
\caption{
Empirical Interpolation Method for vector-valued fields.}     
\label{EIM_vec}                           


 \small
\begin{flushleft}
\begin{tabular}{| l | l |  }
\hline
\textbf{Inputs:} & $\{f^k \}_{k=1}^{n_{\rm train}}$, $M$ \\[2mm]
\hline
\textbf{Outputs:}
&
$\{ \zeta_m \}_{m=1}^M, \mathbb{B}^{\dagger} \in \mathbb{R}^{M \times M}, \{ x_m^{\rm i}  \}_{m=1}^M$
\\[2mm]
  \hline
\end{tabular}
\end{flushleft}  
 \normalsize 

\begin{algorithmic}[1]
 \State
 Build the POD space $\zeta_1,\ldots,\zeta_M$ based on the snapshot set $\{f^k \}_{k=1}^{n_{\rm train}}$.
 \smallskip
  
 \State
 Set 
 $x_1^{\rm i} := {\rm arg} \max_{x \in \overline{\Omega}} \, \|\zeta_1(x) \|_2$,
and  $\mathbb{B}_{M=1}$ using \eqref{eq:EIM_vec}.
  \smallskip
 
\For{$m=2,\ldots,M$}

\State
$r_m = \zeta_m - \mathcal{I}_{m-1} \zeta_m$
  \smallskip

\State
Set 
$x_m^{\rm i} := {\rm arg} \max_{x \in \overline{\Omega}} \, \|r_m(x) \|_2$, and update
 $\mathbb{B}_{M=m}$ using \eqref{eq:EIM_vec}. 
 \EndFor
    \smallskip

\State
Compute $\mathbb{B}^{\dagger} = (\mathbb{B}^T \mathbb{B})^{-1} \mathbb{B}$.
\end{algorithmic}

\end{algorithm}

\section{Proof of Proposition \ref{th:relationship_MJQ}}
\label{sec:proof_tricky}
In view of the proof of Proposition \ref{th:relationship_MJQ}, we need the following Lemma.

\begin{Lemma}
\label{th:EIM_proof}
Let $\mathcal{W}_N = {\rm span}\{ w_n \}_{n=1}^N \subset C(\bar{\Omega})$ be a $N$-dimensional space. Then, there exist $x_1^o,\ldots,x_N^o \in \bar{\Omega}$ and
$\psi_1,\ldots,\psi_N$ such that $\psi_n(x_{n'}^o) = \delta_{n,n'}$ $n,n'=1,\ldots,N$, and
\begin{equation}
\label{eq:EIM_proof}
w(x) = \sum_{n=1}^N \, w(x_n^o) \, \psi_n(x) \quad
\forall \, x \in \Omega, \quad
\forall \, w \in \mathcal{W}_N.
\end{equation}

Similarly, given the matrix $\mathbb{G} \in \mathbb{R}^{K \times \mathcal{N}_{\rm q}}$ such that ${\rm rank}(\mathbb{G}) = N$, there exist
$\mathbb{A} \in \mathbb{R}^{K \times N}$, $\mathbb{B} \in \mathbb{R}^{N \times  \mathcal{N}_{\rm q}}$ and $\mathcal{I}_N \subset \{  1,\ldots, \mathcal{N}_{\rm q} \}$ such that
\begin{equation}
\label{eq:EIM_proof_algebraic}
\mathbb{G} = \mathbb{A} \mathbb{B},
\quad
\mathbb{B}(:  , \mathcal{I}_N) = \mathbb{I}_N.
\end{equation}
\end{Lemma}

\begin{proof}
Proofs of \eqref{eq:EIM_proof} and  \eqref{eq:EIM_proof_algebraic} 
are analogous;
for this reason, we prove  below  \eqref{eq:EIM_proof}, and we omit the proof of \eqref{eq:EIM_proof_algebraic} .

We proceed by induction.
For $N=1$, if we define $x_1^o = {\rm arg} \max_{x \in \bar{\Omega}} |w_1(x)|$ and
$\psi_1(\cdot) = \frac{1}{w_1(x_1^o)} w_1(\cdot)$, we find $w(x) = w(x_1^o) \psi_1(x)$ for all $x \in \Omega$ and $w \in \mathcal{W}_{N=1}$, which is \eqref{eq:EIM_proof}.

We now assume that the thesis holds for $N-1=N_0$, and we prove that it holds also for $N=N_0+1$.
With this in mind, we consider 
$w = \sum_{n=1}^N a_n w_n$ for some $a_1,\ldots,a_{N} \in \mathbb{R}$. We observe that
$$
w(x) - a_N w_N(x)
=
\sum_{n=1}^{N-1} \, a_n w_n(x) 
\quad
\forall \, x \in \Omega. 
$$
Then, exploiting the fact that the result holds for $N-1=N_0$,  we obtain
$$
w(x) = a_N w_N(x) + \sum_{n=1}^{N-1} \, \left(w(x_n^o) - a_N w_N(x_n^o)  \right) \tilde{\psi}_n(x)
=
a_N \tilde{\psi}_N(x) + \sum_{n=1}^{N-1} \, w(x_n^o)  \tilde{\psi}_n(x),
$$
where  $\{  x_n^o \}_{n=1}^{N-1} \subset \bar{\Omega}$,
$\tilde{\psi}_n(x_{n'}^o) = \delta_{n,n'}$ for $n,n'=1,\ldots,N-1$, and 
$\tilde{\psi}_N$ is defined as 
$$
\tilde{\psi}_N(x)
= w_N(x) - \sum_{n=1}^{N-1} \, w_N(x_n^o) \, \tilde{\psi}_n(x).
$$
If we define 
$$
x_N^o \in  {\rm arg} \max_{x \in \bar{\Omega}} \, | \tilde{\psi}_N(x)|,
\quad
\psi_N(x) := \frac{1}{\psi_N(x_N^o)} \, \psi_N(x),
$$
we find that $\psi_N(x_n^o) = \delta_{N,n}$ for $n=1,\ldots,N$, and
$$
\begin{array}{rl}
w(x) = & \displaystyle{ \left( w(x_N^o)  - \sum_{n=1}^{N-1} w(x_n^o) \, \tilde{\psi}_n(x_N^o) \right) \psi_N(x)
\,+\, \sum_{n=1}^{N-1} w(x_n^o) \, \tilde{\psi}_n(x) }
\\[3mm]
= &
\displaystyle{
w(x_N^o) \, \psi_N(x) \, + \,
\sum_{n=1}^{N-1} w(x_n^o) \left(
\tilde{\psi}_n(x) - \tilde{\psi}_n(x_N^o)   \psi_N(x)
  \right)
}
\\
\end{array}
$$
Thesis follows by defining $\psi_n(x) :=  \tilde{\psi}_n(x) - \tilde{\psi}_n(x_N^o)   \psi_N(x)$
and
observing that 
$\psi_n(x_{n'}^o) = \delta_{n,n'}$ for $n,n'=1,\ldots,N$. 
\qed
\end{proof}

\begin{proof}
(\emph{Proposition \ref{th:relationship_MJQ}})
We define $\eta_M(x; \phi, \mu) :=  \Upsilon_{M,\mu}(x) \cdot F(x; \phi)$ and the matrix $\mathbb{G}_M \in \mathbb{R}^{K \times  \mathcal{N}_{\rm q}}$ such that
$$
\mathbb{G}_M =
\left[
\begin{array}{ccc}
\eta_M (x_1^{\rm hf}; \phi_1, \mu^1 ), & \ldots &  \eta_M (x_{\mathcal{N}_{\rm q}}^{\rm hf}; \phi_1, \mu^1 ) \\ 
& \vdots & \\
\eta_M (x_1^{\rm hf}; \phi_{J_{\rm es}}, \mu^{n_{\rm train}^{\rm eq}} ), & \ldots &  \eta_M (x_{\mathcal{N}_{\rm q}}^{\rm hf}; \phi_{J_{\rm es}}, \mu^{n_{\rm train}^{\rm eq}} ) \\[2mm]
1 & \ldots & 1 \\
\end{array}
\right].
$$
By construction, for any choice of $\Xi^{\rm train,eq} = \{  \mu^{\ell} \}_{\ell=1}^{n_{\rm train}^{\rm eq}}$, we have
that 
$$
\{  \eta_M(x; \phi_j, \mu^{\ell})\}_{j,\ell} \cup \{  1 \}
\subset 
\mathcal{W}_N:=
{\rm span} \left\{
1, \,
\zeta_1 \cdot F(\cdot; \phi_1),\ldots,
\zeta_M \cdot F(\cdot; \phi_{J_{\rm es}})
\right\},
$$
with ${\rm dim}(\mathcal{W}_N) = N \leq M J_{\rm es} + 1$.
Recalling Lemma \ref{th:EIM_proof}, there exist 
$ \mathbb{A}_{N,M}^{\rm train} \in \mathbb{R}^{K \times  N}$,
$ \mathbb{B}_{N,M} \in \mathbb{R}^{N \times  \mathcal{N}_{\rm q}}$
and $\mathcal{I}_N \subset \{1,\ldots, \mathcal{N}_{\rm q}\}$
such that
$\mathbb{G}_M =  \mathbb{A}_{N,M}^{\rm train} \, \mathbb{B}_{N,M} $
and $\mathbb{B}_{N,M}(:, \mathcal{I}_N) = \mathbb{I}_N$.

We now introduce $\hat{\boldsymbol{\rho}} \in \mathbb{R}^{\mathcal{N}_{\rm q}}$ such that
$\hat{\boldsymbol{\rho}}_i = 0$ if $i \notin \mathcal{I}_N$ and
$\hat{\boldsymbol{\rho}}(\mathcal{I}_N) =: \hat{\boldsymbol{\rho}}_N = \mathbb{B}_{N,M} \boldsymbol{\rho}^{\rm hf}$. Then, we find
$$
\|  {\mathbb{G}} \hat{\boldsymbol{\rho}} - \mathbf{y}^{\rm hf}  \|_{\infty}
 =
 \|  {\mathbb{G}} ( \hat{\boldsymbol{\rho}} - \boldsymbol{\rho}^{\rm hf} )  \|_{\infty}
 \leq
\underbrace{ 
 \|  {\mathbb{G}}_M ( \hat{\boldsymbol{\rho}} - \boldsymbol{\rho}^{\rm hf} )  \|_{\infty}
}_{=\rm (I)} 
 +
 \underbrace{ 
 \| ( {\mathbb{G}} -  {\mathbb{G}}_M) \hat{\boldsymbol{\rho}} \|_{\infty}
}_{=\rm (II)} 
 +
 \underbrace{ 
 \| ( {\mathbb{G}} -  {\mathbb{G}}_M) {\boldsymbol{\rho}}^{\rm hf}  \|_{\infty}
}_{=\rm (III)} 
$$
By construction, (I)$=0$, while exploiting \eqref{eq:calLM_MJQ_a} we find
(III)$\leq \delta_{\rm ati}$. Finally, recalling the definition of $\epsilon_{\rm ati}$ in \eqref{eq:calLM_MJQ_b},  we obtain
$$
{\rm (II)}
\leq
\epsilon_{\rm ati} \|  \hat{\boldsymbol{\rho}}  \|_1
= \epsilon_{\rm ati}  \, \big\| \mathbb{B}_{N,M}  \, \boldsymbol{\rho}^{\rm hf} \big\|_1 
$$
If we set $C_{M,J_{\rm es}}=\big\| \mathbb{B}_{N,M}  \, \boldsymbol{\rho}^{\rm hf} \big\|_1$, 
we obtain that $\hat{\boldsymbol{\rho}}$ is admissible and has $N\leq M J_{\rm es} + 1$ non-zero entries. Thesis follows.
\qed
\end{proof}

\bibliographystyle{plain}
\footnotesize
\bibliography{all_refs}

\end{document}